\documentclass{article}
\usepackage[preprint]{neurips_2026}

\usepackage[utf8]{inputenc}
\usepackage[T1]{fontenc}
\usepackage{url}
\usepackage{booktabs}
\usepackage{nicefrac}
\usepackage{microtype}
\usepackage{blkarray}

\usepackage{amsmath, amsthm, amsfonts, amssymb, mathtools}
\usepackage{graphicx}
\usepackage{enumerate}
\usepackage{enumitem}
\usepackage{algorithm}
\usepackage{algorithmic}
\usepackage{hyperref}
\hypersetup{colorlinks,linkcolor={red},citecolor={olive},urlcolor={red}}
\usepackage{cleveref}
\usepackage{caption}
\usepackage{subcaption}
\usepackage{xcolor}
\usepackage{wrapfig}
\usepackage{pgfplots}
\pgfplotsset{compat=1.18}

\theoremstyle{plain}
\newtheorem{theorem}{Theorem}
\newtheorem{proposition}[theorem]{Proposition}
\newtheorem{lemma}[theorem]{Lemma}
\newtheorem{corollary}{Corollary}[theorem]
\theoremstyle{definition}

\newtheorem{example}[theorem]{Example}
\newtheorem{remark}{Remark}
\newtheorem{assumption}{Assumption}

\DeclareMathOperator{\Span}{span}
\DeclareMathOperator*{\argmin}{arg\,min}
\newcommand{\norm}[1]{\| #1 \|}

\def\R{\mathbb{R}}
\def\C{\mathbb{C}}

\def\tp{\intercal}
\def\X{\mathbf{X}}
\def\x{\mathbf{x}}
\def\Y{\mathbf{Y}}
\def\y{\mathbf{y}}
\def\by{\times}
\def\Koop{\mathcal{K}}
\def\Fc{\mathcal{F}}
\def\Dc{\mathcal{D}}
\def\Mc{\mathcal{M}}
\def\Pc{\mathcal{P}}
\newcommand{\anonrepo}{https://anonymous.4open.science/r/Koopman-PageRank-C8D3/}
\newcommand{\arXivtag}[1][]{#1}

\DeclareCaptionFont{red}{\color{red!50!black}}
\DeclareCaptionFont{green}{\color{green!40!black}}
\title{Finding Koopman Invariant Subspaces via Personalized PageRank}

\author{
	Hyukpyo Hong \\
	Department of Mathematics \\
	University of Wisconsin--Madison \\
	Madison, WI 53706, USA \\
    \texttt{hhong78@wisc.edu}
    \And
    Qin Li \\
	Department of Mathematics \\
	University of Wisconsin--Madison \\
	Madison, WI 53706, USA \\
    \texttt{qinli@math.wisc.edu} \\
	\And
	Matthew J.~Colbrook \\
	DAMTP, University of Cambridge \\
	Cambridge, CB3 0WA, UK \\
    \texttt{m.colbrook@damtp.cam.ac.uk} \\
	\And
	Hanbaek Lyu\thanks{Corresponding author} \\
	Department of Mathematics \\
	University of Wisconsin--Madison \\
	Madison, WI 53706, USA \\
    \texttt{hlyu@math.wisc.edu}
}

\begin{document}
	\maketitle
    \vspace{-5mm}
	\begin{abstract}
		Selecting a finite dictionary of observables whose span is Koopman-invariant is a central challenge in data-driven Koopman operator approximation. We address this problem by exploiting zero-block structure in Extended Dynamic Mode Decomposition (EDMD) matrices. We show that any sub-dictionary whose span is Koopman-invariant induces an exact zero block in the EDMD matrix, even for finite data. We then show that such blocks can be detected by applying PageRank to a row-normalized EDMD matrix constructed from a large initial dictionary. The theory extends to approximately invariant subspaces and yields stronger guarantees for personalized PageRank (PPR) when the seed observables lie inside the target block and reach all observables in that block. Combining EDMD concentration bounds with PageRank perturbation theory gives end-to-end detection guarantees with $O(\arXivtag[1/\sqrt{M}])$ finite-sample scaling and explicit constants. More generally, without assuming an invariant subspace exists, high PPR mass on a sub-dictionary controls discounted multi-step leakage from the seed observables. Numerical experiments on the Duffing oscillator, Van der Pol oscillator, Lorenz system, and a three-well Ramachandran potential suggest that the method identifies compact, interpretable dictionaries with accurate predictions.
	\end{abstract}
	\section{Introduction}\label{sec:intro}
	Predicting long-term behavior of a nonlinear dynamical system $\x_{k+1}=\mathbf{F}(\x_k)$ is fundamentally difficult: trajectories can exhibit sensitive dependence on initial conditions, complex bifurcation structures, and chaotic attractors. 
    The Koopman operator~\cite{koopman1931hamiltonian, Koopman-Neumann1932} offers an alternative perspective by shifting attention from the state $\x \in \Mc \subset \R^d$ to \emph{observables}, i.e., scalar-valued functions $f:\Mc \to \C$. The operator $\Koop$ acts by $\Koop f = f \circ \mathbf{F}$, transforming the nonlinear state-space dynamics into a \emph{linear} evolution in a function space. The price paid for linearity is infinite-dimensionality, but finite-dimensional projections of $\Koop$ still capture physically meaningful dynamical features.
    \arXivtag[Its growing adoption sparked by~\cite{mezic2004PhysicD, mezic2005NonDyn} is reflected in the coinage ``Koopmanism''~\cite{Budisic2012Koopmanism} and thousands of articles over the last decade (see~\cite{Mezic2013AnnRevFluid, Schmid2022DMDreview, Brunton2022SIAMReviews, Colbrook2024-MultiverseDMD, Colbrook2026} for comprehensive reviews). Some popular applications span fluid dynamics~\cite{Mezic2013AnnRevFluid}, system identification~\cite{brunton2016SINDy}, neuroscience~\cite{brunton2016neuro}, and molecular dynamics~\cite{wu2017variational}.]
	
Extended Dynamic Mode Decomposition (EDMD)~\cite{williams2015data} approximates $\Koop$ by projecting onto a finite \emph{dictionary} of observables $\Dc_{\tilde{N}} = \{\psi_1,\ldots,\psi_{\tilde{N}}\}$, yielding an $\tilde{N} \times \tilde{N}$ matrix $K_{\tilde{N},M}$ estimated from $M$ data pairs  $(\mathbf{x}^{(i)},\mathbf{y}^{(i)})_{i=1,\ldots,M}$ via least squares where $\mathbf{y}^{(i)} = \mathbf{F}(\mathbf{x}^{(i)})$. While several convergence results are now well understood~\cite{korda2018convergence, Kostic-NeuRips2023, nuske2023finite, colbrook2024rigorous}, these guarantees depend on the choice of dictionary. \arXivtag[In addition, the choice of dictionary strongly affects performance of the approximation]~\cite{Colbrook2024-MultiverseDMD, otto2021koopman}. This raises a central practical question: \emph{which observables should the dictionary contain?}
	
	\paragraph{The dictionary learning problem.}
	A natural strategy is to start with a large initial dictionary, for example, all polynomials up to a certain degree, or a dictionary of radial basis functions (RBFs), and hope that they capture relevant dynamics. However, using a large dictionary introduces several difficulties. First, when $\tilde{N}$ is large relative to $M$, the EDMD matrix can be poorly conditioned, leading to unstable predictions. Second, a large dictionary obscures interpretability: the practitioner cannot tell which observables are relevant and which are harmful. Third, and most fundamentally, EDMD accuracy is not monotone in dictionary size due to spectral pollution~\cite{Colbrook2024-MultiverseDMD, colbrook2024rigorous}.

	These difficulties motivate the question: \emph{given a large initial dictionary $\Dc_{\tilde{N}}$, how can we select a small sub-dictionary $\Dc_N \subset \Dc_{\tilde{N}}$ that captures essential dynamics?} The ideal sub-dictionary would span a \emph{Koopman invariant subspace} (i.e., a subspace closed under $\Koop$) so that the EDMD matrix on this subspace exactly represents the dynamics without leakage to the complementary observables.
	
	\paragraph{Our approach.}
	We propose an algorithm that selects a sub-dictionary by exploiting a \emph{zero block structure} of an EDMD matrix. This zero block can be efficiently detected by applying the PageRank (PR) algorithm to the row-normalized EDMD matrix, which reinterprets the matrix as a Markov chain on observables and identifies those forming a ``closed community.'' We show that \emph{personalized PR} (PPR), seeded at observables whose prediction matters most, achieves stronger detection guarantees by eliminating the bias-variance tradeoff inherent in standard PR's uniform teleportation. Because the initial dictionary consists of observables with known functional forms (e.g., polynomials, RBFs), the selected sub-dictionary yields an explicit, interpretable lifted dynamical system.
	
	\paragraph{Contributions.}
	\begin{enumerate}[leftmargin=0.5cm]
		\item \textbf{Zero block structure and invariant subspaces.} We show that any sub-dictionary spanning a Koopman invariant subspace produces an exact zero block in the EDMD matrix, even with finitely many data (Theorem~\ref{thm:blockstr}). When a spanned subspace is not exactly invariant, the norm of the off-diagonal block serves as a computable surrogate for prediction error (Propositions~\ref{prop:error-block}, \ref{prop:error-bound}).
		\item \textbf{PR/PPR-based detection.} To detect zero blocks, we propose Algorithm~\ref{algor:PR-EDMD} that reinterprets the row-normalized EDMD matrix as a Markov chain on observables and ranks them via PR or PPR. For an \emph{exact} zero block, PPR succeeds at \emph{any} damping factor $\alpha \in (0,1)$, whereas standard PR requires a mixing condition on the internal block (Theorem~\ref{thm:detection}). 
		For an \emph{approximate} zero block, PPR still succeeds with a range of a damping factor broader than standard PR.
		\item \textbf{Finite-sample guarantees.} Combining the perturbation theory with the EDMD matrix concentration bounds yields end-to-end detection guarantees at rate $O(1/\sqrt{M})$ with explicit constants (Theorem~\ref{thm:end-to-end}). The resulting sample-complexity bounds for both PR and PPR show that PPR requires fewer samples than PR (Corollary~\ref{cor:sample-complexity}).
		\item \textbf{Unconditional leakage bounds.} 
		Even without assuming an invariant subspace exists, we show that multi-step prediction error of any selected sub-dictionary is bounded by how much PPR mass falls outside that sub-dictionary (Proposition~\ref{prop:unconditional}). This gives a principled justification for PPR-based selection purely in terms of prediction accuracy, without any block-structure assumption.  
	\end{enumerate}
	
	\subsection{Related work}\label{subsec:related}
	
	\paragraph{Koopman operator theory and dictionary learning.}
	Dictionary choices range from classical bases (polynomials or RBFs~\cite{williams2015data}), kernelized EDMD~\cite{williams2015data, kawahara2016dynamic}, time-delay embeddings~\cite{arbabi2017ergodic}, to neural-network-learned dictionaries~\cite{alford2022deep, takeishi2017learning, li2017extended, wehmeyer2018time, meng2024Koopman, yeung2019learning}. These approaches focus on constructing expressive dictionaries but do not address sub-dictionary selection from a given large dictionary. Conradie et al.~\cite{conradie2026trustworthy} used a principal angle decomposition for compressing an initial dictionary, selecting modes whose Koopman image remains approximately within the subspace. They also derived error bounds for selected modes, provided a selection criterion, and used multistep bounds to train a dictionary. Haseli and Cort\'{e}s~\cite{haseli2021learning, haseli2023generalizing} proposed a subspace decomposition for iterative dictionary pruning, but provides no sample complexity guarantees. In contrast, our method selects an explicit subset of the original dictionary rather than a compressed linear combination of modes, and it provides perturbation and finite-sample guarantees for the resulting selection rule. Our PPR-based sub-dictionary selection criterion has a principled justification through Markov chain theory.
	
	A complementary line of work approximates the Koopman (or transfer) operator directly by a \emph{stochastic matrix}. Classical Ulam-type discretizations of the Perron--Frobenius operator partition the state space and estimate transition probabilities between cells, yielding a row-stochastic matrix whose stationary and spectral properties approximate those of the transfer operator~\cite{ulam2004problems, LI1976177, Klus2016JCompDyn}. Transition matrices in our work can be viewed as an EDMD-level analogue of this construction: instead of discretizing the state space, we row-normalize (the transpose of) the learned \arXivtag[EDMD] matrix to obtain a Markov chain \emph{on observables} whose PPR structure encodes sub-dictionary selection.
	
	\paragraph{PageRank, block structure detection, and spectral methods.}
	The PR algorithm~\cite{PageRank1998} computes stationary distributions of regularized Markov chains~\cite{langville2004deeper}, with well-studied perturbation theory~\cite{ipsen2006pagerank, meyer1989stochastic}. Our perturbation analysis builds on the theory of nearly completely decomposable\arXivtag Markov chains~\cite{simon1961aggregation, courtois1977decomposability}. PPR, introduced by Haveliwala~\cite{haveliwala2003topic}, was connected to local graph partitioning by Andersen, Chung, and Lang~\cite{andersen2006local}. Block structure detection relates to spectral clustering~\cite{von2007tutorial}, community detection~\cite{fortunato2010community}, and Perron cluster analysis (PCCA/PCCA+)~\cite{deuflhard2005robust, roblitz2013fuzzy}; Lai et al.~\cite{lai2010finding} studied PR for block detection in stochastic matrices. Our approach differs in starting from EDMD matrices (requiring row normalization and its perturbation analysis) and targeting Koopman invariant subspaces rather than communities or metastable states. Classical invariant subspace algorithms~\cite{golub2013matrix, stewart2001matrix} do not exploit the specific block structure that Theorem~\ref{thm:blockstr} guarantees.

	\section{Algorithm}\label{sec:algorithm}
	Given a large initial dictionary $\Dc_{\tilde{N}} = \{\psi_1,\ldots,\psi_{\tilde{N}}\}$, our goal is to select a sub-dictionary $\Dc_N \subset \Dc_{\tilde{N}}$ of $N$ observables that captures\arXivtag essential dynamics.
	We view this sub-dictionary selection problem as a \textit{ranking problem}. Observables receive high PPR scores when they are strongly coupled, through the projected Koopman dynamics, to the seed observables of interest. Our algorithm is essentially a \arXivtag[PPR] algorithm run on a directed network constructed from the EDMD matrix. Before we state our algorithm, we briefly review EDMD and PPR.
	\paragraph{Koopman operator and EDMD.}
	Consider a discrete-time dynamical system $\x_{k+1} = \mathbf{F}(\x_k)$ with $\mathbf{F}:\Mc \to \Mc$ and $\Mc \subset \R^d$. Let $\Fc = L^2(\Mc, \mu)$ where $\mu$ is the data distribution from which data $\mathbf{x}^{(i)}$ is sampled.
	The induced \emph{Koopman operator} $\Koop:\Fc\to\Fc$ is defined by $\Koop f = f\circ\mathbf{F}$ for $f\in \Fc$. Throughout, we assume $\Koop$ is a bounded operator (see Appendix~\ref{app:prelim} for details). 
	Given a \emph{dictionary} of linearly independent observables $\Dc_N = \{\psi_1, \ldots, \psi_N\} \subset \Fc$, let $\Fc_N = \Span\{\psi_1, \ldots, \psi_N\}$.
	
	With $M$ data pairs $(\x^{(i)}, \y^{(i)} = \mathbf{F}(\x^{(i)}))$, let
	\begin{equation}\label{eq:Psi-row-stack}
		\mathbf{\Psi}(\x) = [\psi_1(\x), \ldots, \psi_N(\x)] \in \mathbb{C}^{1\times N},
		\quad
		\mathbf{\Psi}(\X) =\begin{bmatrix} \mathbf{\Psi}(\x^{(1)})^{\tp} & \cdots & \mathbf{\Psi}(\x^{(M)})^{\tp} \end{bmatrix}^{\tp} \in \mathbb{C}^{M\times N},
	\end{equation}
	and similarly $\mathbf{\Psi}(\Y) \in \mathbb{C}^{M\times N}$. The \emph{EDMD matrix} $K_{N,M}$ is a minimizer of the least-squares problem
	\begin{equation}\label{eq:EDMD-mat-formula}
		K_{N,M} \;=\; \mathbf{\Psi}(\X)^{+}\,\mathbf{\Psi}(\Y) \;\in\; \argmin_{\tilde{K} \in \mathbb{C}^{N \times N}} \bigl\|\mathbf{\Psi}(\Y) - \mathbf{\Psi}(\X)\tilde{K}\bigr\|_F^2,
	\end{equation}
    \arXivtag[where $^{+}$ represents the Moore-Penrose pseudo inverse.] 
    Note that if $\Fc_N$ is $\Koop$-invariant, the Frobenius error vanishes and $\mathbf{\Psi}(\X)\,K_{N,M} = \mathbf{\Psi}(\Y)$, so the EDMD matrix \arXivtag[\textit{exactly}] captures the Koopman dynamics on $\Fc_N$. \arXivtag
    
    When the matrix $\mathbf{\Psi}(\X)$ has full column rank, $K_{N,M}$ admits the closed-form expression given by the normal equations (see Appendix~\ref{sec:app-edmd} for details). To ensure this, we impose the following assumption throughout.
	
	\begin{assumption}\label{assume:full-rank}
		For a set of $N$ observables $\{\psi_1, \ldots, \psi_N\}$ and a set of $M$ data points $\{\mathbf{x}^{(1)}, \ldots, \mathbf{x}^{(M)}\}$, the matrix $\mathbf{\Psi}(\X)$ in~\eqref{eq:Psi-row-stack} has full column rank.
	\end{assumption}
	This is a mild condition asserting that the linear independence of the observables carries over to that of their pointwise evaluations at the data points (see Remark~\ref{rem:fullcolrank} for its measure-theoretic form).
	
	\paragraph{Personalized PageRank.}
	Let $Q \in \R^{n \times n}$ be row stochastic, $\mathbf{s} \in \R^n$ a preference vector ($\mathbf{s} \geq 0$, $\mathbf{s}^\tp\mathbf{1}=1$), and $\alpha \in [0,1)$ a damping factor. The \emph{PPR vector} ${\pi}_{\mathbf{s}}$ is the stationary distribution of the transition kernel $G_{\mathbf{s}} = \alpha Q + (1-\alpha)\mathbf{1}\mathbf{s}^\tp$, which describes a walker who follows $Q$ with probability $\alpha$ and teleports to a node drawn from $\mathbf{s}$ with probability $1-\alpha$. From the stationary equation ${\pi}_{\mathbf{s}}^\tp = \alpha{\pi}_{\mathbf{s}}^\tp Q + (1-\alpha)\mathbf{s}^\tp$ (see Appendix~\ref{app:prelim} for the full derivation):
	\begin{equation}\label{eq:ppr-resolvent}
		{\pi}_{\mathbf{s}}^\tp = (1-\alpha)\,\mathbf{s}^\tp(I - \alpha Q)^{-1}.
	\end{equation}
	Since the resolvent $(I - \alpha Q)^{-1}$ is independent of $\mathbf{s}$, the map $\mathbf{s} \mapsto {\pi}_{\mathbf{s}}$ is \emph{linear}. For a \emph{seed set} $\mathcal{S} \subset \{1,\ldots,n\}$ with $\mathbf{s} = \frac{1}{|\mathcal{S}|}\sum_{v \in \mathcal{S}}\mathbf{e}_v$, the \emph{multi-seed PPR vector} is $	{\pi}_{\mathcal{S}} = \frac{1}{|\mathcal{S}|}\sum_{v \in \mathcal{S}} {\pi}_v$, 
	where ${\pi}_v$ is the \emph{single-seed} PPR vector with $\mathcal{S} = \{v\}$. \emph{Standard PR}~\cite{PageRank1998, langville2004deeper} is a case where $\mathcal{S} = \{1,\ldots,n\}$. Thus, the \emph{PR vector} ${\pi}^\tp$ is given by ${\pi}^\tp = \frac{1}{n}\sum_{v=1}^n {\pi}_v^\tp$.

	\paragraph{Algorithm: EDMD+PPR.}
	
	The key idea in our observable ranking algorithm is to interpret the EDMD matrix as a Markov chain on observables to discover zero block structure, indicating Koopman invariant subspaces. Define $P\coloneqq \mathcal{R}(|K_{\tilde{N},M}^\tp|)$ where $\mathcal{R}$ is the row-normalization, i.e., the $(i,j)$-entry of $\mathcal{R}([a_{ij}])$ is $a_{ij}/\sum_{k=1}^na_{ik}$. The transition probability from observable $i$ to $j$, $P[i,j]$, is proportional to $|K_{\tilde{N},M}[j,i]|$, the coefficient of $\psi_j$ in the expansion of $\Koop\psi_i$ --- i.e., the coupling strength from $\psi_i$ to $\psi_j$ under the projected dynamics. We then apply PPR with a prescribed \textit{seed set} $\mathcal{S}$ of observables to rank all $\tilde{N}$ observables \arXivtag[(Algorithm \ref{algor:PR-EDMD})]. 
	
	\begin{algorithm}
		\caption{Sub-dictionary selection via Personalized PageRank on the EDMD matrix}
		\begin{algorithmic}\label{algor:PR-EDMD}
			\STATE \textbf{Input:} Dictionary $\Dc_{\tilde{N}} = \{ \psi_1, \ldots, \psi_{\tilde{N}}\}$, target size $N$, seed set $\mathcal{S} \subset \{1,\ldots,\tilde{N}\}$ with $|\mathcal{S}|\leq N$, data $(\x^{(i)}, \y^{(i)} = \mathbf{F}(\x^{(i)}))_{i=1}^M$, and damping $\alpha$.
			\STATE \textbf{Step 1:} Compute the EDMD matrix $K_{\tilde{N},M} = \mathbf{\Psi}_{\tilde{N}}(\X)^{+}\,\mathbf{\Psi}_{\tilde{N}}(\Y)$.
			\STATE \textbf{Step 2:} Row-normalize $|K_{\tilde{N},M}^{\tp}|$ to obtain the row stochastic matrix $P$. 
			\STATE \textbf{Step 3:} Compute the multi-seed PPR vector ${\pi}_{\mathcal{S}, M} = \frac{1}{|\mathcal{S}|}\sum_{v \in \mathcal{S}}{\pi}_{v,M}$ of $P$ at damping $\alpha$; select the set of indices, $S_N$, with the top-$N$ PPR scores and the corresponding sub-dictionary $\Dc_N$. 
			\STATE \textbf{Step 4:} Recompute EDMD on the selected sub-dictionary:
			$K_{N,M} = \mathbf{\Psi}_N(\X)^{+}\,\mathbf{\Psi}_N(\Y)$.
			\STATE \textbf{Output:} $\Dc_N$ and $K_{N,M}$.
		\end{algorithmic}
	\end{algorithm}
	\arXivtag[If a row of $|K_{\tilde{N},M}^{\tp}|$ has zero sum, the row cannot be normalized; such rows (and corresponding columns) are] discarded in Step~2. Step~4 \emph{recomputes} the EDMD matrix with the selected $N$ observables, yielding the optimal predictor within $\Fc_N$, rather than extracting a submatrix of $K_{\tilde{N},M}$. 
	
	The seed set may contain observables whose  predictions matter most for specific application (e.g., coordinates for state prediction, slow mode observables), or could be taken to be the whole set of $\tilde{N}$ observables if no prior information is available. See Remark~\ref{rem:practical} for details. The PPR scores on the seed observables are usually high due to the teleportation, so they are likely to be included in the selected sub-dictionary $\Dc_N$ consisting of the top $N$ observables with the highest PPR scores. 
	
	The selected sub-dictionary $\Dc_N$ is where the Koopman coupling concentrates: the dynamics starting from the seed set $\mathcal{S}$ stays mostly within $\Fc_N$ with minimum leakage outside of it. 
	This can be interpreted in terms of multi-step prediction. 
	Let $r_{\max,M} = \|K_{\tilde{N},M}^{\tp}\|_\infty$ and assume $\alpha < r_{\max,M}$ so that $\gamma := \alpha/r_{\max,M} \in (0,1)$ (automatic when $r_{\max,M} \geq 1$). For any preference vector $\mathbf{s}$ supported on the selected indices $S_N$, we denote the \emph{$\gamma$-discounted multi-step leakage} by $\Lambda_{\mathbf{s}}^{\gamma}(S_N)$. This quantity measures the total $\gamma$-discounted weight that escapes $S_N$ under repeated application of $|K_{\tilde{N},M}^{\tp}|$. We show this satisfies
	$\Lambda_{\mathbf{s}}^{\gamma}(S_N) \leq \frac{1-\gamma}{1-\alpha}(1 - \pi_{\mathbf{s}, M}(S_N))$,
	where $\pi_{\mathbf{s},M}$ is the PPR vector found in Step~3 (see Proposition~\ref{prop:unconditional} in Appendix~\ref{sec:proofs-leakage} for the precise statement). Thus, selecting observables with high PPR scores implicitly reduces multi-step leakage from the seed set across all time horizons simultaneously. The bound passes through row normalization of $|K_{\tilde{N},M}^{\tp}|$, which discards row-magnitude information, so it may not be tight; the main results in Section~\ref{sec:main-results} develop sharper guarantees by exploiting the structure of the EDMD matrix directly.
	\vspace{-0.2cm}
	\section{Theoretical results}\label{sec:main-results}
	\vspace{-0.3cm}
	We first present error bounds for prediction accuracy (Section~\ref{subsec:error-bounds}), develop detection theory for exact and approximate invariant subspaces (Section~\ref{subsec:block-theory}), and give finite-sample extensions (Section~\ref{subsec:end-to-end}). A worked three-state example is in Appendix~\ref{app:3state}, finite-sample details in Appendix~\ref{app:finite-sample}, key lemmas in Section~\ref{sec:lemmas}, and proofs in Appendix~\ref{sec:proofs}.
	
		\vspace{-0.3cm}
	\subsection{EDMD prediction error}\label{subsec:error-bounds}
		\vspace{-0.2cm}
	We derive error bounds that reveal the central role of off-diagonal blocks in the EDMD matrix, motivating the entire framework. Let $\Dc = \{\psi_1, \psi_2, \ldots\}$ be a countable basis for the Hilbert space $\Fc = L^2(\Mc, \mu)$. 
	Writing $\mathbf{\Psi} = [\psi_1, \psi_2, \ldots]$ as an infinite row vector of observables and $K^{\text{true}}$ for the (infinite) matrix whose $j$-th column collects the expansion coefficients of $\Koop\psi_j$ in $\Dc$, the image of each basis function under the Koopman operator is
	\begin{equation}\label{eq:psi_evolve_true}
		\Koop\psi_i
		\;=\;
		\mathbf{\Psi}\, K^{\text{true}}\, \mathbf{e}_i
		\;=\;
		\begin{bmatrix}
			\mathbf{\Psi}_N & \mathbf{\Psi}_{N^\mathsf{c}}
		\end{bmatrix}
		\begin{bmatrix}
			K^{11, \text{true}} & K^{12, \text{true}} \\
			K^{21, \text{true}} & K^{22, \text{true}}
		\end{bmatrix}
		\mathbf{e}_i,
	\end{equation}
	where $\mathbf{\Psi}_N = [\psi_1,\ldots,\psi_N]$ and $\mathbf{\Psi}_{N^\mathsf{c}} = [\psi_{N+1},\psi_{N+2},\ldots]$, and the block partition is by the first $N$ rows/columns and their complements. For $i \leq N$, the selector $\mathbf{e}_i$ lies in the left block, so that $\Koop\psi_i = \mathbf{\Psi}_N K^{11,\text{true}}\mathbf{e}_i + \mathbf{\Psi}_{N^\mathsf{c}} K^{21,\text{true}}\mathbf{e}_i$; the second summand is the ``leakage'' of $\Koop\psi_i$ outside $\Fc_N$, carried by column $i$ of the bottom-left block $K^{21,\text{true}}$. 
	
	The EDMD approximation of $\Koop\psi_i$ with a finite dictionary $\Dc_N$ and finitely many data points $\{\x^{(i)}\}_{i=1}^M$ is given by $\Koop_{N,M}\psi_i \coloneqq \sum_{j=1}^N K_{N,M}[j,i]\psi_j$ where $K_{N,M}$ is the EDMD matrix in \eqref{eq:EDMD-mat-formula}. We want to reduce the $L^2$-error of this approximation where the dictionary $\Dc_N$ is a subset of a larger initial dictionary $\Dc_{\tilde{N}}$. Thus, we use the EDMD matrix with the initial dictionary $\Dc_{\tilde{N}}$ to provide an upper bound of the $L^2$-error. Specifically, we partition the EDMD matrix with $\Dc_{\tilde{N}}$ as 
	\begin{equation}\label{eq:K-block}
		K_{\tilde{N},M} = \begin{bmatrix} K_{\tilde{N},M}^{11} & K_{\tilde{N},M}^{12} \\ K_{\tilde{N},M}^{21} & K_{\tilde{N},M}^{22} \end{bmatrix},
	\end{equation}
	where the rows and columns are divided into two groups: from 1 to $N$ and from $N+1$ to $\tilde{N}$. Then we prove the following proposition.

	\begin{proposition}\label{prop:error-block}
		Let $\Dc_{\tilde N} =\{\psi_1, \dots , \psi_{\tilde N}\}$ be an initial dictionary and $\Dc_N =\{\psi_1, \dots , \psi_N\}$ be a sub-dictionary of $\Dc_{\tilde N}$. Suppose $\Dc_{\tilde{N}}$ is an orthonormal set and satisfies Assumption~\ref{assume:full-rank} with $M$ i.i.d.\ samples from $\mu$, $\{\x^{(i)}\}_{i=1}^M$.
		For each $\varepsilon > 0$, there exists $M_0 > 0$ such that for all $M \geq M_0$,
		\begin{equation}\label{eq:error-block-bound}
			\sum_{i=1}^N \norm{\Koop \psi_i - \Koop_{N,M}\,\psi_i}_{L^2}^2
			\leq
			2\left( \norm{ K_{\tilde{N},M}^{21} }_F^2 + \norm{ K^{\textnormal{\arXivtag[true]}}[(\tilde{N}+1):\infty,\, 1:\tilde{N}] }_F^2 \right) + \varepsilon.
		\end{equation}
	\end{proposition}	
	Importantly, Proposition~\ref{prop:error-block} applies to \emph{any} sub-dictionary of size $N$ formed from $\Dc_{\tilde{N}}$: relabeling the observables by any permutation of $\{1,\ldots,\tilde{N}\}$ and taking the first $N$ elements under the permutation yields a valid instance of the proposition, with the off-diagonal block $K_{\tilde{N},M}^{21}$ recomputed accordingly.
	
	On the right-hand side, the first term $\|K_{\tilde{N},M}^{21}\|_F^2$ is \emph{directly computable from data}; the second reflects truncation beyond $\tilde{N}$ and depends on the initial dictionary $\Dc_{\tilde{N}}$, not the sub-dictionary $\Dc_{N}$.
	The practical implication is therefore: \emph{to find the best sub-dictionary of size $N$, one should search for the permutation $\{\psi_1,\ldots,\psi_{\tilde{N}}\}$ that minimizes the Frobenius norm of the resulting off-diagonal block, $\norm{ K_{\tilde{N},M}^{21} }_F^2$}. The ideal case is when this block is exactly zero for some permutation. Exact zero blocks correspond precisely to Koopman invariant subspaces within the span of the initial dictionary.
	
	We now show that EDMD ``discovers'' these invariant subspaces automatically. Note that the initial dictionary $\Dc_{\tilde{N}}$ in Theorem~\ref{thm:blockstr} is not necessarily orthonormal, \arXivtag[unlike Proposition~\ref{prop:error-block}].
	\begin{theorem}\label{thm:blockstr}
		Let $\Dc_{\tilde{N}} = \{\psi_1, \ldots, \psi_{\tilde N}\}$ be an initial dictionary (not necessarily spanning an invariant subspace) and $\{\mathbf{x}^{(i)}\}_{i=1}^M$ be a set of data points, which satisfy Assumption~\ref{assume:full-rank}. 
		If the sub-dictionary $\Dc_N = \{\psi_1, \ldots, \psi_N\} \subset \Dc_{\tilde{N}}$ spans a Koopman invariant subspace, then 
		\vspace{-0.2cm}
		\begin{description}[itemsep=0.5cm, itemsep=-0.2cm]
			\item[(i)] The top-left $N \by N$ submatrix of $K_{\tilde{N},M}$ is identical to  $K_{N,M}$.
			\item[(ii)] The bottom-left $(\tilde{N} - N) \by N$ submatrix of $K_{\tilde{N},M}$ (i.e., $K_{\tilde{N},M}^{21}$) is identically zero. 
		\end{description}
	\end{theorem}	
	This theorem suggests that if there is any sub-dictionary that spans a Koopman invariant subspace such a zero block structure emerges up to permutation. Similar results when $M=\infty$ or $\Dc_N$ is chosen as the initial dictionary were noted in literature~\cite{williams2015data, Klus2016JCompDyn, korda2018convergence, mezic2021koopman}. We emphasize that the result holds for $M<\infty$ as well even with a larger initial dictionary, making ours more computationally feasible. See Appendix~\ref{app:toy-model} for illustration with a toy model. 
	
	\arXivtag[These results recast the problem of finding an invariant subspace as identifying a zero block structure in the EDMD matrix. However, exhaustive search over all $\binom{\tilde{N}}{N}$ sub-dictionaries to obtain the clear zero block submatrix at the bottom-left corner is combinatorially prohibitive. 
    Our framework (Algorithm~\ref{algor:PR-EDMD}) circumvents this obstacle by considering the Markov chain induced by the EDMD matrix, while also handling the practically relevant case of approximate invariance.]
    
	\vspace{-0.1cm}
	\subsection{Finding Koopman invariant subspaces via personalized PageRank}\label{subsec:block-theory}
	
	\vspace{-0.2cm}
	We now develop a finer theory for the case when the EDMD matrix admits (approximate) zero block structure. Let $P = \mathcal{R}(|K_{\tilde{N},M}^{\tp}|)$ be the row normalization of $|K_{\tilde{N},M}^{\tp}|$. Let $K_{\tilde{N},M}^{(0)}$ denote $K_{\tilde{N},M}$ with the bottom-left block zeroed, and define the row-stochastic matrix $P^{(0)} := \mathcal{R}(|K_{\tilde{N},M}^{(0)}|^{\tp})$. 
    Equivalently, $P^{(0)}$ is obtained from $P$ by zeroing its top-right block $P^{12}$ and \emph{renormalizing \arXivtag[the rows of the top block]}, so that $P^{11,(0)} = \mathcal{R}(P^{11})$, $P^{21,(0)} = P^{21}$, $P^{22,(0)} = P^{22}$, with the top-right block identically zero, \arXivtag[i.e., $P^{12,(0)} = O$.] 
    All matrices are partitioned conformally with \eqref{eq:K-block}. The coupling parameter is $\eta := 1 - \|P^{22}\|_\infty > 0$. 

	Algorithm~\ref{algor:PR-EDMD} ranks observables by the PR or PPR vector $\pi_{\mathbf{s}}$ of $P$ (with $\mathbf{s} = \mathbf{1}/\tilde{N}$ for PR, $\mathbf{s} = |\mathcal{S}|^{-1}\sum_{v\in\mathcal{S}}\mathbf{e}_v$ for PPR), so it correctly identifies $\Dc_N$ iff the corresponding gap on $P$ is positive:
	\begin{equation}\label{eq:delta-defs}
		\Delta_{\textnormal{PR}} := \min_{i\leq N}\pi(i) - \max_{j > N}\pi(j) > 0, \qquad
		\Delta_{\textnormal{PPR}} := \min_{i\leq N}\pi_{\mathcal{S}}(i) - \max_{j > N}\pi_{\mathcal{S}}(j) > 0.
	\end{equation}
	Theorem \ref{thm:detection} provides sufficient conditions for detection.

	\begin{theorem}[Koopman invariant subspace detection]\label{thm:detection}
		Under Assumption~\ref{assume:full-rank} and $\alpha \in (0,1)$:
		
		\smallskip
		\noindent\textbf{(i) Standard PR} ($\mathcal{S} = \{1,\ldots,\tilde{N}\}$). 
		Let $(\mathbf{q}^{\alpha})^\tp := \tfrac{1-\alpha}{N}\mathbf{1}_N^\tp(I - \alpha P^{11,(0)})^{-1}$ denote the standard PR vector of $P^{11,(0)}$ and let  $q_{\min}^\alpha := \min_{i \leq N}q_i^\alpha$. Then detection is successful (i.e., $\Delta_{\textup{PR}}>0$) provided
		\begin{equation}\label{eq:pr-gap-condition-explicit}
			\|P^{12}\|_\infty \;<\; \frac{(1-\alpha)N}{4\alpha\tilde{N}}\!\left[q_{\min}^\alpha - \frac{(\tilde{N}-N)(1-\alpha)}{N(1-\alpha+\alpha\eta)}\right].
		\end{equation}
		
		\smallskip
		\noindent\textbf{(ii) Multi-seed PPR.} Let $\mathcal{S} \subset \{1,\ldots,N\}$ satisfy the reachability condition: every $i \leq N$ is reachable from some $v \in \mathcal{S}$ in $P^{11,(0)}$. Then detection is successful (i.e., $\Delta_{\textnormal{PPR}} > 0$) provided
	\begin{equation}\label{eq:ppr-gap-condition-explicit}
		\|P^{12}\|_\infty \;<\; \frac{(1-\alpha)^2}{4\alpha|\mathcal{S}|}\min_{i\leq N}\sum_{v\in\mathcal{S}}\bigl[(I - \alpha P^{11,(0)})^{-1}\bigr]_{vi}.
	\end{equation}
	In particular, when $\Fc_N$ is $\Koop$-invariant, $P^{12} = O$ by Theorem~\ref{thm:blockstr}, so \eqref{eq:ppr-gap-condition-explicit} holds: $\pi_{\mathcal{S}} = \pi^{(0)}_{\mathcal{S}}$, $\pi_{\mathcal{S}}(j) = 0$ for all $j > N$, and the recomputed $K_{N,M}$ \arXivtag[coincides with the population-level Koopman matrix $K_N\coloneqq\lim_{M\to\infty}K_{N,M}$.]
\end{theorem}

The proof of the above result is in Appendix \ref{sec:detection_proofs}. A brief sketch of the proof is provided in Sec. \ref{sec:lemmas}. To demonstrate the implications of Theorem \ref{thm:detection}, in 	Figure~\ref{fig:detection-windows}  we plot the detection window $\alpha^*(\epsilon) := \sup\{\alpha : \text{Theorem~\ref{thm:detection}(i) or (ii) holds}\}$ as a function of leakage $\epsilon := \|P^{12}\|_\infty$ for a 
\begin{wrapfigure}[11]{r}{0.5\linewidth}
	\vspace{-12pt}
	\centering
	\includegraphics[width=1\linewidth]{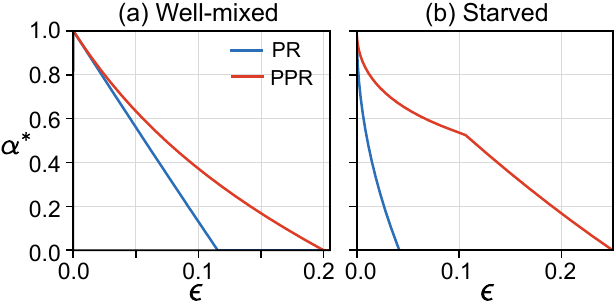}
	\vspace{-0.6cm}
	\caption{Detection windows $\alpha^\star(\epsilon)$. 
		Detection holds below each curve.\arXivtag }
	\label{fig:detection-windows}
	\vspace{-15pt}
\end{wrapfigure}
simple 3-dimensional example. 
A wider window means detection succeeds  for a  broader range of the damping factor $\alpha$. PPR (red) dominates PR (blue) in both panels (see Appendix~\ref{app:3state} for details). In Panel~(a) (well-mixed internal block), PPR tolerates leakage up to $\epsilon = 1/4 = 0.25$, whereas PR fails beyond $\epsilon = 1/8 = 0.125$. The gap is more dramatic in Panel~(b), where the internal block has a \emph{starved node} (a node receiving no inflow from within the block): PR's window collapses beyond $\epsilon = 1/24 \approx 0.042$, while PPR remains effective up to $\epsilon = 1/4 = 0.25$. The underlying reason is structural: PR requires the internal block 	$P^{11,(0)}$ to be well-mixed, a condition that fails whenever a starved node is present, whereas PPR  requires only that the seed set $\mathcal{S}$ can reach every node in the block, which is a strictly weaker condition that holds independently of $\alpha$.

	\vspace{-0.1cm}
\subsection{Finite-sample guarantees}\label{subsec:end-to-end}
	\vspace{-0.2cm}
\arXivtag[To establish finite-sample error bounds and detection guarantees, we] additionally assume the dictionary and its Koopman images are uniformly bounded:
\begin{assumption}\label{assume:bounded-dict}
	There exists a positive constant $D > 0$ such that $\sup_{\x \in \Mc}|\psi_i(\x)| \leq D$ and $\sup_{\x \in \Mc}|\Koop\psi_i(\x)| \leq D$ for all $i \in \mathbb{N}$.
\end{assumption}
Under Assumption~\ref{assume:bounded-dict}, matrix Bernstein inequalities yield $\|K_{\tilde{N},M}^\tp - K_{\tilde{N}}^\tp\|_\infty \leq \varepsilon_M = O(1/\sqrt{M})$ for $M$ sufficiently large (Proposition~\ref{prop:finite-sample-edmd} in Appendix~\ref{sec:proofs-finite}). Let $\lambda_{\min}$ denote the smallest eigenvalue (in modulus) of the Gram matrix $G_{\tilde{N}} \coloneqq [\langle \psi_i, \psi_j\rangle_{L^2(\mu)}]_{i,j=1}^{\tilde{N}}$. Throughout this subsection, let $P_{\textnormal{pop}} = \mathcal{R}(|K_{\tilde{N}}^\tp|)$ be the row-normalized \arXivtag[population-level] EDMD matrix with top-right block $P^{12}_{\textnormal{pop}}$, and let $P^{(0)}_{\textnormal{pop}} := \mathcal{R}(|K^{(0)}_{\tilde{N}}|^\tp)$ be the row-stochastic matrix obtained from $P_{\textnormal{pop}}$ by zeroing $P^{12}_{\textnormal{pop}}$ and renormalizing \arXivtag[the rows of the top block]. Let $K_{\tilde{N}}^{(0)}$ denote $K_{\tilde{N}}$ with its bottom-left block zeroed, and set $r^{(0)}_{\min} := \min_i \sum_j |(K_{\tilde{N}}^{(0)})^\tp[i,j]|$ and $r_{\max} := \|K_{\tilde{N}}^\tp\|_\infty$. Define $\varepsilon_0 := \|(K_{\tilde{N}}^{(0)} - K_{\tilde{N}})^\tp\|_\infty = \max_{i \leq N}\sum_{j > N}|K_{\tilde{N}}^\tp[i,j]|$ as the population operator-level leakage, which satisfies $\|P^{12}_{\textnormal{pop}}\|_\infty\leq \varepsilon_0/r^{(0)}_{\min}$. Let $\pi_{\textnormal{pop}}^{(0)}$ denote the PR (or PPR) vector of the row-stochastic matrix $P^{(0)}_{\textnormal{pop}}$. Combining with the perturbation theory:

\begin{theorem}[End-to-end detection]\label{thm:end-to-end}
	Let $\Dc_{\tilde{N}}$ satisfy Assumptions~\ref{assume:full-rank} and~\ref{assume:bounded-dict}, and let $\rho \in (0,1/2)$. Let $\varepsilon_M = C_{\textnormal{EDMD}}\sqrt{(2/M)\log(2\tilde{N}/\rho)}$ (see Appendix~\ref{sec:proofs-finite} for the definition of $C_{\textnormal{EDMD}}$), and let ${\pi}_M$ be the PR (or PPR) vector of $P = \mathcal{R}(|K_{\tilde{N},M}^\tp|)$. Then for $M \geq \frac{32\,\tilde{N}^2 D^4}{\lambda_{\min}^2}\log\frac{2\tilde{N}}{\rho}$, with probability $\geq 1-2\rho$, we have $	\norm{{\pi}_M - \pi_{\textnormal{pop}}^{(0)}}_1 \;\leq\; \frac{2\alpha(\varepsilon_0 + \varepsilon_M)}{(1-\alpha)\,r^{(0)}_{\min}}$. 
\end{theorem}
This result, combined with Theorem~\ref{thm:detection}, provides the sample-complexity bounds for a correct detection.

\begin{corollary}[Sample complexity for PPR detection with $\alpha=1/2$]\label{cor:sample-complexity}
	Let $\Delta^{(0)}_{\textnormal{PPR}}$ be the PPR auxiliary gap \eqref{eq:delta0-defs} of Section~\ref{sec:lemmas} computed from $P^{(0)}_{\textnormal{pop}}$.
	If $\Delta^{(0)}_{\textnormal{PPR}} > 0$ and $\|P_{\textnormal{pop}}^{12}\|_\infty < r^{(0)}_{\min}\Delta^{(0)}_{\textnormal{PPR}}/(4 r_{\max})$, detection by Algorithm~\ref{algor:PR-EDMD} (i.e.\ $\Delta_{\textnormal{PPR}} > 0$) holds with probability $\geq 1-2\rho$ for
	\begin{equation}\label{eq:ppr-sample-complexity}
		M \;\geq\; \max\!\left\{ \frac{32\,\tilde{N}^2 D^4}{\lambda_{\min}^2}\log\frac{2\tilde{N}}{\rho},\;\; \frac{32\log(2\tilde{N}/\rho)\, C_{\text{EDMD}}^2}{(r^{(0)}_{\min})^2\bigl(\Delta^{(0)}_{\textnormal{PPR}} - 4 r_{\max}\|P_{\textnormal{pop}}^{12}\|_\infty/r^{(0)}_{\min}\bigr)^2}\right\}.
	\end{equation}
\end{corollary}

The bound is of the form $M \gtrsim C_{\text{EDMD}}^2\log(\tilde N/\rho)/\bigl(\Delta^{(0)}_{\textnormal{PPR}} - c\,(r_{\max}/r^{(0)}_{\min})\|P_{\textnormal{pop}}^{12}\|_\infty\bigr)^2$: the leakage shrinks the effective gap from $\Delta^{(0)}_{\textnormal{PPR}}$ to $\Delta^{(0)}_{\textnormal{PPR}} - c\,(r_{\max}/r^{(0)}_{\min})\|P_{\textnormal{pop}}^{12}\|_\infty$ and inflates the sample requirement quadratically. The $\alpha = 1/2$ specialization eliminates any $\alpha$-dependent prefactor: because PPR leaks no mass outside \arXivtag[$\Dc_N$], $\Delta^{(0)}_{\textnormal{PPR}}$ stays positive without requiring $\alpha$ near $1$. The proof, the analogous PR sample complexity bound (Corollary~\ref{cor:pr-sample-complexity}), and the related finite-sample leakage bound (Corollary~\ref{cor:finite-sample-leakage}) are deferred to Appendix~\ref{app:finite-sample}.

	\vspace{-0.1cm}
\subsection{Sketch of proof for Theorem \ref{thm:detection}}\label{sec:lemmas}
	\vspace{-0.2cm}
Full proofs of our main results are in Appendix~\ref{sec:detection_proofs}. Here we give a brief sketch of the proof of Theorem \ref{thm:detection}. Observe that, due to the block lower-triangular structure, it may be easier to detect on the proxy matrix  $P^{(0)}$. Accordingly, define the auxiliary gaps 
\begin{equation}\label{eq:delta0-defs}
	\Delta^{(0)}_{\textnormal{PR}} := \min_{i\leq N}\pi^{(0)}(i) - \max_{j > N}\pi^{(0)}(j), \qquad
	\Delta^{(0)}_{\textnormal{PPR}} := \min_{i\leq N}\pi^{(0)}_{\mathcal{S}}(i) - \max_{j > N}\pi^{(0)}_{\mathcal{S}}(j),
\end{equation}
where $\pi^{(0)}, \pi^{(0)}_{\mathcal{S}}$ are the PR/PPR vectors of $P^{(0)}$. 
If the off-diagonal block $P^{12}$ is small, then $P$ and $P^{(0)}$ are close so they should have similar PR/PPR vectors. More precisely, we show (see Lemma~\ref{lem:pr-ppr-perturbation} in  Appendix~\ref{sec:detection_proofs})  the PR/PPR vectors of $P$ and $P^{(0)}$ differ in $\ell^1$ by at most $\frac{2\alpha}{1-\alpha}\|P^{12}\|_\infty$. 

Therefore, detection on $P$ succeeds if detection on $P^{(0)}$ holds by a large enough margin compared to the off-diagonal block $P^{12}$: For $\textup{method}\in \{ \textup{PR}, \textup{PPR} \}$, 
\begin{align}\label{eq:detection_reduction}
	\Delta_{\textup{method}} > 0 \quad  \Longleftarrow  \quad \|P^{12}\|_\infty \;<\; \frac{1-\alpha}{4\alpha}\,\Delta^{(0)}_{\textnormal{method}}. 
\end{align}

Now, due to the block-triangular structure of $P^{(0)}$, we can establish the following closed-form expressions for the auxiliary gaps $\Delta^{(0)}_{\textnormal{PR}}$ and $\Delta^{(0)}_{\textnormal{PPR}}$ in terms of the diagonal blocks of the resolvent $(I - \alpha P^{(0)})^{-1}$, namely $R_{11} := (I - \alpha P^{11,(0)})^{-1}$ and $R_{22} := (I - \alpha P^{22,(0)})^{-1}$: 

\begin{lemma}[Closed-form auxiliary gaps]\label{lem:gaps-closed-form}
	For $\alpha \in (0,1)$, both $R_{11}$ and $R_{22}$ are entrywise nonnegative with finite entries, and
	\begin{equation}\label{eq:delta-closed}
		\Delta^{(0)}_{\textnormal{PR}} \;=\; \frac{1-\alpha}{\tilde{N}}\left[\min_{i \leq N}\sum_{v \leq N} s_v^{\textnormal{eff}}\,[R_{11}]_{vi} \;-\; \max_{j > N}\sum_{w > N}[R_{22}]_{wj}\right],\,\, \Delta^{(0)}_{\textnormal{PPR}} \;=\; \min_{i \leq N}\frac{1-\alpha}{|\mathcal{S}|}\sum_{v \in \mathcal{S}}[R_{11}]_{vi},
	\end{equation}
	where $s_v^{\textnormal{eff}} := 1 + \alpha\sum_{w > N}[R_{22}P^{21,(0)}]_{wv} \geq 1$ for $v \leq N$. Moreover, $\pi_{\mathcal{S}}^{(0)}(j) = 0$ for every $j > N$, so the max-leakage term in \eqref{eq:delta0-defs} vanishes identically. \arXivtag[Here, row and column indices of $R_{22}$ start at $N+1$, consistent with the full $\tilde{N}\times\tilde{N}$ indexing.]
\end{lemma}
Obtaining explicit lower bounds on the auxiliary gaps $\Delta^{(0)}_{\textup{method}}$ from the above expressions and combining with \eqref{eq:detection_reduction} then yields the sufficient condition  for detection on $P$ in Theorem \ref{thm:detection}.\arXivtag The two expressions enter Theorem~\ref{thm:detection} in structurally different ways.

\emph{PPR: direct use in (ii).} The closed form \eqref{eq:delta-closed} depends only on $R_{11}$ evaluated at the seed set $\mathcal{S}$ and is directly usable as a detection threshold. Positivity $\Delta^{(0)}_{\textnormal{PPR}} > 0$ reduces to reachability---every $i \leq N$ must be reachable from some $v \in \mathcal{S}$ in $P^{11,(0)}$---because $[R_{11}]_{vi} > 0$ iff $v$ reaches $i$ in $P^{11,(0)}$.

\emph{PR: relaxation to a usable bound for (i).} The closed form \eqref{eq:delta-closed} couples $R_{11}$ to \arXivtag[bottom block] quantities through $s^{\textnormal{eff}}$ and the column sums of $R_{22}$, obstructing a clean condition on $P^{11,(0)}$ alone. Dropping $s_v^{\textnormal{eff}} \geq 1$ to $1$ and bounding $\|R_{22}\|_\infty \leq 1/(1-\alpha+\alpha\eta)$ (from $\|P^{22}\|_\infty \leq 1-\eta$) yields the \emph{mixing relaxation}
\begin{equation}\label{eq:delta-PR-bound}
	\Delta^{(0)}_{\textnormal{PR}} \;\geq\; \frac{N}{\tilde{N}}\,q_{\min}^\alpha \;-\; \frac{(\tilde{N}-N)(1-\alpha)}{\tilde{N}(1-\alpha+\alpha\eta)},
\end{equation}
where $(\mathbf{q}^{\alpha})^\tp := \tfrac{1-\alpha}{N}\mathbf{1}_N^\tp R_{11}$ and $q_{\min}^\alpha := \min_{i \leq N}q_i^\alpha$. Positivity of the right-hand side requires the \emph{mixing condition} $q_{\min}^\alpha > (\tilde{N}-N)(1-\alpha)/\bigl(N(1-\alpha+\alpha\eta)\bigr)$, which is strictly stronger than reachability.

	\vspace{-0.2cm}
\section{Numerical experiments}\label{sec:examples}
	\vspace{-0.3cm}
We now test Algorithm~\ref{algor:PR-EDMD} on four systems of increasing complexity. For error plots, we use 20 random seeds and depict the mean$\pm$1SD by shaded regions. 
Full experimental details (dictionaries, parameters, precise definitions of errors, additional figures) are in Appendix~\ref{app:experiments}, and all relevant codes can be found at \texttt{\anonrepo}.

\begin{wrapfigure}[14]{r}{0.52\textwidth}
	\vspace{-12pt}
	\centering
	\includegraphics[width=0.52\textwidth]{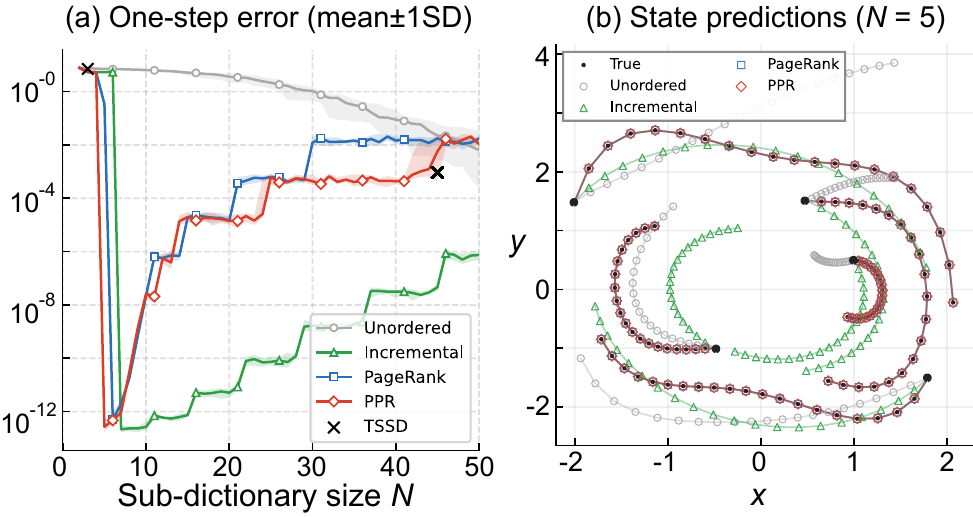}
	\caption{Duffing oscillator. (a) One-step error vs.\ sub-dictionary size. (b) 20-step state predictions with 5 observables.}
	\label{fig:Duffing-pred-error}
\end{wrapfigure}
\textbf{Duffing and Van der Pol oscillators.} We first apply the method to the Duffing oscillator ($\dot{x}=y$, $\dot{y}=-\delta y + \gamma x - \beta x^3$) and the Van der Pol oscillator, using 90 bivariate Laguerre polynomials with $M = 2000$ samples. In both systems, the PPR-ordered EDMD matrix reveals a clear approximate zero block structure (Figs.~\ref{fig:Duffing-heatmap}--\ref{fig:vdp-heatmap} in Appendix~\ref{app:experiments}). 
The PPR-selected sub-dictionary achieves $\sim 10^{-12}$ one-step error with only 5 observables, outperforming the others in multi-step state predictions as well (Fig.~\ref{fig:Duffing-pred-error}).
For the tunable symmetric subspace decomposition (TSSD), we varied the accuracy parameter $\epsilon$ from 0 to 1 with 0.01 intervals, but only three sizes $N=3, 45, 65$ are detected.
For the Van der Pol oscillator, the advantage of PPR over standard PR is more pronounced (Fig.~\ref{fig:vdp-pred-error} in Appendix~\ref{app:experiments}). 

\begin{wraptable}{r}{0.52\textwidth}
	\centering
	\small
	\caption{Ratio of the error for each ordering to the error for a random ordering, reported as mean \(\pm\) standard deviation over 20 random seeds. Values below 1 indicate improvement over random ordering.}
	\label{tab:ala2}
	\begin{tabular}{l c c c c}
		\toprule
		Method & $N\!=\!5$ & $N\!=\!10$ & $N\!=\!20$ \\
		\midrule
		PPR   & \textbf{0.30 $\pm$ 0.04} & \textbf{0.38 $\pm$ 0.06} & \textbf{0.55 $\pm$ 0.08} \\
		PR    & 1.09 $\pm$ 0.18 & 1.18 $\pm$ 0.21 & 1.33 $\pm$ 0.22 \\
		PCCA+ & 0.68 $\pm$ 0.11 & 0.77 $\pm$ 0.14 & 0.99 $\pm$ 0.16 \\
		TICA  & 0.94 $\pm$ 0.15 & 1.06 $\pm$ 0.19  & 1.44 $\pm$ 0.24 \\
		\bottomrule
	\end{tabular}
	\vspace{-8pt}
\end{wraptable}

\textbf{Three-well Ramachandran potential.} We test on a molecular dynamics benchmark: overdamped Langevin diffusion on a three-well Ramachandran potential~\cite{ramachandran1963stereochemistry} (see Appendix~\ref{app:experiments} for details). The dictionary consists of $\tilde{N} = 236$ observables (Fourier modes, cross terms, Gaussian RBFs, and 4 state coordinates $\sin\phi, \cos\phi, \sin\psi, \cos\psi$), with $M = 10^5$ and 20 random seeds. We compare PPR (seeded at the 4 state coordinates, since predicting the angular state is the primary objective), standard PR, time-lagged independent component analysis (TICA)~\cite{perez2013identification}, PCCA+~\cite{deuflhard2005robust}, and random ordering. Note that this example is stochastic and is therefore outside the deterministic theory above; it is included as an empirical demonstration of the PPR selection principle for Markovian dynamics.

\begin{wrapfigure}[13]{r}{0.52\textwidth}
	 \vspace{-12pt}
	\centering
	\includegraphics[width=0.52\textwidth]{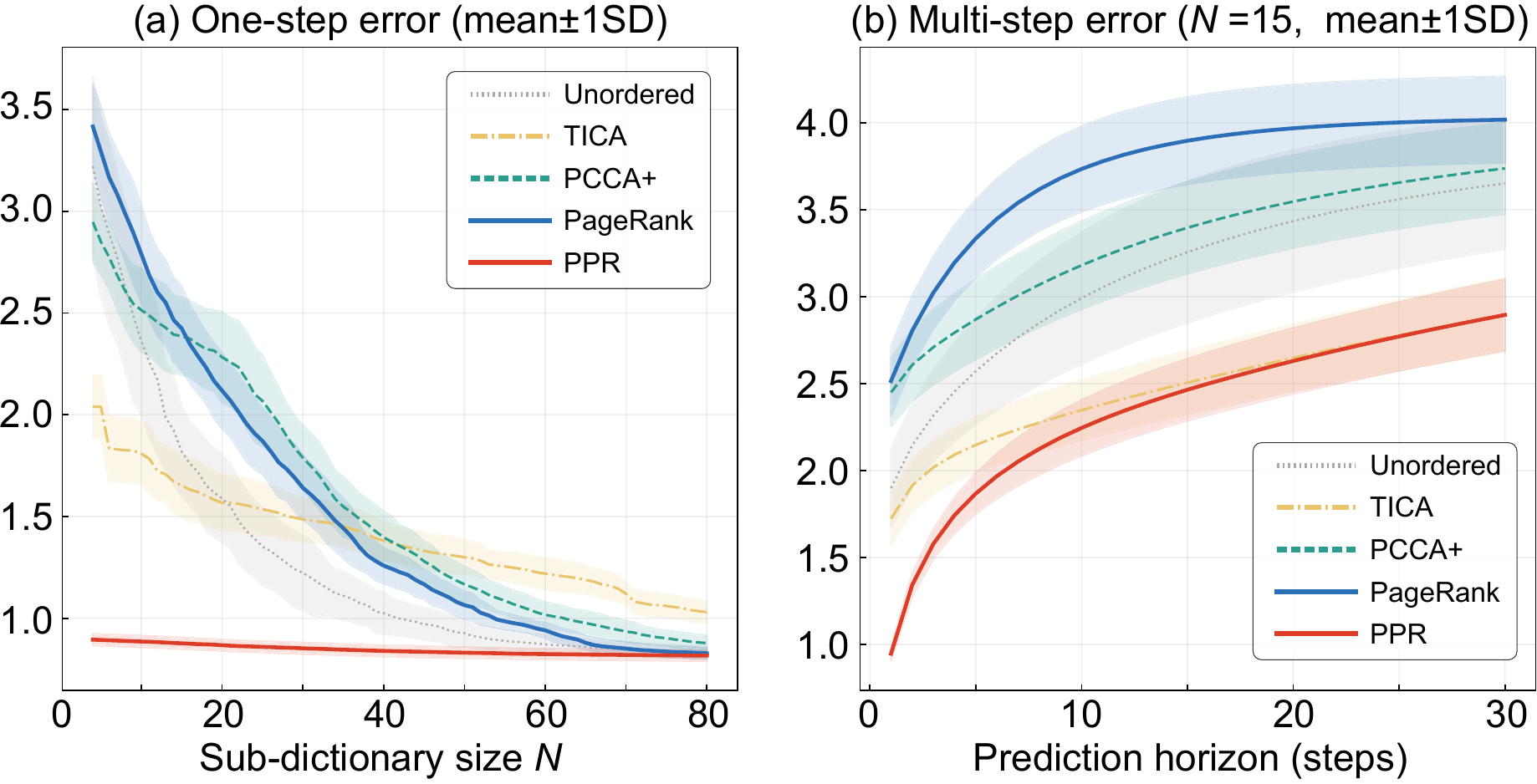}
	\vspace{-0.5cm}
	\caption{Three-well Ramachandran potential ($\tilde{N} = 236$). (a) One-step error vs.\ sub-dictionary size. (b) Multi-step prediction errors at $N=15$. }
	\label{fig:ala2-errors}
	\vspace{-6pt}
\end{wrapfigure}

For one-step predictions, PPR significantly outperforms all other methods at small sub-dictionary sizes (Table~\ref{tab:ala2}): at $N=5$, PPR reduces one-step error by 70\% over Random, compared to 32\% for PCCA+ and 6\% for TICA. The advantage persists at all $N$; TICA and PR degrade as $N$ gets larger, underperforming even Random. Fig.~\ref{fig:ala2-errors} shows that PPR achieves the lowest one-step and multi-step prediction errors at all horizons, confirming good internal dynamics propagation. 

The composition of selected observables is physically revealing: PPR places the 4 state coordinates at ranks 1--4, followed by basin-boundary RBFs (Fig.~\ref{fig:ala2-landscape} in Appendix~\ref{app:experiments}). The four state coordinates capture three macrostates. On the other hand, the basin-boundary RBFs seem to effectively capture inter-basin transitions, yielding accurate predictions with smaller sub-dictionaries. 

\begin{wrapfigure}[14]{r}{0.52\textwidth}
	\vspace{-12pt}
	\centering
	\includegraphics[width=0.52\textwidth]{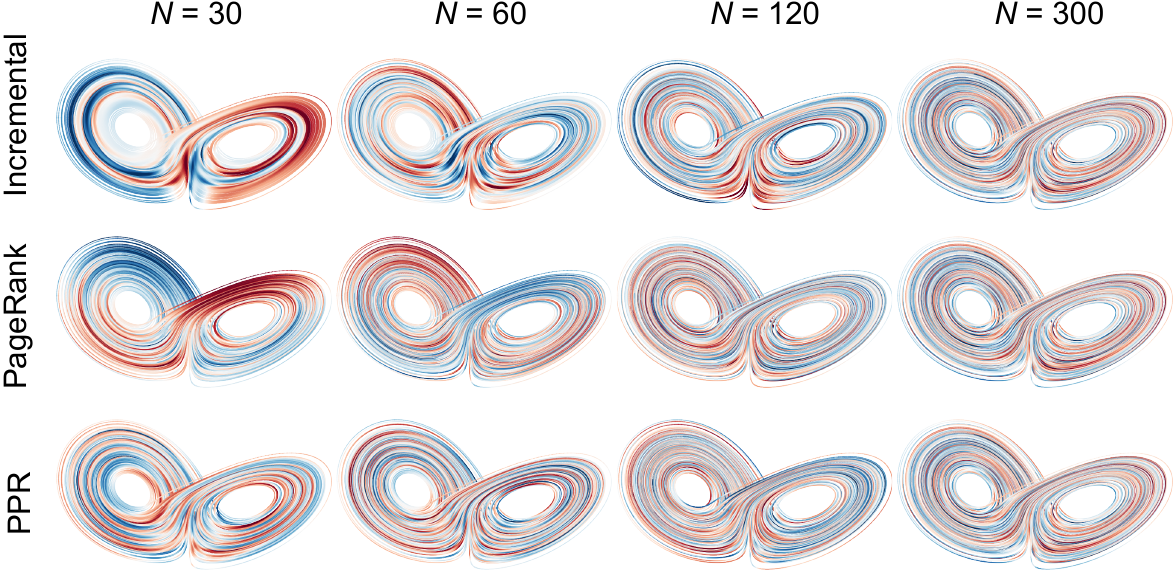}
	\caption{Pseudo-eigenfunctions ($\approx 6\,\textnormal{rad/s}$) on the Lorenz attractor for varying sub-dictionary sizes $N$. PPR recovers the oscillatory structure at $N = 30$--$60$, while Incremental requires $N \geq 120$.}
	\label{fig:lor-eig}
	\vspace{-8pt}
\end{wrapfigure}
\textbf{Lorenz system.} We apply Algorithm~\ref{algor:PR-EDMD} to the Lorenz system in the chaotic regime ($\sigma\!=\!10, \rho\!=\!28, \beta\!=\!8/3$), using a time-delay embedding with $\tilde{N}=300$ observables and $M = 2\times 10^5$ snapshot pairs. This trajectory-sampling experiment is not covered by the i.i.d. finite-sample theorem; it tests the method in the ergodic-data regime, where the data is sampled according to an invariant physical SRB measure. Rather than prediction error, we examine whether a PPR-selected sub-dictionary recovers dynamically relevant pseudo-eigenfunctions. Fig.~\ref{fig:lor-eig} shows the real parts of the pseudo-eigenfunctions at frequency $\approx 6\,\textnormal{rad/s}$: PPR captures the characteristic oscillatory structure wrapping around unstable periodic orbits~\cite{Colbrook2024-MultiverseDMD, colbrook2024rigorous} at $N = 30$--$60$, while Incremental ordering requires $N \geq 120$. This demonstrates that PPR identifies a compressed, informative sub-dictionary for Koopman spectral analysis.

	\vspace{-0.2cm}
\section{Conclusions}\label{sec:conclusions}
	\vspace{-0.3cm}
We developed a framework connecting zero block structures in EDMD matrices to Koopman invariant subspaces, and proposed Algorithm~\ref{algor:PR-EDMD}, which applies PPR to a row-normalized EDMD matrix to select sub-dictionaries. The off-diagonal block norm controls prediction accuracy \arXivtag[(Propositions~\ref{prop:error-block} and \ref{prop:error-bound})], and exact invariant subspaces produce exact zero blocks with finitely many data points (Theorem~\ref{thm:blockstr}). PPR detects these blocks at any damping factor (Theorem~\ref{thm:detection}) and achieves end-to-end detection at rate $O(1/\sqrt{M})$ with sample-complexity bounds for both PR and PPR (Theorem~\ref{thm:end-to-end}, Corollary~\ref{cor:sample-complexity}). More broadly, the multi-step Koopman leakage is bounded by the PPR mass outside the selected set (Proposition~\ref{prop:unconditional}), justifying the method without assuming invariant subspace existence: PPR returns the sub-dictionary that best closes the dynamics around those seeds. 

Experiments on the Duffing and Van der Pol oscillators, the Lorenz system, and the three-well Ramachandran potential confirm that PPR identifies compact, interpretable dictionaries, reducing one-step prediction error by up to $\sim 70\%$ over random selection at \arXivtag[a small sub-dictionary size] $N$ (see Table~\ref{tab:ala2}).
Future directions include constructing invariant structures beyond subsets of a prescribed dictionary and adaptive selection of the damping factor $\alpha$. Although the experiments indicate that our algorithm performs well for ergodic trajectory data (i.e., the Lorenz system) and stochastic systems (i.e., the Ramachandran potential), our deterministic theory assumes noise-free i.i.d. samples. Extension of the theory to ergodic trajectory data~\cite{nuske2023finite, elias2024datadriven} and to stochastic systems, where the Koopman operator measures expectations of observables and variance becomes an independent source of error, is an important open direction; the stochastic ResDMD framework of Colbrook et al.~\cite{colbrook2024beyond} provides a natural starting point for such an extension.

\begin{ack}
Q.L. acknowledge support from NSF-DMS-2308440, ONR-GRANT14408071 and Vilas Associate Award. 
\end{ack}

{\small
	\bibliographystyle{unsrt}
	\bibliography{References-Koopman}
}

\newpage

\newpage 

\appendix

\section{Background}\label{app:prelim}
This appendix provides a more comprehensive overview on the backgrounds on related literature summarized in Section~\ref{sec:algorithm}.

\subsection{Notation}\label{app:notation}
For convenience, we collect the key notation used throughout the paper in Table~\ref{tab:notation}.

\begin{table}[h]
	\centering
	\caption{Summary of key notation.}
	\label{tab:notation}
	\small
	\begin{tabular}{ll}
		\hline
		Symbol & Meaning \\
		\hline
		$\Dc_{N}$ & Dictionary of size $N$, i.e., $\Dc_{N} = \{\psi_1, \ldots, \psi_N\}$ \\ 
		$\Fc_{N}$ & Subspace spanned by $\Dc_{N}$ \\ 
		$K_{N,M}$ & EDMD matrix (with $M$ data points, dictionary of size $N$) \\
		$K_N$ & \arXivtag[Population-level Koopman matrix (i.e., $K_N=\lim_{M\to \infty}K_{N,M}$)] \\
		$K[i,j]$ & $(i,j)$-entry of matrix $K$; $K[i_1:i_2,\, j_1:j_2]$ denotes a submatrix \\
		$K_{\tilde{N}}^{(0)}$ &  Matrix $K_{\tilde{N}}$ with bottom-left block zeroed out \\
		$K_{\tilde{N},M}^{(0)}$ &  Matrix $K_{\tilde{N},M}$ with bottom-left block zeroed out \\
		$\mathcal{R}(A)$ & Row stochastic matrix obtained by row normalizing matrix $A$ \\ 
		$P$ &  Row-normalized $|K_{\tilde N,M}^{\tp}|$, i.e., \arXivtag[$P=\mathcal{R}(|K_{\tilde N,M}^{\tp}|)$], interpreted as a Markov chain \\
		$P^{(0)}$ &  Row-normalized $|K^{(0)}_{\tilde N,M}|^{\tp}$, i.e., $P^{(0)} = \mathcal{R}(|K^{(0)}_{\tilde N,M}|^{\tp})$ \\
		$\pi$ & Generic notation for \arXivtag[PR/PPR] vectors \\
		$\pi_{v}, \pi_{\mathcal{S}}, {\pi}_{\mathbf{s}}$ & PPR vectors with a single seed $v$, a seed set $\mathcal{S}$, and a preference vector $\mathbf{s}$, respectively. \\
		$\pi_M$, $\pi_{\mathbf{s}, M}$ & PR/PPR vectors associated EDMD matrices from $M$ data samples \\ 
		$\pi^{(0)}, \pi_{\mathcal{S}}^{(0)}$ & PR/PPR vectors of $P^{(0)}$ \\
		$\alpha$ & Damping factor $\in (0,1)$ for PR and PPR algorithms\\
		$\eta$ & Coupling parameter $= 1 - \|P^{22}\|_\infty.$ \\
		\hline
	\end{tabular}
\end{table}

\subsection{The Koopman operator}

The Koopman operator framework, introduced by Koopman~\cite{koopman1931hamiltonian} and extended by Koopman and von Neumann~\cite{Koopman-Neumann1932}, lifts a nonlinear dynamical system on the state space $\Mc$ to a linear (but infinite-dimensional) operator on a function space of \emph{observables}. This linearization is the foundation of modern data-driven analysis of nonlinear dynamics: dynamic mode decomposition (DMD)~\cite{schmid2010dynamic,rowley2009spectral,kutz2016dynamic} and its extensions~\cite{williams2015data,Klus2016JCompDyn,kevrekidis2016kernel} approximate the Koopman operator from trajectory data, and its spectral theory underpins mode decomposition, long-time forecasting, coherent-structure discovery, and reduced-order modeling across fluid mechanics, molecular dynamics, power systems, and control~\cite{Mezic2013AnnRevFluid,otto2021koopman,mezic2021koopman,colbrook2024beyond}.

We consider a discrete-time dynamical system
$\x_{k+1} = \mathbf{F}(\x_k)$, $k=0,1,2, \ldots$
with $\mathbf{F}:\Mc \to \Mc$ where $\x \in \Mc$ denotes the state of the system and $\Mc \subset \R^d$ is the state space. 
Let $\Fc$ be the Hilbert space $L^2(\Mc, \mu)$ where $\mu$ is the distribution on $\Mc$ from which data points are sampled. Throughout, we take $\mu$ to be a (positive) Borel probability measure on $\Mc$ and assume $\mathbf{F}$ is $\mu$-\emph{nonsingular}, i.e.\ $\mathbf{F}_{\#}\mu \ll \mu$; this ensures that the pointwise composition $f\circ\mathbf{F}$ is well-defined on $\mu$-equivalence classes and hence on $\Fc$. The induced \emph{Koopman operator} $\Koop:\Fc\to\Fc$, $\Koop f = f\circ\mathbf{F}$, is bounded precisely when the Radon--Nikodym derivative $d(\mathbf{F}_{\#}\mu)/d\mu \in L^\infty(\mu)$, in which case $\|\Koop\|_{\Fc\to\Fc}^2 = \|d(\mathbf{F}_{\#}\mu)/d\mu\|_{L^\infty}$; it is an isometry when $\mu$ is $\mathbf{F}$-invariant. We assume this \arXivtag[nonsingularity] condition throughout. Given linearly independent functions $\psi_i \in \Fc$, $i=1,\ldots, N$, define a finite-dimensional subspace $\Fc_N = \Span\{\psi_1, \ldots, \psi_N\} \subset \Fc$ and denote the Koopman operator restricted on this subspace by $\Koop_{|\Fc_N}:\Fc_N \to \Fc$.

\paragraph{Linearity of $\Koop$.}
The defining feature of the Koopman formalism is that $\Koop$ is a \emph{linear} operator, irrespective of whether the underlying dynamics $\mathbf{F}$ is linear or nonlinear. Linearity is immediate from the definition: for $f,g\in\Fc$ and $a,b\in\mathbb{C}$,
\begin{equation}\label{eq:koopman-linearity}
	\Koop(af + bg)(\x) = (af + bg)\bigl(\mathbf{F}(\x)\bigr) = a\,f(\mathbf{F}(\x)) + b\,g(\mathbf{F}(\x)) = a\,(\Koop f)(\x) + b\,(\Koop g)(\x)
\end{equation}
for every $\x\in\Mc$, so $\Koop(af+bg) = a\Koop f + b\Koop g$. This identity holds pointwise on $\Mc$ and hence in $\Fc = L^2(\Mc,\mu)$. 

\paragraph{Linearity comes at the cost of dimension.}
Linearity on observables of $\Koop$ is a \emph{global and exact} property, but it is purchased by leaving the finite-dimensional state space $\Mc\subset\R^d$ in exchange for the typically infinite-dimensional function space $\Fc$. This is qualitatively different from the Jacobian linearization $D\mathbf{F}(\x^\star)$, which is a \emph{local} linear approximation valid only in a neighborhood of a fixed point: $\Koop$ represents the full nonlinear dynamics everywhere on $\Mc$, losslessly. The central modeling question raised by this trade-off, and the one that motivates our paper, is whether a finite-dimensional subspace $\Fc_N\subset\Fc$ can be chosen so that $\Koop|_{\Fc_N}$ is a faithful finite-matrix restriction of $\Koop$. For example, a subspace spanned by Koopman eigenfunctions is invariant under the action given by $\Koop$. 

\paragraph{Spectral consequences.}
Because $\Koop$ is linear, the full machinery of operator theory applies. It has a spectrum $\sigma(\Koop)\subset\mathbb{C}$, and eigenpairs $(\lambda_j,\varphi_j)$ with $\Koop\varphi_j = \lambda_j\varphi_j$ are the \emph{Koopman eigenvalues} and \emph{eigenfunctions}. Any observable $f\in\Fc$ admitting an eigenfunction expansion $f = \sum_j c_j\varphi_j$ evolves diagonally under iteration:
\begin{equation}\label{eq:koopman-spectral-evolution}
	(\Koop^k f)(\x) = f(\mathbf{F}^k(\x)) = \sum_j c_j\,\lambda_j^k\,\varphi_j(\x),
\end{equation}
so nonlinear trajectories decompose into a superposition of geometrically scaling/rotating \emph{Koopman modes}~\cite{Mezic2013AnnRevFluid,rowley2009spectral}. When $\mathbf{F}$ preserves the measure $\mu$, $\Koop$ is an isometry on $L^2(\Mc,\mu)$ with $\sigma(\Koop)\subset\{z\in\mathbb{C}:|z|\leq 1\}$, and the unimodular part of the spectrum encodes the ergodic and quasi-periodic structure of the dynamics. These consequences\arXivtag, which are inaccessible from nonlinear state-space analysis alone, are what make the Koopman framework powerful for mode decomposition, reduced-order modeling, and long-horizon forecasting of nonlinear systems~\cite{mezic2021koopman,otto2021koopman,colbrook2024beyond}.

\paragraph{Finite-dimensional restriction.}
Linearity also underlies the matrix representation used throughout this paper. If $\Fc_N=\Span\{\psi_1,\ldots,\psi_N\}$ is $\Koop$-invariant, then for each $j$ there exist unique coefficients $K_{ij}\in\mathbb{C}$ with $\Koop\psi_j=\sum_{i=1}^N K_{ij}\psi_i$, and the matrix $K=(K_{ij})\in\mathbb{C}^{N\times N}$ represents $\Koop|_{\Fc_N}$ in the basis $\{\psi_i\}$. When $\Fc_N$ is only \emph{approximately} invariant, the corresponding finite matrix captures only the Galerkin projection of $\Koop$ onto $\Fc_N$, and the deviation of $\Koop\psi_j$ from $\Fc_N$ is the source of the residual structure analyzed in the sequel.

\subsection{Extended dynamic mode decomposition}\label{sec:app-edmd}

\emph{Convention.} Throughout this appendix and all proofs in Appendix~\ref{sec:proofs}, we use the same row-vector convention as in the main text: $\mathbf{\Psi}(\x) \in \mathbb{C}^{1\times N}$, $\mathbf{\Psi}(\X) \in \mathbb{C}^{M\times N}$, and $K_{N,M} = \mathbf{\Psi}(\X)^{+}\,\mathbf{\Psi}(\Y)$ as in \eqref{eq:EDMD-mat-formula}. \arXivtag[To follow the convention of transition matrices in Markov chain theory, we define the row-stochastic matrix $P = \mathcal{R}(|K_{\tilde N,M}^{\tp}|)$] by row-normalizing the \emph{transpose} of $|K_{\tilde N,M}|$.

Extended dynamic mode decomposition (EDMD), introduced by Williams, Kevrekidis, and Rowley~\cite{williams2015data} as a generalization of the original DMD algorithm~\cite{schmid2010dynamic,rowley2009spectral}, approximates the Koopman operator by Galerkin projection onto the span of a user-chosen dictionary of observables. Convergence to the $L^2$-projected Koopman operator as $M \to \infty$ was established by Korda and Mezi\'c~\cite{korda2018convergence}; kernel and Hilbert-space extensions appear in~\cite{kevrekidis2016kernel,Klus2016JCompDyn}, and recent work has produced rigorous finite-sample rates and spectral guarantees~\cite{zhang2023quantitative,nuske2023finite,colbrook2024rigorous,Kostic-NeuRips2023}. The quality of the approximation hinges on whether the dictionary spans an (approximately) Koopman invariant subspace\arXivtag[,] which is the motivation for the sub-dictionary selection problem considered in this paper.

Let $\{\x^{(1)}, \ldots, \x^{(M)}\} \subset \Mc$ be i.i.d. samples from $\mu$ and $\{\y^{(1)}, \ldots, \y^{(M)}\}$ their one-step evolutions, i.e., $\y^{(i)} = \mathbf{F}(\x^{(i)})$.
A \textit{dictionary} of linearly independent observables is $\arXivtag[\Dc_N =] \{\psi_1, \ldots, \psi_N\} \subset \Fc$. With $\mathbf{\Psi}(\x) = [\psi_1(\x), \ldots, \psi_N(\x)] \in \mathbb{C}^{1\times N}$, define
\begin{equation}\label{eq:eval-mat1}
	\mathbf{\Psi}(\X) = \begin{bmatrix}
		\mathbf{\Psi}(\x^{(1)}) \\ \vdots \\ \mathbf{\Psi}(\x^{(M)})
	\end{bmatrix} \in \mathbb{C}^{M\times N}, \quad
	\mathbf{\Psi}(\Y) = \begin{bmatrix}
		\mathbf{\Psi}(\y^{(1)}) \\ \vdots \\ \mathbf{\Psi}(\y^{(M)})
	\end{bmatrix} \in \mathbb{C}^{M\times N}.
\end{equation}
The \textit{EDMD matrix} \arXivtag[with the dictionary $\Dc_N$ and the data set $\{\x^{(1)}, \ldots, \x^{(M)}\}$] is $K_{N,M} = \mathbf{\Psi}(\X)^{+}\,\mathbf{\Psi}(\Y)$. If $\mathbf{\Psi}(\X)$ has full column rank, then $K_{N,M} = \bigl(\mathbf{\Psi}(\X)^{\ast}\mathbf{\Psi}(\X)\bigr)^{-1}\mathbf{\Psi}(\X)^{\ast}\mathbf{\Psi}(\Y)$. Here we make the following remark on Assumption~\ref{assume:full-rank} that states $\mathbf{\Psi}(\X)$ has full column rank.

\begin{remark}\label{rem:fullcolrank}
	Assumption~\ref{assume:full-rank} is ensured  \textit{almost surely} with respect to the distribution $\mu$ by the following condition of measure-theoretic form: 
	$\mu(\{ \x \in \Mc : c_1 \psi_1(\x) + \ldots + c_N\psi_N(\x) = 0 \}) = 0$ for any $(c_1, \ldots, c_N) \in \R^N \setminus\{\mathbf{0}\}$,
    \arXivtag[, which follows Assumption 1 in~\cite{korda2018convergence}.] 
	This holds, for instance, for real-analytic basis functions when $\mu$ is not supported on a proper real-analytic subvariety of $\Mc$; more generally, it suffices that the span of the $\{\psi_i\}$ contains no function that vanishes on a $\mu$-positive set. We state the condition in this measure-theoretic form so that it is meaningful even when $\Mc$ has no interior (e.g., an attractor or fractal invariant set): what matters is that $\mu$ charges enough of $\Mc$ to distinguish every nonzero linear combination of dictionary functions, not that $\Mc$ has full topological support.
\end{remark}

When $\Fc_N$ is invariant under $\Koop$, the Frobenius error vanishes, giving
\begin{equation}\label{eq:KpsiX-psiY}
	\mathbf{\Psi}(\X)\,K_{N,M} = \mathbf{\Psi}(\Y).
\end{equation}
That is, if a dictionary spans a Koopman invariant subspace, the finite-dimensional EDMD matrix exactly captures the Koopman dynamics on that subspace.

\paragraph{Galerkin / $L^2$-projection interpretation.}
In the infinite-data limit $M\to\infty$, the Gram matrices converge (by the strong law of large numbers) to their population counterparts $G_{ij} = \langle \psi_i, \psi_j\rangle_\mu$ and $A_{ij} = \langle \Koop\psi_j, \psi_i\rangle_\mu$, and $K_{N,M}$ converges to $K_N = G^{-1}A$, the matrix of the \emph{Galerkin projection} $\mathcal{P}_{\Fc_N}\Koop|_{\Fc_N}$ in the basis $\{\psi_i\}$~\cite{korda2018convergence}. Here $\mathcal{P}_{\Fc_N}$ is the $L^2(\Mc,\mu)$-orthogonal projection onto $\Fc_N$. Thus EDMD is not an arbitrary fit: it recovers the best $\Fc_N$-representation of $\Koop$ in the $L^2$ sense. When $\Fc_N$ is exactly $\Koop$-invariant, $\mathcal{P}_{\Fc_N}\Koop|_{\Fc_N} = \Koop|_{\Fc_N}$ and EDMD is exact, so the residual $\mathbf{\Psi}(\Y) - \mathbf{\Psi}(\X)\,K_{N,M}$ captures the failure of invariance.

\paragraph{Relation to DMD.}
The original DMD algorithm~\cite{schmid2010dynamic,rowley2009spectral,kutz2016dynamic} corresponds to \arXivtag[EDMD with] the choice of identity dictionary $\mathbf{\Psi}(\x) = \x$: the EDMD matrix reduces to the best linear fit $Y \approx AX$ between snapshot matrices, exact only when $\mathbf{F}$ is itself linear. EDMD generalizes this by letting $\mathbf{\Psi}$ contain arbitrary nonlinear observables, such as monomials, radial basis functions, Hermite polynomials, or data-driven features, so that the method can capture nonlinear dynamics that admit a finite-dimensional linear representation on the lifted space. Kernel DMD~\cite{kevrekidis2016kernel} sidesteps the explicit choice of dictionary by operating in a reproducing kernel Hilbert space; \arXivtag[Klus et al.] unifies EDMD, kernel EDMD, and variational approaches to Koopman and Perron--Frobenius approximation \cite{Klus2016JCompDyn}.

\paragraph{Finite-sample and spectral guarantees.}
Quantitative finite-sample error bounds for EDMD have been a recent focus. Zhang and Zuazua~\cite{zhang2023quantitative}, N\"uske et al.~\cite{nuske2023finite}, and Colbrook et al.~\cite{colbrook2024rigorous} give $\arXivtag[1/\sqrt{M}]$-type rates with explicit constants depending on the dictionary Gram condition number and sub-Gaussian tail parameters of $\mathbf{\Psi}(\x)$. Kostic et al.~\cite{Kostic-NeuRips2023} obtain sharper operator-norm bounds in RKHS settings via concentration of the empirical covariance; \arXivtag[Colbrook et al.] address the difficult problem of \emph{spurious spectra}---eigenvalues of $K_{N,M}$ that do not correspond to any spectral object of $\Koop$---through pseudospectral and residual-DMD techniques~\cite{colbrook2024rigorous,colbrook2024beyond}. The concentration-of-measure tradition built on matrix Bernstein inequalities forms the technical backbone of these results~\cite{tropp2015introduction,tropp2012user}.

\paragraph{Dictionary design and learning.}
The practical performance of EDMD is dominated by dictionary choice. Hand-crafted dictionaries include monomials up to a \arXivtag[certain] degree, radial basis functions~\cite{williams2015data}, and Hermite polynomials; data-driven approaches learn the dictionary jointly with $K_{N,M}$ via neural networks~\cite{li2017extended,yeung2019learning,takeishi2017learning,wehmeyer2018time} or regression-with-sparsity schemes~\cite{brunton2016SINDy,brunton2016PLOSONE}. When samples come from a single trajectory rather than i.i.d.\ draws from $\mu$, ergodic variants of EDMD are required~\cite{arbabi2017ergodic,kawahara2016dynamic}. In all these settings, the underlying question is the same: does the chosen dictionary span (or nearly span) a $\Koop$-invariant subspace~\cite{haseli2021learning,haseli2023generalizing}? Answering this question a posteriori, from the structure of $K_{N,M}$ itself, is the problem addressed by the PPR-based detection method developed in this paper.

\subsection{Personalized PageRank}

PageRank \arXivtag[(PR)] was introduced by \arXivtag[Page and Brin]~\cite{PageRank1998} to rank web pages by modeling a random surfer on the hyperlink graph who occasionally teleports to a uniformly random page. The algorithm's mathematical properties\arXivtag[, such as] convergence, sensitivity to the damping parameter, and robustness of the stationary distribution\arXivtag[,] were further developed in~\cite{langville2004deeper,ipsen2006pagerank}. 
Personalized PageRank (PPR), due to Haveliwala~\cite{haveliwala2003topic}, replaces the uniform teleportation distribution by a user-specified \emph{preference vector}, localizing the ranking around a chosen seed set. Andersen, Chung, and Lang~\cite{andersen2006local} established Cheeger-type guarantees for local clustering via sweep cuts on PPR vectors, and PPR has since become a canonical tool in graph learning, community detection, and semi-supervised learning on networks. In this paper we exploit PPR in a novel setting: we interpret the EDMD matrix itself as a weighted graph on observables and use PPR mass from a seed set of ``target'' observables to rank sub-dictionaries by the strength of their Koopman coupling.

Let $Q \in \R^{n \times n}$ be a row stochastic matrix, $\mathbf{s} \in \R^n$ a \emph{preference vector} satisfying $\mathbf{s}^\tp\mathbf{1} = 1$ and $s_i \geq 0$, and $\alpha \in [0,1)$ a \emph{damping factor}. Define the row stochastic \emph{transition kernel} 
\begin{equation}\label{eq:ppr-kernel}
	G_{\mathbf{s}} \coloneqq \alpha Q + (1-\alpha)\,\mathbf{1}\mathbf{s}^\tp.
\end{equation}
This describes a random walker who, at each step, follows $Q$ with probability $\alpha$ and teleports to a node drawn from $\mathbf{s}$ with probability $1-\alpha$. Since every entry satisfies $G_{\mathbf{s}}[i,j] \geq (1-\alpha)s_j$, the chain is primitive whenever $\mathbf{s}$ has full support, guaranteeing a unique stationary distribution by the Perron--Frobenius theorem.

The \emph{\arXivtag[PPR] vector} ${\pi}_{\mathbf{s}}$ is the stationary distribution of $G_{\mathbf{s}}$. From ${\pi}_{\mathbf{s}}^\tp = {\pi}_{\mathbf{s}}^\tp G_{\mathbf{s}}$ and ${\pi}_{\mathbf{s}}^\tp\mathbf{1} = 1$:
\begin{equation}\label{eq:ppr-fixedpoint}
	{\pi}_{\mathbf{s}}^\tp = \alpha\,{\pi}_{\mathbf{s}}^\tp Q + (1-\alpha)\,\mathbf{s}^\tp.
\end{equation}
Rearranging, ${\pi}_{\mathbf{s}}^\tp(I - \alpha Q) = (1-\alpha)\,\mathbf{s}^\tp$. Since $Q$ is row stochastic, all eigenvalues of $Q$ have modulus at most $1$, and $\alpha < 1$ ensures that $I - \alpha Q$ is nonsingular. Therefore
${\pi}_{\mathbf{s}}^\tp = (1-\alpha)\,\mathbf{s}^\tp(I - \alpha Q)^{-1}$.
Expanding the resolvent as a Neumann series $(I - \alpha Q)^{-1} = \sum_{k=0}^\infty \alpha^k Q^k$ gives
${\pi}_{\mathbf{s}}^\tp = (1-\alpha)\sum_{k=0}^\infty \alpha^k\, \mathbf{s}^\tp Q^k$,
which converges absolutely since $\alpha < 1$ and $\|Q^k\|_\infty = 1$. The entry $\pi_{\mathbf{s}}(j)$ equals the expected discounted fraction of time at node $j$; the factor $\alpha^k$ geometrically discounts $k$-step transitions. \arXivtag
A key structural property is \emph{linearity in the preference vector}: since the resolvent $(I - \alpha Q)^{-1}$ is independent of $\mathbf{s}$, the map $\mathbf{s} \mapsto {\pi}_{\mathbf{s}}$ is linear. All special cases are instances of this linearity.

\paragraph{Multi-seed PPR.}
For a seed set $\mathcal{S} \subset \{1,\ldots,n\}$, setting $\mathbf{s} = \frac{1}{|\mathcal{S}|}\sum_{v \in \mathcal{S}}\mathbf{e}_v$ gives the multi-seed PPR vector ${\pi}_{\mathcal{S}} = \frac{1}{|\mathcal{S}|}\sum_{v \in \mathcal{S}} {\pi}_v$ by linearity. The corresponding transition kernel is $G_{\mathbf{s}} = \alpha Q + \frac{1-\alpha}{|\mathcal{S}|}\sum_{v \in \mathcal{S}}\mathbf{1}\mathbf{e}_v^\tp$: the walker teleports uniformly to a node in $\mathcal{S}$.

\paragraph{Single-seed PPR.}
The special case $\mathcal{S} = \{v\}$ (equivalently, $\mathbf{s} = \mathbf{e}_v$) gives
\begin{equation}\label{eq:ppr-singleseed}
	\pi_v(j) = (1-\alpha)\sum_{k=0}^\infty \alpha^k\, Q^k[v,j],
\end{equation}
the expected discounted occupation of node $j$ for a walker starting at $v$ and teleporting back to $v$ at rate $1-\alpha$. PPR was introduced by Haveliwala~\cite{haveliwala2003topic} for topic-sensitive web search; Andersen, Chung, and Lang~\cite{andersen2006local} proved that sweep cuts on single-seed PPR vectors achieve Cheeger-type guarantees for detecting low-conductance clusters.

\paragraph{Standard PageRank.}
Setting $\mathcal{S} = \{1,\ldots,n\}$ recovers the standard PR vector. The transition kernel is the \emph{Google matrix} $G = \alpha Q + \frac{1-\alpha}{n}\mathbf{1}\mathbf{1}^\tp$~\cite{PageRank1998}. Since every entry of $G$ is at least $(1-\alpha)/n > 0$, $G$ is primitive and ${\pi}$ is the unique positive left eigenvector with ${\pi}^\tp\mathbf{1} = 1$.

For the limiting case $\alpha = 1$, the kernel reduces to $G_{\mathbf{s}} = Q$, and every stationary distribution concentrates on the closed communicating classes of $Q$, assigning zero probability to transient states.


	\section{Toy example for illustrating Theorem~\ref{thm:blockstr}}\label{app:toy-model}
	Here we illustrate zero block structure in an EDMD matrix guaranteed by Theorem~\ref{thm:blockstr} using the following toy model with $\tilde{N} = 9$ observables in the initial dictionary.
	
	\begin{example}\label{ex:toy2d}
		Consider the following discrete-time dynamical system with two variables $x_1$ and $x_2$:
		\begin{equation}\label{eq:toy-model-def}
			\begin{aligned}
				x_1(k+1) & = x_1(k)- \Delta t \cdot a x_1(k), \\ 
				x_2(k+1) & = x_2(k)- \Delta t \cdot b(x_2(k) - x_1(k)^2), 
			\end{aligned}
		\end{equation}
		where $a >0$ and $b >0$, which is a discrete-time approximation of the continuous-time system:
		\begin{equation*}
			\begin{aligned}
				\dot{x}_1 &= -a x_1, \\ 
				\dot{x}_2 &= -b (x_2 - x_1^2).
			\end{aligned}
		\end{equation*}
		This system \eqref{eq:toy-model-def} admits the invariant subspace spanned by $\Dc_3 \coloneqq \{x_1, x_2, x_1^2\}$ because their images under the action by the Koopman operator are contained in the same subspace. In other words, their one-step evolutions at time $k+1$ can be expressed as linear combinations of their values at time $k$ as follows:
		\begin{equation}\label{eq:toy-model-linear}
			\begin{bmatrix}
				(\Koop \psi_1)(\x(k)) \\
				(\Koop \psi_2)(\x(k)) \\
				(\Koop \psi_3)(\x(k))
			\end{bmatrix}^\tp =
			\begin{bmatrix}
				\psi_1(\x(k+1)) \\
				\psi_2(\x(k+1)) \\
				\psi_3(\x(k+1))
			\end{bmatrix}^\tp =
			\begin{bmatrix}
				\psi_1(\x(k)) \\
				\psi_2(\x(k)) \\
				\psi_3(\x(k))
			\end{bmatrix}^\tp
			\begin{bmatrix}
				1 - \Delta t \cdot a  & 0 & 0 \\
				0   & 1 - \Delta t \cdot b &  0 \\
				0  & \Delta t \cdot b  & (1 - \Delta t \cdot a)^2
			\end{bmatrix},
		\end{equation}
		where $\psi_1(\x)=x_1, \psi_2(\x)=x_2$, and $\psi_3(\x)=x_1^2$.
		This $3 \by 3$ matrix form of the Koopman operator is identical to the EDMD matrix obtained with the dictionary $\Dc_3$ as EDMD completely captures the behavior of the Koopman operator on the invariant subspace~\cite{haseli2023generalizing}. Indeed, the EDMD method with $M = 100$ i.i.d. samples from the uniform distribution on the square region $[-2,2]^2$ provides the following matrix:
		\begin{equation}\label{eq:EDMD-mat-3obs}
			\begin{bmatrix}
				0.92 & 0   & 0     \\
				0    & 0.8 & 0     \\
				0    & 0.2 & 0.846
			\end{bmatrix}
		\end{equation}
		which coincides with the matrix form in~\eqref{eq:toy-model-linear} where the parameter values are $\Delta t = 0.2$, $a=0.4$, and $b = 1.0$.
		
		According to Theorem~\ref{thm:blockstr}, EDMD still captures the invariant subspace with a larger dictionary. Here, we perform EDMD with an extended dictionary of polynomials of degree less than 4, $\Dc_9 \coloneqq \{x_1, x_2, x_1^2, x_1x_2, x_2^2, x_1^3, x_1^2x_2, x_1x_2^2, x_2^3\}$. The EDMD matrix is given by
		\begin{equation}\label{eq:EDMD-mat-9obs}
			\begin{blockarray}{cccccccccc}
				& x_1 & x_2 & x_1^2 & x_1x_2 & x_2^2 & x_1^3 & x_1^2x_2 & x_1x_2^2 & x_2^3\\ 
				\begin{block}{c[ccccccccc]}
					x_1\,      &  0.920 &  0.000 &  0.000 &  0.000 &  0.002 &  0.000 &  0.009 & -0.104 & -0.170 \\
					x_2\,      &  0.000 &  0.800 &  0.000 &  0.000 &  0.003 &  0.000 &  0.013 & -0.069 & -0.113 \\
					x_1^2\,    &  0.000 &  0.200 &  0.846 &  0.000 &  0.118 &  0.000 &  0.499 &  0.052 &  0.417 \\
					x_1x_2\,   &  0.000 &  0.000 &  0.000 &  0.736 &  0.003 &  0.000 &  0.012 &  0.722 &  0.018 \\
					x_2^2\,    &  0.000 &  0.000 &  0.000 &  0.000 &  0.627 &  0.000 & -0.057 & -0.050 &  0.379 \\
					x_1^3\,    &  0.000 &  0.000 &  0.000 &  0.184 & -0.003 &  0.779 & -0.015 &  0.147 &  0.089 \\
					x_1^2x_2\, &  0.000 &  0.000 &  0.000 &  0.000 &  0.324 &  0.000 &  0.696 &  0.018 &  0.286 \\
					x_1x_2^2\, &  0.000 &  0.000 &  0.000 &  0.000 &  0.002 &  0.000 &  0.008 &  0.616 & -0.043 \\
					x_2^3\,    &  0.000 &  0.000 &  0.000 &  0.000 & -0.003 &  0.000 & -0.014 & -0.007 &  0.559 \\
				\end{block}
			\end{blockarray}.
		\end{equation}
		
		It is clear that the top-left corner $3\by 3$ submatrix is identical to the EDMD matrix in \eqref{eq:EDMD-mat-3obs} obtained with the dictionary of the three observables, $\Dc_3 = \{x_1, x_2, x_1^2\}$, which spans the Koopman invariant subspace. Also, the bottom-left corner $6 \by 3$ matrix is a zero matrix, $O_{6 \by 3}$, as stated in Theorem~\ref{thm:blockstr}.
		
		There are more zero entries in the matrix \arXivtag[\eqref{eq:EDMD-mat-9obs}] besides the zero block at the bottom-left corner. They are not just numerical artifacts; they are consequences of other invariant subspaces. We can identify other subsets that span other Koopman invariant subspaces by permuting the rows and columns of the EDMD matrix. For instance, an ordered dictionary $\{x_1, x_1x_2, x_1^3, x_2, x_1^2, x_2^2, x_1^2x_2, x_1x_2^2, x_2^3\}$ (yet with the same elements as $\Dc_9$) provides the following EDMD matrix:
		\begin{equation}\label{eq:EDMD-mat-9obs-permuted}
			\begin{blockarray}{cccccccccc}
				& x_1 & x_1x_2 & x_1^3 & x_2 & x_1^2 & x_2^2 & x_1^2x_2 & x_1x_2^2 & x_2^3\\ 
				\begin{block}{c[ccccccccc]}
					x_1\,      &  0.920 &  0.000 &  0.000 &  0.000 &  0.000 &  0.002 &  0.009 & -0.104 & -0.170 \\
					x_1x_2\,   &  0.000 &  0.736 &  0.000 &  0.000 &  0.000 &  0.003 &  0.012 &  0.722 &  0.018 \\
					x_1^3\,    &  0.000 &  0.184 &  0.779 &  0.000 &  0.000 & -0.003 & -0.015 &  0.147 &  0.089 \\
					x_2\,      &  0.000 &  0.000 &  0.000 &  0.800 &  0.000 &  0.003 &  0.013 & -0.069 & -0.113 \\
					x_1^2\,    &  0.000 &  0.000 &  0.000 &  0.200 &  0.846 &  0.118 &  0.499 &  0.052 &  0.417 \\
					x_2^2\,    &  0.000 &  0.000 &  0.000 &  0.000 &  0.000 &  0.627 & -0.057 & -0.050 &  0.379 \\
					x_1^2x_2\, &  0.000 &  0.000 &  0.000 &  0.000 &  0.000 &  0.324 &  0.696 &  0.018 &  0.286 \\
					x_1x_2^2\, &  0.000 &  0.000 &  0.000 &  0.000 &  0.000 &  0.002 &  0.008 &  0.616 & -0.043 \\
					x_2^3\,    &  0.000 &  0.000 &  0.000 &  0.000 &  0.000 & -0.003 & -0.014 & -0.007 &  0.559 \\
				\end{block}
			\end{blockarray}.
		\end{equation}
		Now zero block structures are more noticeable. The $6 \by 3$ submatrix at the bottom-left corner is a zero matrix. Also, the $4 \by 5$ submatrix at the bottom-left corner is a zero matrix. These observations respectively indicate the facts that two dictionaries $\{x_1, x_1x_2, x_1^3\}$ and $\{x_1, x_1x_2, x_1^3, x_2, x_1^2\}$ span Koopman invariant subspaces. This allows the identification of \textit{all} Koopman invariant subspaces spanned by subsets of the dictionary $\Dc_9$.
	\end{example}

\section{Three-state example: exact detection windows}\label{app:3state}

To concretely illustrate the mixing-versus-reachability dichotomy of Theorem~\ref{thm:detection}, we carry out the full computation of the auxiliary gaps $\Delta^{(0)}_{\textnormal{PR}}$, $\Delta^{(0)}_{\textnormal{PPR}}$ and the detection windows $\{\alpha \in (0,1) : \text{Theorem~\ref{thm:detection}(i) or (ii) holds}\}$ on two minimal examples with $\tilde{N} = 3$, $N = 2$. In each example, we specify a row-stochastic matrix $P=\mathcal{R}(|K^\tp|)$ as if we start from $K = P^\tp$, so $r_{\max} = 1$. The matrix $P^{(0)} = \mathcal{R}(|K^{(0)}|^\tp)$ is the row-normalization of $P$ after zeroing $P^{12}$ (equivalently, \arXivtag[the rows of the top block] are renormalized to make $P^{(0)}$ row-stochastic), with $P^{11,(0)} = \mathcal{R}(P^{11})$, $P^{22,(0)} = P^{22}$, $P^{21,(0)} = P^{21}$. Both share the \arXivtag[bottom-block] row $P^{22,(0)} = 1-2\beta$ and $P^{21,(0)} = (\beta, \beta)$ with $\beta \in (0, 1/2]$. 

The leakage parameter is $\epsilon := \|P^{12}\|_\infty \in (0, 1)$. In both, $r^{(0)}_{\min} = 1-\epsilon$. Because the leakage is removed before renormalization, $P^{11,(0)} = \mathcal{R}(P^{11})$ is row-stochastic and the auxiliary gaps $\Delta^{(0)}_{\textnormal{PR}}, \Delta^{(0)}_{\textnormal{PPR}}$ become $\epsilon$-independent: $\epsilon$ enters detection only through $\|P^{12}\|_\infty = \epsilon$. The two examples differ only in the internal block $P^{11,(0)}$: Example~A has a well-mixed internal block, while Example~B has a ``starved'' node with zero column in $P^{11,(0)}$.

Throughout this section we write $\Delta^{(0)}_{\textnormal{PR}}$, $\Delta^{(0)}_{\textnormal{PPR}}$ for the auxiliary gaps defined in \eqref{eq:delta0-defs}.

\subsection{Example A: well-mixed internal block}\label{app:2state-mixing}

\paragraph{Setup.} Take
\begin{equation}\label{eq:exA-K}
	P_A \;=\; \begin{pmatrix} 0 & 1-\epsilon & \epsilon \\ 1-\epsilon & 0 & \epsilon \\ \beta & \beta & 1-2\beta \end{pmatrix}.
\end{equation}
Each row of $P_A$ sums to $1$, so $\|P^{12}\|_\infty = \epsilon$. The top-left block $P^{11}_A = (1-\epsilon)\begin{pmatrix}0&1\\1&0\end{pmatrix}$ has each row summing to $1-\epsilon$, so after zeroing $P^{12}$ and renormalizing \arXivtag[the rows of the top block] we obtain the row-stochastic
	\begin{equation}\label{eq:exA-P0}
		P^{(0)}_A \;=\; \begin{pmatrix} 0 & 1 & 0 \\ 1 & 0 & 0 \\ \beta & \beta & 1-2\beta \end{pmatrix},
		\qquad
		P^{11,(0)}_A \;=\; Q \coloneqq \begin{pmatrix} 0 & 1 \\ 1 & 0 \end{pmatrix}.
	\end{equation}
	The internal block $P^{11,(0)}_A$ is the row-stochastic swap matrix, irreducible with stationary distribution $\mathbf{p} = (1/2, 1/2)$.

\paragraph{Resolvent.} A geometric series gives
	\begin{equation}\label{eq:exA-resolvent}
		R_{11} \;\coloneqq\; (I - \alpha Q)^{-1} \;=\; \frac{1}{1-\alpha^2}\begin{pmatrix} 1 & \alpha \\ \alpha & 1 \end{pmatrix}.
\end{equation}

\paragraph{Standard PR.} With $(\mathbf{q}^{\alpha})^\tp = \tfrac{1-\alpha}{2}\mathbf{1}_2^\tp R_{11}$,
	\[
	q_1^\alpha \;=\; q_2^\alpha \;=\; \frac{(1-\alpha)(1+\alpha)}{2(1-\alpha^2)} \;=\; \frac{1}{2}.
	\]
	So $q_{\min}^\alpha = 1/2$, independent of both $\alpha$ and $\epsilon$. With $\eta = 2\beta$, the mixing condition $q_{\min}^\alpha > (\tilde N - N)(1-\alpha)/(N(1-\alpha+\alpha\eta))$ reduces to
	\begin{equation}\label{eq:exA-mixing}
		\tfrac{1}{2} > \tfrac{1-\alpha}{2(1-\alpha+2\alpha\beta)} \;\Longleftrightarrow\; \alpha\beta > 0,
	\end{equation}
	which holds trivially for any $\alpha, \beta > 0$. Renormalization eliminates the original sub-stochastic-form mixing constraint $\epsilon < 2\beta$.

\paragraph{Exact PR gap.} Since $P^{(0)}_A$ is block lower triangular with the right-top block zero, the resolvent $R \coloneqq (I-\alpha P^{(0)}_A)^{-1}$ is block lower triangular with diagonal blocks $R_{11}$ from \eqref{eq:exA-resolvent} and $R_{22} = (1-\alpha(1-2\beta))^{-1}$, and lower-left block $R_{21} = \alpha R_{22}P^{21,(0)}R_{11}$. Writing ${\pi}^{(0),\tp} = \tfrac{1-\alpha}{3}\mathbf{1}^\tp R$ and computing column sums,
	\[
	\pi^{(0)}_1 \;=\; \pi^{(0)}_2 \;=\; \frac{1-\alpha+3\alpha\beta}{3(1-\alpha+2\alpha\beta)}, \qquad \pi^{(0)}_3 \;=\; \frac{1-\alpha}{3(1-\alpha+2\alpha\beta)},
	\]
	from which
	\begin{equation}\label{eq:exA-PR-gap}
		\Delta^{(0)}_{\textnormal{PR}}(\alpha) \;=\; \pi^{(0)}_2 - \pi^{(0)}_3 \;=\; \frac{\alpha\beta}{1-\alpha+2\alpha\beta}.
	\end{equation}
	$\Delta^{(0)}_{\textnormal{PR}}(\alpha) > 0$ for all $\alpha \in (0,1)$ and $\beta > 0$, monotone increasing in $\alpha$ from $0$ to $\beta/(1+2\beta)$.

\paragraph{PPR with seed $\mathcal{S} = \{1\}$.} Reachability holds ($\{1\} \to \{1,2\}$ via the 2-cycle). From \eqref{eq:exA-resolvent}, $e_1^\tp R_{11} = (1, \alpha)/(1-\alpha^2)$, so
	\[
	\bar\pi^{(0)}_1 \;=\; \frac{1}{1+\alpha}, \qquad \bar\pi^{(0)}_2 \;=\; \frac{\alpha}{1+\alpha}, \qquad \bar\pi^{(0)}_3 \;=\; 0.
	\]
	\arXivtag[The minimum of the PPR masses corresponding to the top block] is $\bar\pi^{(0)}_2$, giving
	\begin{equation}\label{eq:exA-PPR-gap}
		\Delta^{(0)}_{\textnormal{PPR}}(\alpha) \;=\; \frac{\alpha}{1+\alpha}.
	\end{equation}
	Both $\Delta^{(0)}_{\textnormal{PR}}$ and $\Delta^{(0)}_{\textnormal{PPR}}$ are $\epsilon$-independent under renormalization: $\epsilon$ enters detection only through $\|P^{12}\|_\infty = \epsilon$, not through the auxiliary gaps on $P^{(0)}$.

\paragraph{Detection windows.} The sufficient condition from Theorem~\ref{thm:detection}---detection on $P$ succeeds whenever $\|P^{12}\|_\infty < \frac{1-\alpha}{4\alpha}\Delta^{(0)}_{\mathbf{s}}$---specializes here with $\|P^{12}\|_\infty = \epsilon$ to $\epsilon < \frac{1-\alpha}{4\alpha}\Delta^{(0)}_{\mathbf{s}}(\alpha)$. With $\beta = 1/2$ for concreteness, $\Delta^{(0)}_{\textnormal{PR}}(\alpha) = \alpha/2$ and $\Delta^{(0)}_{\textnormal{PPR}}(\alpha) = \alpha/(1+\alpha)$. The PR condition becomes $\epsilon < (1-\alpha)/8$, the PPR condition $\epsilon < (1-\alpha)/(4(1+\alpha))$. Solving for $\alpha$ gives detection windows $(0,\alpha^\star)$ with
	\begin{align}
		\alpha^\star_{\textnormal{PR}}(\epsilon) &\;=\; 1-8\epsilon, \label{eq:exA-PR-star}\\
		\alpha^\star_{\textnormal{PPR}}(\epsilon) &\;=\; \frac{1-4\epsilon}{1+4\epsilon}. \label{eq:exA-PPR-star}
	\end{align}
	Window closure ($\alpha^\star = 0$):
	\begin{equation}\label{eq:exA-epsmax}
		\epsilon^{\max}_{\textnormal{PR}} \;=\; 1/8 \;=\; 0.125, \qquad \epsilon^{\max}_{\textnormal{PPR}} \;=\; 1/4 \;=\; 0.250.
	\end{equation}
	The two curves \eqref{eq:exA-PR-star} and \eqref{eq:exA-PPR-star} are plotted in Panel~\arXivtag[(a)] of Figure~\ref{fig:detection-windows}: the blue curve is $\alpha^\star_{\textnormal{PR}}(\epsilon)$ and the red curve is $\alpha^\star_{\textnormal{PPR}}(\epsilon)$.

\subsection{Example B: starved non-mixing internal block}\label{app:2state-starved}

\paragraph{Setup.} Take
\begin{equation}\label{eq:exB-K}
	P_B \;=\; \begin{pmatrix} 1-\epsilon & 0 & \epsilon \\ 1 & 0 & 0 \\ \beta & \beta & 1-2\beta \end{pmatrix}.
\end{equation}
The top-left block $P^{11}_B = \begin{pmatrix}1-\epsilon&0\\1&0\end{pmatrix}$ has row $1$ summing to $1-\epsilon$ and row $2$ summing to $1$. After zeroing $P^{12}$ and renormalizing \arXivtag[the rows of the top block,] we obtain the row-stochastic
	\begin{equation}\label{eq:exB-P0}
		P^{(0)}_B \;=\; \begin{pmatrix} 1 & 0 & 0 \\ 1 & 0 & 0 \\ \beta & \beta & 1-2\beta \end{pmatrix},
		\qquad
		P^{11,(0)}_B \;=\; \begin{pmatrix} 1 & 0 \\ 1 & 0 \end{pmatrix}.
	\end{equation}
	Column $2$ of $P^{11,(0)}_B$ is zero: node $2$ receives no inflow \arXivtag[within top-left block]---a \emph{starved node} in the sense of the discussion after Theorem~\ref{thm:detection}. Row $1$ is a self-loop, and row $2$ jumps deterministically to node $1$. The starved-node property is preserved under renormalization.

\paragraph{Resolvent.} Direct inversion gives
	\begin{equation}\label{eq:exB-resolvent}
		R_{11} \;\coloneqq\; (I - \alpha P^{11,(0)}_B)^{-1} \;=\; \begin{pmatrix} 1/(1-\alpha) & 0 \\ \alpha/(1-\alpha) & 1 \end{pmatrix}.
\end{equation}

\paragraph{Standard PR.} Column sums of \eqref{eq:exB-resolvent} yield
	\[
	q_1^\alpha \;=\; \frac{1+\alpha}{2}, \qquad q_2^\alpha \;=\; \frac{1-\alpha}{2}.
	\]
	So $q_{\min}^\alpha = q_2^\alpha = (1-\alpha)/2$, the pure-teleportation value, again \emph{independent of $\epsilon$}. The mixing condition becomes $1-\alpha+2\alpha\beta > 1 \Leftrightarrow \beta > 1/2$, $\alpha$- and $\epsilon$-independent: no choice of $\alpha$ or reduction of leakage can rescue mixing once $P^{11,(0)}_B$ has a starved node.

\paragraph{Exact PR gap.} Using \eqref{eq:exB-resolvent} and the block lower triangular structure,
	\[
	\pi^{(0)}_1 = \frac{(1+\alpha)(1-\alpha+3\alpha\beta)}{3(1-\alpha+2\alpha\beta)}, \quad \pi^{(0)}_2 = \frac{(1-\alpha)(1-\alpha+3\alpha\beta)}{3(1-\alpha+2\alpha\beta)}, \quad \pi^{(0)}_3 = \frac{1-\alpha}{3(1-\alpha+2\alpha\beta)}.
	\]
	Since $\pi^{(0)}_1/\pi^{(0)}_2 = (1+\alpha)/(1-\alpha) \geq 1$, \arXivtag[the minimum of the first two PR masses] is at the starved node $2$, and the gap is
	\begin{equation}\label{eq:exB-PR-gap}
		\Delta^{(0)}_{\textnormal{PR}}(\alpha) \;=\; \pi^{(0)}_2 - \pi^{(0)}_3 \;=\; \frac{\alpha(1-\alpha)(3\beta-1)}{3(1-\alpha+2\alpha\beta)}.
	\end{equation}
	Positive only if $\beta > 1/3$. Interior max in $\alpha$, vanishing at both endpoints with the $(1-\alpha)$ prefactor.

\paragraph{PPR with seed $\mathcal{S} = \{2\}$.} Seed $\{1\}$ fails reachability (row $1$ of $P^{11,(0)}_B$ is a self-loop), but seed $\{2\}$ reaches $\{1,2\}$. From \eqref{eq:exB-resolvent}, $e_2^\tp R_{11} = (\alpha/(1-\alpha), 1)$, giving
	\[
	\bar\pi^{(0)}_1 \;=\; \alpha, \qquad \bar\pi^{(0)}_2 \;=\; 1-\alpha, \qquad \bar\pi^{(0)}_3 \;=\; 0.
	\]
	The \arXivtag[minimum of $\bar\pi^{(0)}_1$ and $\bar\pi^{(0)}_2$] crosses over at $\alpha_c \coloneqq 1/2$: $\bar\pi^{(0)}_1 \lessgtr \bar\pi^{(0)}_2 \Leftrightarrow \alpha \lessgtr 1/2$. So
	\begin{equation}\label{eq:exB-PPR-gap}
		\Delta^{(0)}_{\textnormal{PPR}}(\alpha) \;=\; \begin{cases} \alpha & \alpha \leq 1/2,\\[4pt] 1-\alpha & \alpha \geq 1/2, \end{cases}
	\end{equation}
	continuous at $\alpha_c$ with value $1/2$ and \emph{$\epsilon$-independent}. Note that $\Delta^{(0)}_{\textnormal{PPR}} > 0$ on all of $(0,1)$ --- reachability delivers a positive gap even though PR's mixing condition is uniformly violated.

\paragraph{Detection windows.} For $\beta = 1/2$, \eqref{eq:exB-PR-gap} simplifies to $\Delta^{(0)}_{\textnormal{PR}}(\alpha) = \alpha(1-\alpha)/6$, so the PR detection condition $\epsilon < (1-\alpha)\Delta^{(0)}_{\textnormal{PR}}/(4\alpha) = (1-\alpha)^2/24$ inverts to
	\begin{equation}\label{eq:exB-PR-star}
		\alpha^\star_{\textnormal{PR}}(\epsilon) \;=\; \max\!\bigl(0,\; 1 - \sqrt{24\epsilon}\bigr), \qquad \epsilon^{\max}_{\textnormal{PR}} = 1/24 \approx 0.0417.
	\end{equation}
	For PPR, the detection condition $\epsilon < (1-\alpha)\Delta^{(0)}_{\textnormal{PPR}}(\alpha)/(4\alpha)$ on each branch is solved separately. On the lower branch ($\alpha \leq 1/2$, $\Delta^{(0)}_{\textnormal{PPR}} = \alpha$), it becomes $\epsilon < (1-\alpha)/4$, i.e.\ $\alpha < 1-4\epsilon$. On the upper branch ($\alpha \geq 1/2$, $\Delta^{(0)}_{\textnormal{PPR}} = 1-\alpha$), it becomes $\epsilon < (1-\alpha)^2/(4\alpha)$, which solves to $\alpha < (1+2\epsilon) - 2\sqrt{\epsilon(1+\epsilon)}$. The two branches meet at $\alpha = 1/2$, corresponding to $\epsilon = 1/8$, so
	\begin{equation}\label{eq:exB-PPR-star}
		\alpha^\star_{\textnormal{PPR}}(\epsilon) \;=\; \begin{cases} (1+2\epsilon) - 2\sqrt{\epsilon(1+\epsilon)} & \epsilon \leq 1/8,\\[4pt] 1-4\epsilon & \epsilon \geq 1/8, \end{cases}
	\end{equation}
	with window closure $\epsilon^{\max}_{\textnormal{PPR}} = 1/4$, attained as $\alpha \to 0$ on the lower branch. The two curves are plotted in Panel~\arXivtag[(b)] of Figure~\ref{fig:detection-windows}, with the branch change at $\epsilon = 1/8$ visible as a slight kink.

\paragraph{Takeaway.} PR tolerates roughly $2\times$ less leakage than PPR in the well-mixed Example~A ($1/8$ vs $1/4$) and $6\times$ less in the starved Example~B ($1/24$ vs $1/4$). In Example~B, PR's detection window collapses by $\alpha \to 0$ at any $\epsilon > 1/24$, while PPR stays open up to $\epsilon = 1/4$; this is the mixing-versus-reachability separation in its starkest form.

\subsection{Remarks on the two examples}\label{app:2state-discuss}

(a) \emph{Role of $\alpha$.} In both examples and both methods, the detection window is an interval $(0,\alpha^\star)$ opening at $\alpha = 0$; no interior-$\alpha$ tuning is required. The $1/\alpha$ factor in the perturbation amplifier $2\alpha/((1-\alpha)r^{(0)}_{\min})$ dominates any interior maximum of the baseline gap, so small $\alpha$ always yields the widest tolerance to leakage. This is consistent with the remark in Section~\ref{subsec:block-theory} that PPR ``does not require $\alpha$ near $1$\arXivtag[.]''

(b) \emph{What the starved example rules out.} Example~B shows that PR's mixing condition is not merely a sufficient abstraction but a genuine obstruction: once $P^{11,(0)}$ has a starved node, $q_{\min}^\alpha$ collapses to $(1-\alpha)/N$ for every $\alpha$, and the mixing condition becomes a condition on $\beta$ alone, independent of both $\alpha$ and $\epsilon$. In contrast, PPR's reachability condition is structural and passes trivially when any seed $v$ has $P^{11,(0)}_{vi} \neq 0$ cumulatively reaching all \arXivtag[$i \in \{1,\ldots, N\}$].

(c) \emph{Size of the PPR advantage.} In the well-mixed example, PPR's maximum tolerable leakage is $\epsilon^{\max}_{\textnormal{PPR}}/\epsilon^{\max}_{\textnormal{PR}} = (1/4)/(1/8) = 2\times$ larger than PR's. In the starved example, the ratio is $(1/4)/(1/24) = 6\times$. The reachability-over-mixing advantage compounds the $\alpha$-decoupling advantage from zero leakage when internal dynamics are non-mixing.

\section{Experimental details}\label{app:experiments}

This appendix provides full experimental setup, dictionary construction, and additional figures for the experiments in Section~\ref{sec:examples}. All relevant codes can be found at \texttt{\anonrepo}. For the Duffing and Van der Pol oscillators, we used bivariate Laguerre polynomials, which are not bounded on $\R^2$, while our finite-sample error bound (Proposition~\ref{prop:finite-sample-edmd}) assumes boundedness of observables (Assumption~\ref{assume:bounded-dict}). Nevertheless, the data distribution $\mu$ is compactly supported on $[-2,2]^2$, so any fixed polynomial is bounded on $\textnormal{supp}(\mu)$; the finite-sample bounds of Theorem~\ref{thm:end-to-end} apply with $D = \sup_{\x \in [-2,2]^2} |\psi_i(\x)|$. Also, we always include the first two Laguerre polynomials, $1-x$ and $1-y$, in sub-dictionaries for the two oscillator examples in order to reconstruct the state variables, $x$ and $y$, and measure state prediction errors. 
For the three-well Ramachandran potential and Lorenz system, we chose the set of state coordinates as the set of seed nodes for PPR.

	\subsection{Quantifying one-step/multi-step state prediction errors}
	For the two oscillator examples, the one-step state prediction error (Figs.~\ref{fig:Duffing-pred-error}a and~\ref{fig:vdp-pred-error}a) is computed as follows. The state variables $x$ and $y$ are recovered from the first two Laguerre basis functions $\psi_1(\mathbf{x}) = 1 - x$ and $\psi_2(\mathbf{x}) = 1 - y$, which are always included in every sub-dictionary.
	Given a sub-dictionary $\Dc_N$ of size $N$, we fit the EDMD matrix $K_{N,M}$ and measure the one-step prediction error by
	\begin{equation}
		\mathcal{E}(N) = \left( \sum_{j=1}^{2} \left\| \Koop\psi_j - \Koop_{N,M}\psi_j \right\|_{L^2(\mu_{\text{test}})}^2 \right)^{1/2},
	\end{equation}
	where $\Koop\psi_j$ is approximated by evaluating $\psi_j$ at the true next state $\mathbf{F}(\mathbf{x})$ on the test data, and $\Koop_{N,M}\psi_j$ is the prediction obtained by applying $\Koop_{N,M}$ to the observable vector at the current state. 
	Note that $\Koop_{N,M}$ is the operator naturally associated with the EDMD matrix $K_{N,M}$, i.e., $\Koop_{N,M}\psi_i \coloneqq \sum_{j=1}^N K_{N,M}[j,i]\psi_j$. 
	Here, $\|\cdot\|_{L^2(\mu_{\text{test}})}$ denotes the empirical $L^2$ norm over $M_{\text{test}} = 2000 $ i.i.d.\ test points drawn from the same distribution as the training data.
	Since the state coordinates are affine functions of $\psi_1$ and $\psi_2$, this error directly measures how accurately the sub-dictionary EDMD matrix propagates the physical state one step forward.
	
	For the three-well Ramachandran potential, the system evolves on the torus $[-\pi,\pi]^2$, so the state variables $\phi$ and $\psi$ are not globally recoverable from a single linear observable. Instead, we embed the state via the four circular coordinates $\textnormal{Coord} \coloneqq \arXivtag[\{\sin\phi,\cos\phi,\sin\psi,\cos\psi\}]$, which are always included in every sub-dictionary and serve as the seed set for PPR. The one-step state prediction error is then
	\begin{equation}
		\mathcal{E}(N) = \left( \sum_{f \in \textnormal{Coord}} \left\| \Koop f - \Koop_{N,M} f \right\|_{L^2(\mu_{\textnormal{test}})}^2 \right)^{1/2}.
	\end{equation}
	Since $(\sin\phi,\cos\phi,\sin\psi,\cos\psi)$ uniquely determines the state on the torus, this error faithfully measures one-step prediction accuracy for the physical state (Fig.~\ref{fig:ala2-errors}a).
	The same error quantification extends naturally to multi-step prediction (Fig.~\ref{fig:ala2-errors}b): starting from a test point, we apply $\Koop$ and $\Koop_{N,M}$ repeatedly to propagate the full observable vector forward, reading off the four circular coordinates at each step and comparing against the true trajectory.
	
	\subsection{Compared methods}\label{app:comp-methods}
	
	For the two oscillator examples, we compare four dictionary ordering methods alongside the tunable symmetric subspace decomposition (TSSD) baseline. \textbf{Unordered} uses a uniformly random permutation of all $\tilde{N}$ observables as a baseline, so that the first $N$ selected observables carry no structural information\arXivtag[, except that the state variables are always included in the sub-dictionary.] \textbf{Incremental} selects the first $N$ observables in lexicographic order of total degree, i.e., the natural ordering of the Laguerre polynomial basis from lowest to highest degree. \textbf{PageRank (PR)} and \textbf{Personalized PageRank (PPR)} apply Algorithm~\ref{algor:PR-EDMD} to the full EDMD matrix $K_{\tilde{N},M}$: the $\tilde{N}$ observables are ranked by their PR or PPR scores, and the top-$N$ are selected to form the sub-dictionary $\Dc_N$. For PPR, the seed set is taken to be the state-coordinate observables. 
	
	\textbf{TSSD}~\cite{haseli2021learning} iteratively prunes the dictionary by retaining only directions that are approximately invariant under both the forward and backward EDMD operators; the tolerance parameter $\epsilon \in [0,1]$ controls how strictly invariance is enforced, with $\epsilon = 0$ recovering exact symmetric subspace decomposition (SSD). We sweep $\epsilon$ over a uniform grid of $[0,1]$ with spacing $0.01$ and report each detected sub-dictionary size as a scatter point in the error plots. Despite the fine grid, TSSD identifies only a small number of distinct sizes: three for the Duffing oscillator ($N = 3, 45, 65$; Fig.~\ref{fig:Duffing-pred-error}a) and seven for the Van der Pol oscillator ($N = 3, 24, 25, 26, 27, 28, 51$; Fig.~\ref{fig:vdp-pred-error}a), reflecting the coarse granularity of the invariance structure detected by the algorithm.

\subsection{Duffing oscillator}\label{app:duffing}

The Duffing oscillator is
\begin{equation}\label{eq:Duffing-model}
	\dot x = y, \quad \dot y =  -\delta y + \gamma x - \beta x^3,
\end{equation}
with $\delta = 0.3$, $\gamma = 1$, $\beta = 1$, discretized at $\Delta t = 0.1$. \arXivtag[We use the] dictionary \arXivtag[consisting] of 90 bivariate Laguerre polynomials of order up to 12\arXivtag with $M = 2000$ data points, which are i.i.d. samples from the \arXivtag[two-dimensional] uniform distribution on $[-2,2]^2$. Fig.~\ref{fig:Duffing-heatmap} shows the EDMD matrix before and after PR-based reordering, with entries rescaled via $a_{ij} \mapsto -1+2/(1+\exp(-a_{ij}))$. \arXivtag
For the one-step prediction error (Fig.~\ref{fig:Duffing-pred-error}a), we use 20 random seeds and depict mean$\pm$1SD for each curve. \arXivtag
For the multi-step state prediction (Fig.~\ref{fig:Duffing-pred-error}b), we use one representative among the 20 random seeds. Nevertheless, the overall tendency remains the same regardless of the random seed value, as evidenced by the small deviation in the one-step prediction error plot in Panel~(a).

\begin{figure}[tbh]
	\centering
	\includegraphics[width=1\linewidth]{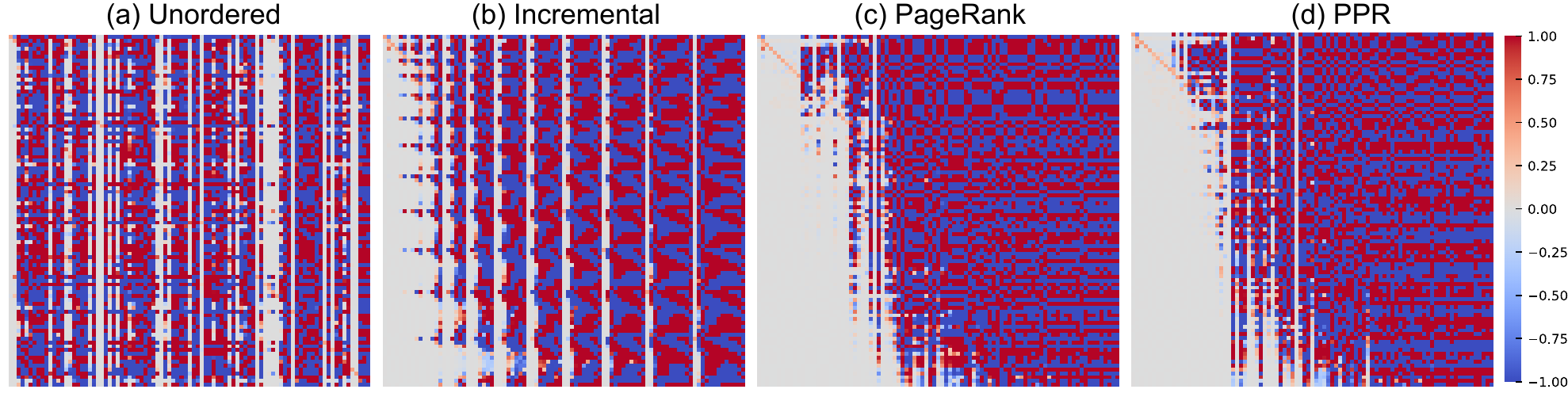}
	\caption{EDMD matrices for the Duffing oscillator with various permutations. (a) Randomly permuted. (b) Lexicographic order. (c) PR-ordered. (d) PPR-ordered.}
	\label{fig:Duffing-heatmap}
\end{figure}

\subsection{Van der Pol oscillator}\label{app:vdp}
\arXivtag[For the] Van der Pol oscillator $\dot x = y$, $\dot y = \mu(1-x^2)y - x$ with $\mu = 1.1$, discretized at $\Delta t = 0.1$, \arXivtag[we use] the same 90 Laguerre polynomials and $M = 2000$ data points, which are i.i.d. samples from the two dimensional uniform distribution on $[-2,2]^2$. \arXivtag[Figs.~\ref{fig:vdp-pred-error} and \ref{fig:vdp-heatmap}] show \arXivtag[the EDMD prediction errors and heatmaps] for the state $(x,y)$. \arXivtag
As done in the Duffing oscillator example above, we use 20 random seeds and depict mean$\pm$1SD for each curve, for the one-step prediction error (Fig.~\ref{fig:vdp-pred-error}a). 
For the multi-step state prediction (Fig.~\ref{fig:vdp-pred-error}b), we use one representative among the 20 random seeds. Nevertheless, the overall tendency remains the same regardless of the random seed value, as evidenced by the small deviation in the one-step prediction error plot in Panel~(a). 

\begin{figure}[tbh]
	\centering
	\includegraphics[width=0.7\linewidth]{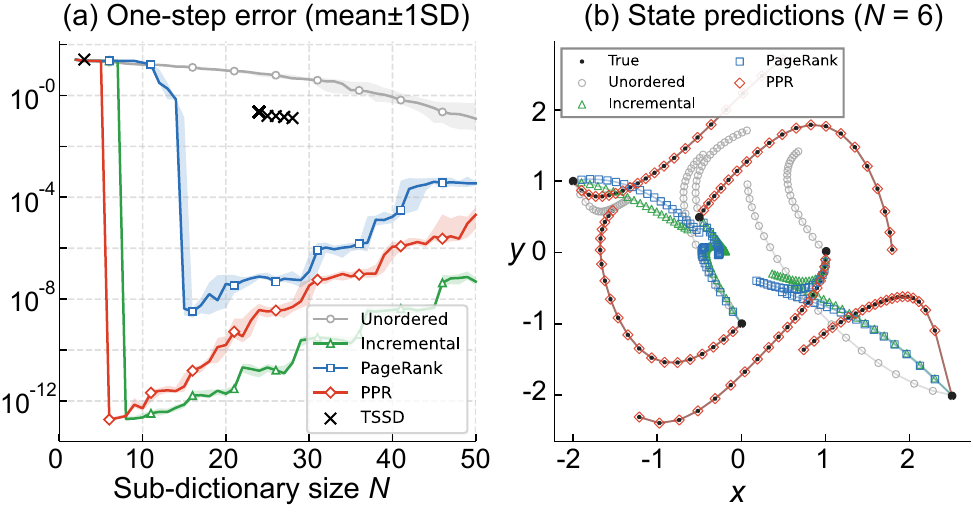}
	\caption{Van der Pol oscillator. (a) One-step errors. The curves represent mean$\pm$1SD over 20 random seeds. 
		The tunable symmetric subspace decomposition (TSSD) detects only 7 sizes: 3, 24, 25, 26, 27, 28 and 51. (b) 20-step state predictions with 6 observables.}
	\label{fig:vdp-pred-error}
\end{figure}

\begin{figure}[tbh]
	\centering
	\includegraphics[width=1\linewidth]{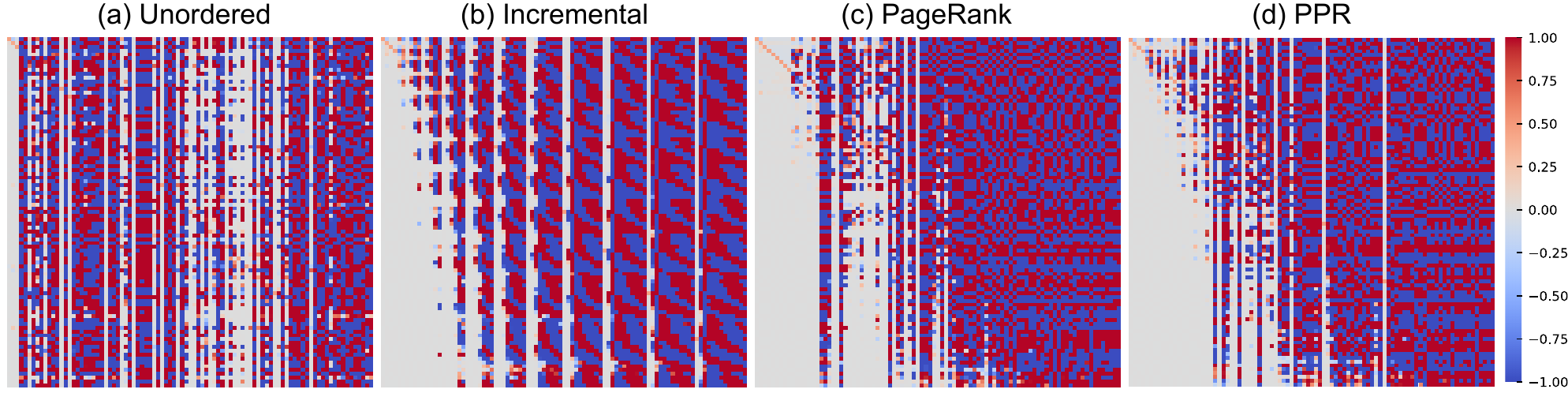}
	\caption{EDMD matrices for the Van der Pol oscillator with various permutations. (a) Randomly permuted. (b) Lexicographic order. (c) PR-ordered. (d) PPR-ordered.}
	\label{fig:vdp-heatmap}
\end{figure}

\subsection{Three-well Ramachandran potential}\label{app:ramachandran}

\paragraph{Potential and dynamics.}
The potential on the torus $[-\pi,\pi]^2$ is
\begin{equation}\label{eq:ramachandran-potential}
	V(\phi,\psi) = \sum_{k=1}^3 A_k \exp\!\left(-\frac{d(\phi,\phi_k^0)^2 + d(\psi,\psi_k^0)^2}{2\sigma_k^2}\right) + 0.3\left(\cos 2\phi + \cos 2\psi\right),
\end{equation}
with geodesic distance $d(\cdot,\cdot)$ on the circle, well centers $(\phi_k^0,\psi_k^0) = (-1.0,-1.0),\, (-1.0,1.2),\, (1.1,-0.3)$, depths $A_k = -6, -6, -4$, widths $\sigma_k = 0.55, 0.55, 0.65$. The dynamics follow the overdamped Langevin equation $d\phi = -\partial_\phi V\,dt + \sqrt{2\beta^{-1}}\,dW_\phi$, $d\psi = -\partial_\psi V\,dt + \sqrt{2\beta^{-1}}\,dW_\psi$ at $\beta = 1.0$, discretized with $\Delta t = 0.005$, producing $M = 100{,}000$ frames. In other words, with 20 different random seeds, we generated 20 stochastic trajectories of the length $100{,}000$. 

\paragraph{Dictionary ($\tilde{N} = 236$).}
(i) state coordinates $\sin\phi, \cos\phi, \sin\psi, \cos\psi$ (4 terms, placed first); (ii) Fourier: $\sin(k\phi)$, $\cos(k\phi)$, $\sin(l\psi)$, $\cos(l\psi)$ for $k,l = 2,\ldots,8$ (28 terms); (iii) cross terms: $\sin(k\phi)\cos(l\psi)$, etc., for $k,l = 1,\ldots,4$ (64 terms); (iv) diagonal Fourier: $\sin(k\phi + l\psi)$, $\cos(k\phi + l\psi)$ for $|k|+|l| \leq 4$ (40 terms); (v) Gaussian RBFs on a $10 \times 10$ grid with $\sigma_{\text{RBF}} = 0.45$ (100 terms).

\paragraph{Comparison methods.}
(a) PR ($\alpha = 0.85$); (b) PPR ($\alpha = 0.85$, multi-seed at the 4 state coordinates); (c) TICA at lag 1, ranking by maximum absolute loading on top-10 slow components; (d) PCCA+ with 50 microstates (k-means), lag 5, 3 macrostates, ranking by Fisher discriminant; (e) Random. All methods rank the same $\tilde{N} = 236$ dictionary; EDMD is recomputed on the selected $N$ observables via least squares. Error is measured on a 20\% held-out test set.

\paragraph{Additional figures.}
Fig.~\ref{fig:ala2-heatmap} shows the EDMD matrix under different orderings: both PR and PPR reveal approximate zero block structure invisible under random, TICA, PCCA+ ordering. These heatmaps show log absolute values for visibility. 
Fig.~\ref{fig:ala2-landscape}a shows that PPR selects all 4 state coordinates at ranks 1--4 followed by basin-boundary RBFs (white dots). Such basin-boundary RBFs seem to effectively capture inter-basin transitions, yielding accurate predictions with smaller sub-dictionaries. Fig.~\ref{fig:ala2-landscape}b,c shows the score distributions: PPR exhibits a sharp two-tier structure (state coordinates receiving $\approx$17\% of total mass, with a factor-of-8 drop to the next tier). Compared with the randomly permuted sub-dictionaries, PPR yielded the largest improvement (Fig.~\ref{fig:ala2-landscape}d). 

\begin{figure}[tbh]
	\centering
	\includegraphics[width=1.0\linewidth]{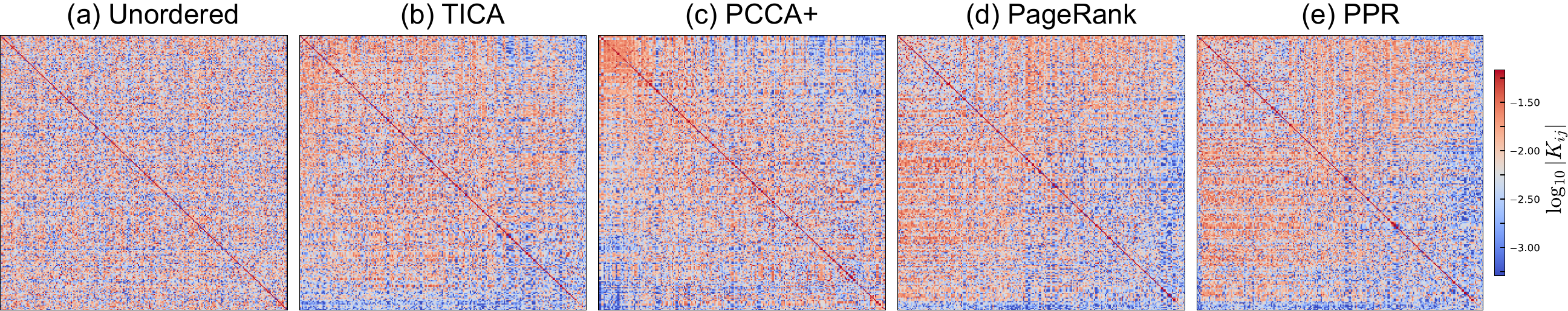}
	\caption{EDMD matrices for the Ramachandran potential with various permutations. (a) Random. (b) TICA. (c) PCCA+ (d) PR. (e) PPR.}
	\label{fig:ala2-heatmap}
\end{figure}

\begin{figure}[tbh]
	\centering
	\includegraphics[width=1.0\linewidth]{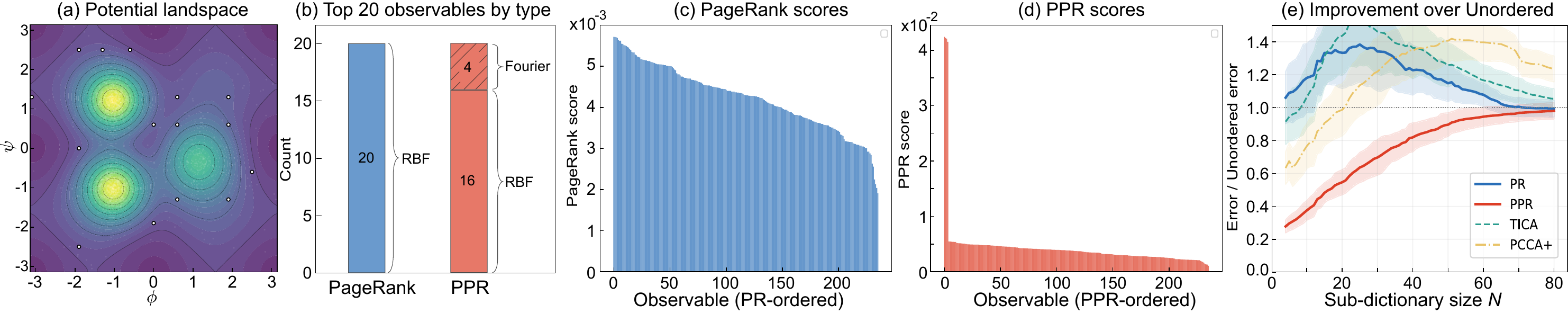}
	\caption{Three-well Ramachandran potential with one representative experiment for illustration. (a) Potential energy landscape. White dots represent the centers of top 15 RBFs detected by PPR. 
		(b) Types of top 20 observables chosen by PageRank and PPR. 
		(c) Standard PR scores. (d) PPR scores: sharp two-tier structure. (e) Improvement ratio vs.\ $N$.}
	\label{fig:ala2-landscape}
\end{figure}

\subsection{Lorenz system}\label{app:lorenz}

\arXivtag[We generate $3 \times 10^5$ data points from a single initial condition at $\Delta t = 0.001$, discard the first $10^5$ as burn-in, and use the remaining $M=2\times 10^5$ data points for EDMD.] The dictionary is a time-delay embedding of the three-dimensional state $(x,y,z)$ with $N_{\textnormal{delay}} = 100$ delay levels and stride $h = 100$ fine steps, yielding $\tilde{N} = 300$ observables and $M = 2 \times 10^5$ snapshot pairs. The observable matrix is 
$$
\mathbf{\Psi}(\mathbf{X}) = \begin{pmatrix}
	\mathbf{x}(0) & \mathbf{x}(h\Delta t) & \cdots & \mathbf{x}((N_{\text{delay}}\!-\!1)h \Delta t) \\
	\mathbf{x}(\Delta t) & \mathbf{x}(\Delta t + h\Delta t) & \cdots & \mathbf{x}(\Delta t + (N_{\text{delay}}\!-\!1)h \Delta t) \\
	\vdots & \vdots & & \vdots \\
	\mathbf{x}((M\!-\!1)\Delta t)& \cdots & \cdots & \mathbf{x}((M\!-\!1)\Delta t 
	+ (N_{\text{delay}}\!-\!1)h \Delta t)
\end{pmatrix} \in \R^{M \by \tilde{N}}.
$$
For the Incremental, PR, and PPR orderings, we reconstruct the EDMD with $N_r \in \{30, 60, 120, 300\}$ and identify the pseudo-eigenfunction at frequency $\log(\lambda)/(i\Delta t) \approx 6\,\textnormal{rad/s}$, consistent with prior work~\cite{arbabi2017ergodic, susuki2015prony, brunton2017chaos}. 

We note that a direct one-step prediction error comparison analogous to Figs.~\ref{fig:Duffing-pred-error} and~\ref{fig:vdp-pred-error} is not presented for the Lorenz system. This is because the time-delay embedding dictionary directly contains the state coordinates as its first-level observables, so the recomputed EDMD matrix on any sub-dictionary that includes the first delay level already achieves accurate one-step state prediction. The eigenfunction comparison of Fig.~\ref{fig:lor-eig} therefore provides a more discriminating quality measure for this dictionary class, probing whether the selected sub-dictionary recovers the correct spectral geometry of the Koopman operator, which is a strictly harder task than one-step prediction accuracy alone.

\subsection{Computational resources}\label{app:compute}

All experiments were conducted on a standard laptop computer equipped with an Apple M5 CPU and 32\,GB of RAM. No GPU acceleration was used. All code is implemented in Python, using \texttt{numpy} for linear algebra and \texttt{numba} for JIT-compiled simulation of the Ramachandran potential; the overdamped Langevin time-stepping loop is compiled to native machine code via LLVM at the first call, with a short warmup run performed before production to amortize the compilation cost. \arXivtag
The total compute time for all experiments reported in the paper is estimated at approximately 30--40 minutes on the hardware described above.

\section{Proofs}\label{sec:proofs}

This appendix contains all proofs. We organize them following the logical structure of the paper: error bounds and zero block structure (Appendix~\ref{sec:proofs-error}), the coupling bound and unconditional leakage guarantee (Appendix~\ref{sec:proofs-leakage}), the PR and PPR lemmas on block-structured matrices (Appendix~\ref{sec:detection_proofs}), and finite-sample extensions (Appendix~\ref{sec:proofs-finite}).

\subsection{Error bounds and zero block structure}\label{sec:proofs-error}

\emph{Goal.} This subsection proves the three results that justify reading sub-dictionary quality off the off-diagonal block of the EDMD matrix: Proposition~\ref{prop:error-bound} (intermediate prediction-error bound in the infinite-dimensional Koopman matrix), Proposition~\ref{prop:error-block} (computable surrogate for the infinite tail), and Theorem~\ref{thm:blockstr} (the ideal case of a Koopman invariant sub-dictionary).

\emph{Strategy.} We proceed in two stages. Proposition~\ref{prop:error-bound} decomposes the EDMD prediction error into an approximation term (projection error within $\Fc_N$) and a statistical term (finite-sample EDMD error); the orthonormal basis structure and Parseval's identity turn the approximation term into $\|K^{21,\text{true}}\|_F^2$. Proposition~\ref{prop:error-block} then replaces this infinite-dimensional quantity with the computable finite block $\|K_{\tilde{N},M}^{21}\|_F^2$ plus a truncation error, so the bound becomes evaluable from data. Finally, Theorem~\ref{thm:blockstr} shows the ideal case: when the sub-dictionary spans a Koopman invariant subspace, the off-diagonal block vanishes identically; the proof exploits the Schur complement structure of the augmented Gram matrix.

The one-step evolution of $\psi_i$ decomposes as in \eqref{eq:psi_evolve_true}. The prediction error is first bounded in terms of the infinite-dimensional Koopman matrix:

\begin{proposition}\label{prop:error-bound}
	Under the same assumptions as in Proposition~\ref{prop:error-block}, for each $\varepsilon > 0$ there exists $M > 0$ such that:
	$\sum_{i=1}^N \|\Koop\psi_i - \Koop_{N,M}\psi_i\|_{L^2}^2 \leq 2\|K^{21, \arXivtag[\textnormal{true}]}\|_F^2 + \varepsilon$.
\end{proposition}

Since $K^{21, \text{true}}$ is not computable from data, Proposition~\ref{prop:error-block} reformulates the bound in terms of the finite EDMD matrix: the first term $\|K_{\tilde{N},M}^{21}\|_F^2$ is then computable from data, and the second is a truncation error decreasing in $\tilde{N}$.

\begin{proof}[Proof of Proposition~\ref{prop:error-bound}]
	We first derive the error term for a single $\psi_i$ and then the one for the summation over $i=1, \ldots, N$. For any $i \in \{1,\ldots,N\}$, we have
	\begin{equation}\label{eq:proof-single-i}
		\norm{\Koop \psi_i - \Koop_{N,M}\,\psi_i}_{L^2}^2 
		\leq
		2\left(\norm{\Koop \psi_i - \Koop_{N}\,\psi_i}_{L^2}^2  + \norm{\Koop_{N} \psi_i - \Koop_{N,M}\,\psi_i}_{L^2}^2\right),
	\end{equation}
	where $\Koop_N$ denotes $\Pc_{\Fc_N}\Koop_{|\Fc_N}$ and $\Pc_{\Fc_N} : \Fc \to \Fc_N$ is the projection operator onto $\Fc_N$. Note that this projection is well defined as $\Fc_N$ is a finite-dimensional subspace of $\Fc = L^2(\Mc, \mu)$ and thus a closed subspace. 
	Under Assumption~\ref{assume:full-rank},
	Theorem 2 in \cite{korda2018convergence} implies that the following holds:
	\begin{equation*}
		\lim_{M\to\infty}\norm{\Koop_{N} \psi_i - \Koop_{N,M}\,\psi_i}_{L^2} = 0.
	\end{equation*}
	
	Therefore, there exists $M>0$ such that 
	\begin{equation*}
		\norm{\Koop_{N} \psi_i - \Koop_{N,M}\,\psi_i}_{L^2} < \sqrt{\frac{\varepsilon}{2N}},
	\end{equation*}
	and thus 
	\begin{equation}\label{eq:proof-eps-N}
		\norm{\Koop_{N} \psi_i - \Koop_{N,M}\,\psi_i}_{L^2}^2 < \frac{\varepsilon}{2N}.
	\end{equation}
	From \eqref{eq:psi_evolve_true}, the term, $\norm{\Koop \psi_i - \Koop_{N}\,\psi_i}_{L^2}^2$, in \eqref{eq:proof-single-i} can be expressed as follows:
	\begin{align*}
		\norm{\Koop \psi_i - \Koop_{N}\,\psi_i}_{L^2}^2 & = \norm{\mathbf{\Psi}\,K^{\text{true}}\mathbf{e}_i - \mathbf{\Psi}_N\,K^{11,\text{true}}\mathbf{e}_i}_{L^2}^2 \\
		& = \norm{\mathbf{\Psi}_{N^\mathsf{c}}\,K^{21,\text{true}}\mathbf{e}_i}_{L^2}^2,
	\end{align*}
	which is identical to
	$\norm{K^{21,\text{true}}\mathbf{e}_i}_{\ell^2}^2$
	by Parseval's identity for the orthonormal set.
	
	Combining this with \eqref{eq:proof-single-i} and \eqref{eq:proof-eps-N}, we have the following inequality for each $i \in \{1,\ldots, N\}$: 
	\begin{align*}
		\norm{\Koop \psi_i - \Koop_{N,M}\,\psi_i}_{L^2}^2 
		&  \leq 
		2\left(\norm{\Koop \psi_i - \Koop_{N}\,\psi_i}_{L^2}^2  + \norm{\Koop_{N} \psi_i - \Koop_{N,M}\,\psi_i}_{L^2}^2\right)
		\\
		& \leq 2 \left( \norm{K^{21,\text{true}}\mathbf{e}_i}_{\ell^2}^2 + \frac{\varepsilon}{2N}\right).
	\end{align*}
	
	Taking the sum over $i=1,\ldots,N$, we obtain:
	\begin{equation*}
		\sum_{i=1}^N \norm{\Koop \psi_i - \Koop_{N,M}\,\psi_i}_{L^2}^2
		\leq 2\norm{K^{21,\text{true}}}_{F}^2+ \varepsilon,
	\end{equation*}
	which completes the proof. 
\end{proof}

The next step splits $K^{21,\text{true}}$ (which has infinitely many rows) into a finite computable part $K_{\tilde{N},M}^{21}$ and a tail that decreases with $\tilde{N}$, then uses EDMD convergence to relate the finite part to data.

\begin{proof}[Proof of Proposition~\ref{prop:error-block}]
	Given Proposition~\ref{prop:error-bound}, it is sufficient to show
	\begin{equation*}
		\norm{K^{21,\text{true}}}_{F}^2
		\leq
		\norm{ K_{\tilde{N},M}^{21} }_F^2 + \norm{ K^{\text{true}}[(\tilde{N}+1):\infty,\, 1:\tilde{N}] }_F^2 + \frac{\varepsilon}{2}
	\end{equation*}
	for a large $M>0$. By definition,
	\begin{equation*}
		\norm{K^{21,\text{true}}}_{F}^2 = \norm{ K^{\text{true}}[(N+1):\tilde{N},\, 1:N] }_F^2
		+
		\norm{ K^{\text{true}}[(\tilde{N}+1):\infty,\, 1:N] }_F^2,
	\end{equation*}
	and
	\begin{equation*}
		\norm{ K^{\text{true}}[(\tilde{N}+1):\infty,\, 1:N] }_F^2
		\leq
		\norm{ K^{\text{true}}[(\tilde{N}+1):\infty,\, 1:\tilde{N}] }_F^2
	\end{equation*}
	since the left-hand side sums over columns $j = 1,\ldots,N$ while the right-hand side sums over the larger set of columns $j = 1,\ldots,\tilde{N}$. Moreover,
	\begin{equation*}
		\norm{ K_{\tilde{N},M}^{21} }_F^2 \to \norm{ K^{\text{true}}[(N+1):\tilde{N},\, 1:N] }_F^2 \quad \text{as} \quad M \to \infty
	\end{equation*}
	since $\Koop_{\tilde{N},M}$ converges to the projection of $\Koop$ onto $\Fc_{\tilde{N}}$ (see Theorem 2 in \cite{korda2018convergence}) and the matrix $K_{\tilde{N},M}^{21}$ contains only finitely many terms.
	Therefore, there exists $M > 0$ such that
	\begin{equation*}
		\norm{ K^{\text{true}}[(N+1):\tilde{N},\, 1:N] }_F^2
		\leq
		\norm{ K_{\tilde{N},M}^{21} }_F^2  + \frac{\varepsilon}{2},
	\end{equation*}
	which completes the proof.
\end{proof}

Theorem~\ref{thm:blockstr} is the exact-case companion: when $\Fc_N$ is $\Koop$-invariant, the off-diagonal block vanishes identically. The proof uses the block structure of the augmented evaluation matrices $\mathbf{\Psi}_{\tilde{N}}(\X),\mathbf{\Psi}_{\tilde{N}}(\Y) \in \mathbb{C}^{M\times\tilde N}$, combined with the Schur complement formula for block matrix inversion. The crucial step is substituting the invariance relation $\mathbf{\Psi}_N(\X)\,K_{N,M} = \mathbf{\Psi}_N(\Y)$ (which holds exactly when $\Fc_N$ is $\Koop$-invariant) into the block product; the off-diagonal block $K_{\tilde N,M}^{21}$ then collapses to zero through an algebraic cancellation involving the Schur complement $W = A - BD^{-1}C$.

\begin{proof}[Proof of Theorem~\ref{thm:blockstr}]
	Let $\{\x^{(1)}, \ldots, \x^{(M)}\}$ be i.i.d.\ samples from $\mu$ and $\y^{(i)} = \mathbf{F}(\x^{(i)})$. \arXivtag[With the initial dictionary $\Dc_{\tilde N}=\{\psi_1, \ldots, \psi_{\tilde N}\}$ and the sub-dictionary $\Dc_{N}=\{\psi_1, \ldots, \psi_N\}$], the EDMD matrices\arXivtag are
	\begin{align}
		K_{\tilde N, M} & = \mathbf{\Psi}_{\tilde N}(\X)^{+}\,\mathbf{\Psi}_{\tilde N}(\Y) \;=\; \bigl(\mathbf{\Psi}_{\tilde N}(\X)^{*}\mathbf{\Psi}_{\tilde N}(\X)\bigr)^{-1}\mathbf{\Psi}_{\tilde N}(\X)^{*}\mathbf{\Psi}_{\tilde N}(\Y), \label{eq:EDMDmat-tildeN} \\
        K_{N, M} & = \mathbf{\Psi}_{N}(\X)^{+}\,\mathbf{\Psi}_{N}(\Y) \;=\; \bigl(\mathbf{\Psi}_{N}(\X)^{*}\mathbf{\Psi}_{N}(\X)\bigr)^{-1}\mathbf{\Psi}_{N}(\X)^{*}\mathbf{\Psi}_{N}(\Y). \label{eq:EDMDmat-N}
	\end{align}\arXivtag
	Partition the columns of the augmented evaluation matrices into the first $N$ (i.e.\ $\Dc_N$) and the remaining $\tilde N - N$ (i.e.\ $\Dc_{N_2} := \Dc_{\tilde N}\setminus\Dc_N$):
	\begin{equation*}
		\mathbf{\Psi}_{\tilde N}(\X) = \begin{bmatrix} \mathbf{\Psi}_{N}(\X) & \arXivtag[\mathbf{\Psi}_{N_2}](\X) \end{bmatrix}
		\quad\text{and}\quad
		\mathbf{\Psi}_{\tilde N}(\Y) = \begin{bmatrix} \mathbf{\Psi}_{N}(\Y) & \arXivtag[\mathbf{\Psi}_{N_2}](\Y) \end{bmatrix},
	\end{equation*}
	with $\mathbf{\Psi}_{N}(\X),\mathbf{\Psi}_{N}(\Y) \in \mathbb{C}^{M\times N}$ and $\arXivtag[\mathbf{\Psi}_{N_2}](\X),\arXivtag[\mathbf{\Psi}_{N_2}](\Y) \in \mathbb{C}^{M\times(\tilde N - N)}$. With this partition, \eqref{eq:EDMDmat-tildeN} reads
	\begin{align}
		K_{\tilde N, M}
		& = \begin{bmatrix}
			\mathbf{\Psi}_{N}(\X)^{*}\mathbf{\Psi}_{N}(\X) & \mathbf{\Psi}_{N}(\X)^{*}\arXivtag[\mathbf{\Psi}_{N_2}](\X) \\
			\arXivtag[\mathbf{\Psi}_{N_2}](\X)^{*}\mathbf{\Psi}_{N}(\X) & \arXivtag[\mathbf{\Psi}_{N_2}](\X)^{*}\arXivtag[\mathbf{\Psi}_{N_2}](\X)
		\end{bmatrix}^{-1}
		\begin{bmatrix}
			\mathbf{\Psi}_{N}(\X)^{*}\mathbf{\Psi}_{N}(\Y) & \mathbf{\Psi}_{N}(\X)^{*}\arXivtag[\mathbf{\Psi}_{N_2}](\Y) \\
			\arXivtag[\mathbf{\Psi}_{N_2}](\X)^{*}\mathbf{\Psi}_{N}(\Y) & \arXivtag[\mathbf{\Psi}_{N_2}](\X)^{*}\arXivtag[\mathbf{\Psi}_{N_2}](\Y)
		\end{bmatrix}.
		\label{eq:KtildeN-block-3}
	\end{align}
	To further simplify \eqref{eq:KtildeN-block-3}, we use the block matrix inversion formula
	\begin{equation}\label{eq:matrix-inversion}
		\begin{bmatrix}
			A & B \\
			C & D
		\end{bmatrix}^{-1}
		=
		\begin{bmatrix}
			W^{-1} & -W^{-1}BD^{-1}\\
			- D^{-1}C W^{-1} & D^{-1}+D^{-1}C W^{-1}BD^{-1}
		\end{bmatrix},
	\end{equation}
	valid whenever $D$ is invertible and $W\coloneqq A -BD^{-1}C$ (\arXivtag[i.e.,] the Schur complement of $D$) is invertible.
	Applying \eqref{eq:matrix-inversion} with $A=\mathbf{\Psi}_{N}(\X)^{*}\mathbf{\Psi}_{N}(\X)$, $B = \mathbf{\Psi}_{N}(\X)^{*}\arXivtag[\mathbf{\Psi}_{N_2}](\X)$, $C = \arXivtag[\mathbf{\Psi}_{N_2}](\X)^{*}\mathbf{\Psi}_{N}(\X)$, $D = \arXivtag[\mathbf{\Psi}_{N_2}](\X)^{*}\arXivtag[\mathbf{\Psi}_{N_2}](\X)$, we rewrite \eqref{eq:KtildeN-block-3} as
	\begin{equation}\label{eq:KtildeN-block-withW}
		K_{\tilde N, M} =
		\begin{bmatrix}
			W^{-1} & -W^{-1}BD^{-1}\\
			- D^{-1}C W^{-1} & D^{-1}+D^{-1}C W^{-1}BD^{-1}
		\end{bmatrix}
		\begin{bmatrix}
			\mathbf{\Psi}_{N}(\X)^{*}\mathbf{\Psi}_{N}(\Y) & \mathbf{\Psi}_{N}(\X)^{*}\arXivtag[\mathbf{\Psi}_{N_2}](\Y) \\
			\arXivtag[\mathbf{\Psi}_{N_2}](\X)^{*}\mathbf{\Psi}_{N}(\Y) & \arXivtag[\mathbf{\Psi}_{N_2}](\X)^{*}\arXivtag[\mathbf{\Psi}_{N_2}](\Y)
		\end{bmatrix}.
	\end{equation}
	By Assumption~\ref{assume:full-rank}, the Gram matrix $\mathbf{\Psi}_{\tilde N}(\X)^{*}\mathbf{\Psi}_{\tilde N}(\X)$ and its block $D$ are invertible, so the Schur complement $W$ is also invertible; hence \eqref{eq:KtildeN-block-withW} is valid.\arXivtag
	
	\emph{Part (\textit{i}).} From \eqref{eq:KtildeN-block-withW}, the top-left $N\by N$ block of $K_{\tilde N, M}$ equals
	\begin{align}
		& W^{-1}\mathbf{\Psi}_{N}(\X)^{*}\mathbf{\Psi}_{N}(\Y) - W^{-1}BD^{-1}\arXivtag[\mathbf{\Psi}_{N_2}](\X)^{*}\mathbf{\Psi}_{N}(\Y) \notag \\
		& \qquad = W^{-1}\bigl\{\mathbf{\Psi}_{N}(\X)^{*} - BD^{-1}\arXivtag[\mathbf{\Psi}_{N_2}](\X)^{*}\bigr\}\mathbf{\Psi}_{N}(\Y) \label{eq:KtildeN-top-left1} \\
		& \qquad = W^{-1}\bigl\{\mathbf{\Psi}_{N}(\X)^{*} - BD^{-1}\arXivtag[\mathbf{\Psi}_{N_2}](\X)^{*}\bigr\}\mathbf{\Psi}_{N}(\X)\,K_{N,M} \label{eq:KtildeN-top-left2} \\
		& \qquad = W^{-1}\bigl\{\mathbf{\Psi}_{N}(\X)^{*}\mathbf{\Psi}_{N}(\X) - BD^{-1}\arXivtag[\mathbf{\Psi}_{N_2}](\X)^{*}\mathbf{\Psi}_{N}(\X)\bigr\}K_{N,M} \notag \\
		& \qquad = W^{-1}\bigl\{A - BD^{-1}C\bigr\}K_{N,M} \;=\; W^{-1}\cdot W\cdot K_{N,M} \;=\; K_{N,M}. \notag
	\end{align}
	From \eqref{eq:KtildeN-top-left1} to \eqref{eq:KtildeN-top-left2} we used the invariance of $\Fc_N$, namely $\mathbf{\Psi}_{N}(\Y) = \mathbf{\Psi}_{N}(\X)\,K_{N,M}$, as in \eqref{eq:KpsiX-psiY}.
	
	\emph{Part (\textit{ii}).} Again from \eqref{eq:KtildeN-block-withW}, the bottom-left $(\tilde N - N)\by N$ block of $K_{\tilde N, M}$ equals
	\begin{align*}
		& - D^{-1}CW^{-1}\mathbf{\Psi}_{N}(\X)^{*}\mathbf{\Psi}_{N}(\Y) + \bigl(D^{-1} + D^{-1}CW^{-1}BD^{-1}\bigr)\arXivtag[\mathbf{\Psi}_{N_2}](\X)^{*}\mathbf{\Psi}_{N}(\Y) \\
		& \qquad = -D^{-1}CW^{-1}\mathbf{\Psi}_{N}(\X)^{*}\mathbf{\Psi}_{N}(\X)\,K_{N,M} + D^{-1}\arXivtag[\mathbf{\Psi}_{N_2}](\X)^{*}\mathbf{\Psi}_{N}(\X)\,K_{N,M} \\
		& \qquad\phantom{AAAAAAAAAAAAAA} + D^{-1}CW^{-1}BD^{-1}\arXivtag[\mathbf{\Psi}_{N_2}](\X)^{*}\mathbf{\Psi}_{N}(\X)\,K_{N,M} \\
		& \qquad = -D^{-1}CW^{-1}A\,K_{N,M} + D^{-1}C\,K_{N,M} + D^{-1}CW^{-1}BD^{-1}C\,K_{N,M} \\
		& \qquad = D^{-1}C\,W^{-1}\bigl(-A + W + BD^{-1}C\bigr)\,K_{N,M} \\
		& \qquad = D^{-1}C\,W^{-1}\cdot O_{N\by N}\cdot K_{N,M} \;=\; O_{(\tilde N - N)\times N}.
	\end{align*}
	This completes the proof.
\end{proof}

\begin{remark}[Remark for the proof]
	It is noteworthy that the data set to obtain the EDMD matrix $K_{\tilde N,M}$ does not need to be the same as the one for $K_{N,M}$. Since $\Fc_N$ is invariant, $K_{N,M}$ is identical to the matrix form (i.e., $K_N$) of the projection of the restricted Koopman operator, $\Koop_N$, under Assumption~\ref{assume:full-rank}.  Therefore, $K_{N,M} = K_N$ has no data dependency, and this proof can be simply reproduced with different data sets for the two EDMD matrices, $K_{\tilde N,M}$ and $K_{N,M}$.
\end{remark}

\subsection{Coupling bound and unconditional leakage}\label{sec:proofs-leakage}

\emph{Goal.} This subsection proves the unconditional leakage guarantee (Proposition~\ref{prop:unconditional}), which bounds multi-step Koopman leakage by the PPR mass outside the selected set.

\emph{Strategy.} The bridge between multi-step Koopman dynamics and the PPR Neumann series is a single entrywise domination result (Lemma~\ref{lem:coupling-bound}) showing that the entries of $K_{\tilde{N}}^k$ are controlled by those of the row-normalized matrix $P^k$ up to a factor $r_{\max}^k$. Given this control, the discounting rates align via the change of variables $\gamma = \alpha/r_{\max}$, which converts the Koopman-side geometric sum into the PPR Neumann series and reveals $(1-\gamma)/(1-\alpha)\cdot(1 - \pi_{\mathbf{s}}(S_N))$ as the leakage bound.

\begin{lemma}[Entrywise coupling to a row-normalized matrix]\label{lem:coupling-bound}
	Let $A$ be a square matrix with absolute row sums $r_i := \sum_l |A[i,l]| > 0$, row normalization $\hat A[i,j] := |A[i,j]|/r_i$, and $r_{\max} := \max_i r_i$. Let $|A|$ denote the entrywise modulus of $A$. Then for all $v, j$ and $k \geq 0$:
	\begin{equation}\label{eq:coupling-bound}
		|A^k[v,j]| \leq |A|^k[v,j] \leq r_{\max}^k\,\hat A^k[v,j].
	\end{equation}
\end{lemma}

\begin{proof}
	The first inequality is the entrywise triangle inequality applied to the $k$-fold product:
	\begin{equation*}
		|A^k[v,j]| = \Bigl|\sum_{l_1,\ldots,l_{k-1}} A[v,l_1]\cdots A[l_{k-1},j]\Bigr| \leq \sum_{l_1,\ldots,l_{k-1}} |A[v,l_1]|\cdots |A[l_{k-1},j]| = |A|^k[v,j].
	\end{equation*}
	We prove the second inequality by induction on $k$. The base case $k=0$ is straightforward. For $k=1$: $|A|[v,j] = r_v\,\hat A[v,j] \leq r_{\max}\,\hat A[v,j]$, using $|A[v,j]| = r_v\hat A[v,j]$ by definition of $\hat A$. For the inductive step $k \to k+1$:
	\begin{align*}
		|A|^{k+1}[v,j] &= \sum_l |A|^k[v,l]\, |A[l,j]| \\
		&\leq \sum_l r_{\max}^k\,\hat A^k[v,l] \cdot r_l\,\hat A[l,j] \leq r_{\max}^{k+1} \sum_l \hat A^k[v,l]\,\hat A[l,j] = r_{\max}^{k+1}\,\hat A^{k+1}[v,j],
	\end{align*}
	where the inequality uses the inductive hypothesis $|A|^k[v,l] \leq r_{\max}^k \hat A^k[v,l]$, the identity $|A[l,j]| = r_l\hat A[l,j]$, and $r_l \leq r_{\max}$.
\end{proof}

\begin{proposition}[Multi-step leakage and PPR]\label{prop:unconditional}
	Let $P = \mathcal{R}(|K_{\tilde{N},M}^{\tp}|)$, $r_{\max,M} = \|K_{\tilde{N},M}^{\tp}\|_\infty$, and $S_N \subset \{1,\ldots,\tilde{N}\}$ with $|S_N| = N$. Assume $\alpha < r_{\max,M}$ and set $\gamma := \alpha/r_{\max,M} \in (0,1)$ (when $r_{\max,M} \geq 1$ this holds for every $\alpha \in (0,1)$; when $r_{\max,M} < 1$, the damping factor must satisfy $\alpha < r_{\max,M}$). For any preference vector $\mathbf{s}$ supported on $S_N$, the \emph{$\gamma$-discounted multi-step leakage} defined by
	\begin{equation}\label{eq:leakage-ppr}
		\Lambda_{\mathbf{s}}^{\gamma}(S_N)\coloneqq(1-\gamma)\sum_{k=1}^{\infty} \gamma^k \sum_{j \notin S_N} [\mathbf{s}^\tp  |K_{\tilde{N},M}^{\tp}|^k]_j
	\end{equation}
	satisfies
	$$
	\Lambda_{\mathbf{s}}^{\gamma}(S_N) \leq \frac{1-\gamma}{1-\alpha} \left(1 - \pi_{\mathbf{s}, M}(S_N)\right),
	$$
	where $\pi_{\mathbf{s}, M}$ is the PPR vector of $P$ with preference $\mathbf{s}$ and damping $\alpha$.
\end{proposition}

\begin{proof}
	Write $P_v^{(k)} = \sum_{j \notin S_N} P^k[v,j]$. For $v \in S_N$, we have $P_v^{(0)} = 0$. Apply Lemma~\ref{lem:coupling-bound} with $A = K_{\tilde{N},M}^\tp$; then $\hat A = \mathcal{R}(|K_{\tilde{N},M}^\tp|) = P$, and equation~\eqref{eq:coupling-bound} gives $|K_{\tilde{N},M}^\tp|^k[v,j] \leq r_{\max,M}^k\,P^k[v,j]$, so
	\begin{equation*}
		\sum_{j \notin S_N} |K_{\tilde{N},M}^\tp|^k[v,j] \leq r_{\max,M}^k P_v^{(k)}.
	\end{equation*}
	Expanding the definition of $\Lambda_{\mathbf{s}}^\gamma(S_N) = (1-\gamma)\sum_{k=1}^\infty \gamma^k \sum_{j \notin S_N}[\mathbf{s}^\tp |K_{\tilde{N},M}^\tp|^k]_j$ and using $[\mathbf{s}^\tp|K_{\tilde{N},M}^\tp|^k]_j = \sum_v s_v |K_{\tilde{N},M}^\tp|^k[v,j]$,
	\begin{equation*}
		\Lambda_{\mathbf{s}}^\gamma(S_N) \leq \sum_v s_v \cdot (1-\gamma)\sum_{k=1}^\infty \gamma^k r_{\max,M}^k\,P_v^{(k)}.
	\end{equation*}
	Setting $\gamma = \alpha/r_{\max,M}$ gives $\gamma^k r_{\max,M}^k = \alpha^k$, so
	\begin{equation*}
		\Lambda_{\mathbf{s}}^\gamma(S_N) \leq \sum_v s_v \cdot (1-\gamma)\sum_{k=1}^\infty \alpha^k P_v^{(k)} = \sum_v s_v \cdot \frac{1-\gamma}{1-\alpha}(1 - \pi_{v,M}(S_N)),
	\end{equation*}
	where the last equality uses the single-seed PPR Neumann series \eqref{eq:ppr-singleseed}: $(1-\alpha)\sum_{k=1}^\infty \alpha^k P_v^{(k)} = 1 - \pi_{v,M}(S_N)$. By linearity of $\mathbf{s} \mapsto \pi_{\mathbf{s}, M}$, $\sum_v s_v\,\pi_{v,M}(S_N) = \pi_{\mathbf{s}, M}(S_N)$, giving
	\begin{equation*}
		\Lambda_{\mathbf{s}}^\gamma(S_N) \leq \frac{1-\gamma}{1-\alpha}\left(1 - \pi_{\mathbf{s}, M}(S_N)\right). \qedhere
	\end{equation*}
\end{proof}

\subsection{Proof of Theorem \ref{thm:detection}}\label{sec:detection_proofs}

\emph{Goal.} This subsection proves Theorem~\ref{thm:detection}. The main-text sketch (\S\ref{sec:lemmas}) reduces detection on $P$ to (a) a perturbation bound between $P$ and $P^{(0)}$ and (b) explicit lower bounds on the auxiliary gaps $\Delta^{(0)}_{\textnormal{PR}}$, $\Delta^{(0)}_{\textnormal{PPR}}$ via the closed-form expressions of Lemma~\ref{lem:gaps-closed-form} (only this lemma is stated in the main text).

\emph{Strategy.} We first prove Lemma~\ref{lem:gaps-closed-form} (\S\ref{sec:proof-gaps-closed-form}) by computing $(I-\alpha P^{(0)})^{-1}$ via the block-lower-triangular structure of $P^{(0)}$. We then derive the explicit lower bounds: on $\Delta^{(0)}_{\textnormal{PR}}$ (Lemma~\ref{lem:pagerank-combined}, by dropping $s^{\textnormal{eff}} \geq 1$ to $1$ and bounding $\|R_{22}\|_\infty$ via the coupling parameter $\eta$), and on $\Delta^{(0)}_{\textnormal{PPR}}$ (Lemma~\ref{lem:ppr-combined}, by reading off the closed form). Next, the shared perturbation bound (Lemma~\ref{lem:pr-ppr-perturbation}) follows from a generic resolvent perturbation lemma (Lemma~\ref{lem:pagerank-perturbation})\arXivtag. Finally, an abstract detection criterion (Lemma~\ref{lem:abstract-detection}) composes these ingredients to yield Theorem~\ref{thm:detection}. The row-normalization stability lemma (Lemma~\ref{lem:row-norm-perturbation}) is also stated here but is used only in the finite-sample proof of Theorem~\ref{thm:end-to-end}.

\subsubsection{Closed-form detection gaps}\label{sec:proof-gaps-closed-form}

We begin with the closed-form detection gaps of Lemma~\ref{lem:gaps-closed-form}. The proof proceeds in three steps: a block-lower-triangular form of the resolvent $(I-\alpha P^{(0)})^{-1}$ (Step~1), from which the standard PR (Step~2) and multi-seed PPR (Step~3) closed forms follow by direct computation.

\begin{proof}[Proof of Lemma~\ref{lem:gaps-closed-form}]
	\emph{Step 1: block form of the resolvent.} Since $P^{(0)}$ has zero right-top block, an induction on $k$ shows that $(P^{(0)})^k$ inherits the same zero block for every $k \geq 0$, with top-left block $(P^{11,(0)})^k$ and bottom-right block $(P^{22,(0)})^k$: if $(P^{(0)})^k = \begin{bmatrix}(P^{11,(0)})^k & O \\ C_k & (P^{22,(0)})^k\end{bmatrix}$, then
	\begin{equation*}
		(P^{(0)})^{k+1} = P^{(0)}(P^{(0)})^k = \begin{bmatrix}(P^{11,(0)})^{k+1} & O \\ P^{21,(0)}(P^{11,(0)})^k + P^{22,(0)}C_k & (P^{22,(0)})^{k+1}\end{bmatrix},
	\end{equation*}
	with $C_0 = O$. Summing the Neumann series $(I - \alpha P^{(0)})^{-1} = \sum_{k\geq 0}\alpha^k(P^{(0)})^k$ yields
	\begin{equation}\label{eq:resolvent-block}
		(I - \alpha P^{(0)})^{-1} = \begin{bmatrix} R_{11} & O \\ \alpha R_{22}P^{21,(0)}R_{11} & R_{22}\end{bmatrix},
	\end{equation}
	where $R_{11} := (I - \alpha P^{11,(0)})^{-1}$ and $R_{22} := (I - \alpha P^{22,(0)})^{-1}$. The lower-left block follows either from the block-inverse identity or from $\sum_{k\geq 0}\alpha^k C_k$ via the recursion above. Since $P^{11,(0)}, P^{22,(0)}$ are entrywise nonnegative with $\|P^{11,(0)}\|_\infty = 1$ (row-stochastic) and $\|P^{22,(0)}\|_\infty \leq 1$ and $\alpha < 1$, each Neumann series converges absolutely; hence $R_{11}, R_{22}$ are entrywise nonnegative with finite entries.
	
	\emph{Step 2: standard PR closed form.} Writing ${\pi}^{(0),\tp} = \tfrac{1-\alpha}{\tilde{N}}\mathbf{1}^\tp(I-\alpha P^{(0)})^{-1}$ and splitting $\mathbf{1}^\tp = [\mathbf{1}_N^\tp\;\;\mathbf{1}_{\tilde N - N}^\tp]$, equation~\eqref{eq:resolvent-block} gives, for $i \leq N$,
	\begin{equation}\label{eq:pi0-block1-formula}
		\pi_i^{(0)} \;=\; \frac{1-\alpha}{\tilde{N}}\bigl[\mathbf{1}_N^\tp R_{11} + \alpha\,\mathbf{1}_{\tilde N - N}^\tp R_{22}P^{21,(0)}R_{11}\bigr]_i \;=\; \frac{1-\alpha}{\tilde{N}}\sum_{v \leq N} s_v^{\textnormal{eff}}\,[R_{11}]_{vi},
	\end{equation}
	where $s_v^{\textnormal{eff}} := 1 + \alpha\sum_{w > N}[R_{22}P^{21,(0)}]_{wv}$. Since $R_{22},P^{21,(0)} \geq 0$ entrywise, $s_v^{\textnormal{eff}} \geq 1$ for every $v \leq N$. For $j > N$, the top-right block of the resolvent vanishes, so only the $\mathbf{1}_{\tilde N - N}^\tp R_{22}$ contribution survives:
	\begin{equation}\label{eq:pi0-block2-formula}
		\pi_j^{(0)} \;=\; \frac{1-\alpha}{\tilde{N}}\sum_{w > N}[R_{22}]_{wj}.
	\end{equation}
	Taking $\min_{i \leq N}$ of \eqref{eq:pi0-block1-formula} and $\max_{j > N}$ of \eqref{eq:pi0-block2-formula} and subtracting yields the PR closed form in \eqref{eq:delta-closed}.

		\emph{Step 3: multi-seed PPR closed form and vanishing leakage.} For $v \leq N$, define the single seed PPR vector $(\pi_v^{(0)})^\tp := (1-\alpha)\,\mathbf{e}_v^\tp(I-\alpha P^{(0)})^{-1}$. Because the top-right block of the resolvent \eqref{eq:resolvent-block} vanishes, $\pi_v^{(0)}(j) = 0$ for all $j > N$, and $\pi_v^{(0)}(i) = (1-\alpha)[R_{11}]_{vi}$ for $i \leq N$. Averaging over $v \in \mathcal{S} \subset \{1,\ldots,N\}$,
		\begin{equation}\label{eq:ppr-bar-formula}
			\pi_{\mathcal{S}}^{(0)}(j) = 0 \quad (j > N), \qquad \pi_{\mathcal{S}}^{(0)}(i) = \frac{1-\alpha}{|\mathcal{S}|}\sum_{v \in \mathcal{S}}[R_{11}]_{vi} \quad (i \leq N).
		\end{equation}
		Hence the max-leakage term in \eqref{eq:delta0-defs} vanishes identically, and $\Delta^{(0)}_{\textnormal{PPR}} = \min_{i\leq N}\pi_{\mathcal{S}}^{(0)}(i)$ takes the form displayed in \eqref{eq:delta-closed}.

\end{proof}

The next two lemmas specialize the closed-form expressions \arXivtag[\eqref{eq:pi0-block1-formula}--\eqref{eq:ppr-bar-formula}] of Lemma~\ref{lem:gaps-closed-form} into forms used by the detection theorem: the relaxed PR baseline that replaces $s_v^{\textnormal{eff}} \geq 1$ by $1$ and bounds $\|R_{22}\|_\infty$ via the coupling parameter $\eta$, and the multi-seed PPR baseline that reads off the vanishing-leakage form directly.

\begin{lemma}[Standard PR baseline]\label{lem:pagerank-combined}
	Let $\alpha \in (0,1)$ be a damping factor. 
	Let $(\mathbf{q}^\alpha)^\tp := \tfrac{1-\alpha}{N}\mathbf{1}_N^\tp(I-\alpha P^{11,(0)})^{-1}$ and $q_{\min}^\alpha := \min_{i\leq N}q_i^\alpha$. Then
	\begin{align}
		\max_{j > N}\pi_j^{(0)} \;&\leq\; \frac{(1-\alpha)(\tilde{N}-N)}{\tilde{N}(1-\alpha+\alpha\eta)}, \label{eq:PR-max-leakage}\\
		\min_{i \leq N}\pi_i^{(0)} \;&\geq\; \frac{N}{\tilde{N}}\,q_{\min}^\alpha, \label{eq:PR-min-block1}
	\end{align}
	and consequently
	\begin{equation}\label{eq:PR-gap-lemma}
		\Delta^{(0)}_{\textnormal{PR}} \;\geq\; \frac{N}{\tilde{N}}\,q_{\min}^\alpha \;-\; \frac{(1-\alpha)(\tilde{N}-N)}{\tilde{N}(1-\alpha+\alpha\eta)},
	\end{equation}
	positive whenever $q_{\min}^\alpha > (\tilde{N}-N)(1-\alpha)/\bigl(N(1-\alpha+\alpha\eta)\bigr)$.
\end{lemma}

\begin{proof}
	From \eqref{eq:pi0-block1-formula} and $s_v^{\textnormal{eff}} \geq 1$,
	\[
	\pi_i^{(0)} \;\geq\; \frac{1-\alpha}{\tilde N}\bigl[\mathbf{1}_N^\tp R_{11}\bigr]_i \;=\; \frac{N}{\tilde N}\,q_i^\alpha,
	\]
	using $(\mathbf{q}^{\alpha})^\tp = \tfrac{1-\alpha}{N}\mathbf{1}_N^\tp R_{11}$; taking $\min_{i\leq N}$ gives \eqref{eq:PR-min-block1}. Since $R_{11} = \sum_{k \geq 0}\alpha^k(P^{11,(0)})^k$ contains the identity term, $[\mathbf{1}_N^\tp R_{11}]_i \geq 1$, so $q_{\min}^\alpha \geq (1-\alpha)/N$.
	
	For the \arXivtag[bottom-block] upper bound, from \eqref{eq:pi0-block2-formula} and $R_{22} = (I-\alpha P^{22,(0)})^{-1} = \sum_{k\geq 0}\alpha^k(P^{22,(0)})^k \geq 0$ entrywise,
		\[
		\pi_j^{(0)} = \frac{1-\alpha}{\tilde{N}}\sum_{w>N}[R_{22}]_{wj} \geq 0,
		\]
        \arXivtag[where row and column indices of $R_{22}$ start at $N+1$, consistent with the full $\tilde{N}\times\tilde{N}$ indexing. Specifically, $R_{22}$ has the size of $(\tilde N - N) \by (\tilde N - N)$ so its row and column indices are $N+1, N+2, \ldots, \tilde{N}$.]
		Summing over $j > N$ and using $\sum_j [R_{22}]_{wj} \leq \frac{1}{1-\alpha\|P^{22,(0)}\|_\infty} \leq \frac{1}{1-\alpha(1-\eta)}$
		(since $\|P^{22,(0)}\|_\infty \leq 1-\eta$ and $R_{22} = \sum_{k\geq 0}\alpha^k(P^{22,(0)})^k$):
		\[
		\|{\pi}_2^{(0)}\|_1 = \sum_{j>N}\pi_j^{(0)} \leq \frac{(1-\alpha)(\tilde{N}-N)}{\tilde{N}(1-\alpha+\alpha\eta)},
		\]
		which implies \eqref{eq:PR-max-leakage} since $\pi_j^{(0)} \geq 0$. Subtracting \eqref{eq:PR-max-leakage} from \eqref{eq:PR-min-block1} gives \eqref{eq:PR-gap-lemma}.
	
	As $\alpha \to 1$, \eqref{eq:PR-max-leakage} gives $\max_{j>N}\pi_j^{(0)} \to 0$, and from \eqref{eq:PR-min-block1} we get $\Delta^{(0)}_{\textnormal{PR}} \geq (N/\tilde{N})q^\alpha_{\min} \to (N/\tilde{N})\lim_{\alpha\to 1}q^\alpha_{\min}$. Note that $q^\alpha_{\min} \geq (1-\alpha)/N \to 0$, so this does not immediately give a positive lower bound; the bound is non-vacuous only when $q^\alpha_{\min}$ decays slower than $1-\alpha$, which occurs when $P^{11,(0)}$ is ``well-mixed'' (Theorem~\ref{thm:detection}(i)).
\end{proof}

\begin{lemma}[PPR baseline]\label{lem:ppr-combined}
	For $v \leq N$, define the PPR vector $({\pi}_v^{(0)})^\tp := (1-\alpha)\,\mathbf{e}_v^\tp(I-\alpha P^{(0)})^{-1}$ with a single seed $v$,
		and for any seed set $\mathcal{S} \subset \{1,\ldots,N\}$ define
		${\pi}_{\mathcal{S}}^{(0)} := \frac{1}{|\mathcal{S}|}\sum_{v \in \mathcal{S}}{\pi}_v^{(0)}$.
	Then $\pi_{\mathcal{S}}^{(0)}(j) = 0$ for all $j > N$, and
	\begin{equation}\label{eq:PPR-gap-lemma}
		\Delta^{(0)}_{\textnormal{PPR}} \;=\; \min_{i \leq N}\pi_{\mathcal{S}}^{(0)}(i) \;=\; \min_{i \leq N}\frac{1-\alpha}{|\mathcal{S}|}\sum_{v \in \mathcal{S}}\bigl[(I - \alpha P^{11,(0)})^{-1}\bigr]_{vi}.
	\end{equation}
\end{lemma}

\begin{proof}
	This is Step~3 of the proof of Lemma~\ref{lem:gaps-closed-form}: formulas \eqref{eq:ppr-bar-formula} read off directly, and substituting into $\Delta^{(0)}_{\textnormal{PPR}} = \min_{i \leq N}\pi_{\mathcal{S}}^{(0)}(i) - \max_{j > N}\pi_{\mathcal{S}}^{(0)}(j)$ gives \eqref{eq:PPR-gap-lemma}.
\end{proof}

\subsubsection{Perturbation bounds}\label{sec:proofs-perturbation}

The shared perturbation bound (Lemma~\ref{lem:pr-ppr-perturbation}) follows directly from a resolvent perturbation bound (Lemma~\ref{lem:pagerank-perturbation} below): since $P^{(0)} = \mathcal{R}(|K^{(0)}|^\tp)$ is obtained from $P$ by zeroing $P^{12}$ \emph{and renormalizing \arXivtag[the rows of the top block]}, a direct calculation (in the proof) gives $\|P - P^{(0)}\|_\infty = 2\|P^{12}\|_\infty$, and Lemma~\ref{lem:pagerank-perturbation} then gives the bound with sensitivity constant $2\alpha/(1-\alpha)$. The row-normalization stability lemma (Lemma~\ref{lem:row-norm-perturbation} below) is used in the finite-sample proof of Theorem~\ref{thm:end-to-end} to bound $\|\mathcal{R}(|K_{\tilde{N},M}^\tp|) - \mathcal{R}(|K_{\tilde{N}}^\tp|)\|_\infty$. Both lemmas are independent of the preference vector, so they cover PR and PPR uniformly.

We begin with row-normalization stability. The row normalization $\mathcal{R}(A)$ divides each entry by its absolute row sum, introducing a denominator that couples all entries in a row; perturbations affect both the numerator (individual entries) and the denominator (absolute row sum), leading to the factor of $2$.

\begin{lemma}[Stability of row normalization]\label{lem:row-norm-perturbation}
	Let $A, B \in \R^{n \times n}$ have all nonzero rows, and let $r^A_{\min} \coloneqq \min_i \sum_j |A[i,j]|$. Then
	$\|\mathcal{R}(A) - \mathcal{R}(B)\|_\infty \leq \frac{2}{r^A_{\min}} \|A - B\|_\infty$.
\end{lemma}

\begin{proof}
	Write $r_i^A = \sum_l |A[i,l]|$ and $r_i^B = \sum_l |B[i,l]|$ for the row sums. For each row $i$:
	\begin{align*}
		\mathcal{R}(A)[i,j] - \mathcal{R}(B)[i,j] &= \frac{|A[i,j]|}{r_i^A} - \frac{|B[i,j]|}{r_i^B} \\
		&= \frac{|A[i,j]| - |B[i,j]|}{r_i^A} + |B[i,j]|\left(\frac{1}{r_i^A} - \frac{1}{r_i^B}\right) \\
		&= \frac{|A[i,j]| - |B[i,j]|}{r_i^A} + \frac{|B[i,j]|(r_i^B - r_i^A)}{r_i^A \cdot r_i^B}.
	\end{align*}
	Summing moduli over $j$:
	\begin{align*}
		\sum_j \left|\mathcal{R}(A)[i,j] - \mathcal{R}(B)[i,j]\right| &\leq \frac{1}{r_i^A}\sum_j \bigl||A[i,j]| - |B[i,j]|\bigr| + \frac{|r_i^A - r_i^B|}{r_i^A \cdot r_i^B}\sum_j |B[i,j]| \\
		&\leq \frac{\delta_i}{r_i^A} + \frac{\delta_i}{r_i^A},
	\end{align*}
	where $\delta_i = \sum_j |A[i,j] - B[i,j]|$ is the $i$-th row sum of $|A - B|$. In the second inequality, we used $\bigl||A[i,j]| - |B[i,j]|\bigr| \leq |A[i,j] - B[i,j]|$ (reverse triangle inequality), the bound $|r_i^A - r_i^B| \leq \delta_i$ (same argument applied to row sums), and $\sum_j |B[i,j]| = r_i^B$. Taking the maximum over $i$:
	\begin{equation*}
		\norm{\mathcal{R}(A) - \mathcal{R}(B)}_\infty \leq \max_i \frac{2\delta_i}{r_i^A} \leq \frac{2}{r^A_{\min}} \max_i \delta_i = \frac{2}{r^A_{\min}} \norm{A-B}_\infty. \qedhere
	\end{equation*}
\end{proof}

Next, the resolvent perturbation bound. The identity $(I - \alpha S')^{-1} - (I - \alpha S)^{-1} = (I - \alpha S)^{-1}\alpha(S'-S)(I - \alpha S')^{-1}$ shows that the PPR difference factors through the matrix difference $S' - S$, with the resolvent norms contributing the $\alpha/(1-\alpha)$ amplification factor.

\begin{lemma}[Resolvent perturbation bound]\label{lem:pagerank-perturbation}
	Let $S, S'$ be nonnegative matrices with $\|S\|_\infty \leq 1$,
	$\|S'\|_\infty \leq 1$, and $\alpha \in (0,1)$. For any
	$\mathbf{s} \geq 0$ with $\|\mathbf{s}\|_1 = 1$, define
	\[
	\mathbf{q}^\tp := (1-\alpha)\mathbf{s}^\tp(I-\alpha S)^{-1} \quad \arXivtag[\text{and}]
	\quad
	(\mathbf{q}')^\tp := (1-\alpha)\mathbf{s}^\tp(I-\alpha S')^{-1}.
	\]
	Then $\|\mathbf{q} - \mathbf{q}'\|_1 \leq
	\frac{\alpha}{1-\alpha}\|S - S'\|_\infty$.
\end{lemma}

\begin{proof}
	Since $\|S\|_\infty \leq 1$ and $\|S'\|_\infty \leq 1$ with $\alpha < 1$,
	the Neumann series $(I-\alpha S)^{-1} = \sum_{k \geq 0}\alpha^k S^k$ and
	$(I-\alpha S')^{-1} = \sum_{k \geq 0}\alpha^k (S')^k$ both converge
	absolutely, so $\mathbf{q}$ and $\mathbf{q}'$ are well defined.
	Using the resolvent identity
	$(I - \alpha S')^{-1} - (I - \alpha S)^{-1}
	= (I - \alpha S)^{-1}\cdot \alpha(S' - S)\cdot (I - \alpha S')^{-1}$,
	we obtain
	\begin{equation}\label{eq:pi-diff-resolvent}
		(\mathbf{q}')^\tp - \mathbf{q}^\tp
		= \alpha \cdot \mathbf{q}^\tp \cdot (S' - S) \cdot (I - \alpha S')^{-1}.
	\end{equation}
	For any row vector $\mathbf{v}^\tp \in \R^{1 \times n}$ and matrix
	$A \in \R^{n \times n}$, the $\ell^1$-norm satisfies
	\begin{equation}\label{eq:ell1-submult}
		\norm{\mathbf{v}^\tp A}_1
		= \sum_j \left|\sum_i v_i A_{ij}\right|
		\leq \sum_j \sum_i |v_i| |A_{ij}|
		= \sum_i |v_i| \sum_j |A_{ij}|
		\leq \norm{\mathbf{v}}_1 \norm{A}_\infty.
	\end{equation}
	Since $\|S'\|_\infty \leq 1$, we have $\|(S')^k\|_\infty \leq 1$
	for all $k \geq 0$, so
	$\norm{(I - \alpha S')^{-1}}_\infty \leq \sum_{k=0}^\infty \alpha^k
	= \frac{1}{1-\alpha}$.
	Similarly, $\|(I-\alpha S)^{-1}\|_\infty \leq \frac{1}{1-\alpha}$,
	which gives
	\[
	\|\mathbf{q}\|_1
	= (1-\alpha)\|\mathbf{s}^\tp(I-\alpha S)^{-1}\|_1
	\leq (1-\alpha)\|\mathbf{s}\|_1 \cdot \frac{1}{1-\alpha} = 1.
	\]
	Applying \eqref{eq:ell1-submult} twice to \eqref{eq:pi-diff-resolvent}:
	\begin{align*}
		\norm{\mathbf{q}' - \mathbf{q}}_1
		&= \norm{\alpha \cdot \mathbf{q}^\tp(S' - S)(I - \alpha S')^{-1}}_1 \\
		&\leq \alpha \cdot \norm{\mathbf{q}^\tp(S'-S)}_1
		\cdot \norm{(I-\alpha S')^{-1}}_\infty \\
		&\leq \alpha \cdot \underbrace{\norm{\mathbf{q}}_1}_{\leq\,1}
		\cdot \norm{S' - S}_\infty \cdot \frac{1}{1-\alpha}
		\;\leq\; \frac{\alpha}{1-\alpha}\norm{S'-S}_\infty. \qedhere
	\end{align*}
\end{proof}

As a corollary, specializing Lemma~\ref{lem:pagerank-perturbation} to $S = P^{(0)}$ and $S' = P$ yields the perturbation bound used in the end-to-end detection theorem (Theorem~\ref{thm:end-to-end}).
	\begin{lemma}[Perturbation bound for PR/PPR vectors on $P$ and $P^{(0)}$]\label{lem:pr-ppr-perturbation}
		Let $P = \mathcal{R}(|K_{\tilde{N},M}^\tp|)$ be row stochastic and $P^{(0)} := \mathcal{R}(|K_{\tilde{N},M}^{(0)}|^\tp)$ be the row-stochastic matrix obtained from $P$ by zeroing its top-right block $P^{12}$ and renormalizing the \arXivtag[rows in the top block]. For any $\mathbf{s} \geq 0$ with $\|\mathbf{s}\|_1 = 1$, define the PR/PPR vectors $\pi_{\mathbf{s},M}^\tp := (1-\alpha)\mathbf{s}^\tp(I - \alpha P)^{-1}$ and $(\pi_{\mathbf{s},M}^{(0)})^\tp := (1-\alpha)\mathbf{s}^\tp(I - \alpha P^{(0)})^{-1}$. Then $\|P - P^{(0)}\|_\infty = 2\|P^{12}\|_\infty$, and
		\begin{equation}\label{eq:shared-perturbation}
			\|\pi_{\mathbf{s},M} - \pi_{\mathbf{s},M}^{(0)}\|_1 \;\leq\; \frac{2\alpha}{1-\alpha}\|P^{12}\|_\infty.
		\end{equation}
	\end{lemma}

\begin{proof}
	$P$ and $P^{(0)}$ agree on rows $i > N$. For $i \leq N$, write $r_i = \sum_l |K_{\tilde{N},M}^\tp[i,l]|$, $\arXivtag[a_{ij}] = |K_{\tilde{N},M}^\tp[i,j]|$ for $j \leq N$, and $\arXivtag[b_{ij}] = |K_{\tilde{N},M}^\tp[i,j]|$ for $j > N$. Then $P[i,j] = \arXivtag[a_{ij}]/r_i$ and $P[i,j] = \arXivtag[b_{ij}]/r_i$ on the two blocks, while $P^{(0)}[i,j] = \arXivtag[a_{ij}]/(r_i - \sum_{k} \arXivtag[b_{ik}])$ for $j \leq N$ and $0$ for $j > N$ \arXivtag[ where $\sum_{k}$ ranges over $k\in\{N+1,\ldots,\tilde{N}\}$.] A direct calculation gives
		\[
		\sum_{j \leq N}|P[i,j] - P^{(0)}[i,j]| = \sum_{j \leq N} \arXivtag[a_{ij}]\,\frac{\sum_{k} \arXivtag[b_{ik}]}{r_i(r_i-\sum_{k} \arXivtag[b_{ik}])} = \frac{\sum_{k} \arXivtag[b_{ik}]}{r_i}
		\]
        and
        \[
        \sum_{j > N}|P[i,j] - P^{(0)}[i,j]| = \frac{\sum_{k} \arXivtag[b_{ik}]}{r_i},
        \]
		so the row-$i$ $\ell^1$ difference is $2\sum_{k} \arXivtag[b_{ik}]/r_i = 2\|P^{12}_{i,:}\|_1$. 
        Maximizing over $i \leq N$ gives $\|P - P^{(0)}\|_\infty = 2\|P^{12}\|_\infty$. Both $P$ and $P^{(0)}$ are row stochastic so $\|\cdot\|_\infty \leq 1$; Lemma~\ref{lem:pagerank-perturbation} applies with \arXivtag[$S = P$ and $S' = P^{(0)}$], giving $\|\pi_{\mathbf{s},M} - \pi_{\mathbf{s},M}^{(0)}\|_1 \leq \frac{\alpha}{1-\alpha}\|P-P^{(0)}\|_\infty = \frac{2\alpha}{1-\alpha}\|P^{12}\|_\infty$. \qedhere
\end{proof}

\subsection{Proof of detection theorem}\label{sec:proofs-detection}

Now using the perturbation bound in Lemma \ref{lem:pr-ppr-perturbation}, we obtain the following abstract detection criterion. 

\begin{lemma}[Abstract detection criterion]\label{lem:abstract-detection}
	Let ${\pi}_{\mathbf{s}, M}$ and ${\pi}_{\mathbf{s}, M}^{(0)}$ be as in Lemma~\ref{lem:pr-ppr-perturbation}, with gap $\Delta_{\mathbf{s}} = \min_{i \leq N}\pi_{\mathbf{s}, M}^{(0)}(i) - \max_{j > N}\pi_{\mathbf{s}, M}^{(0)}(j)$. If $\Delta_{\mathbf{s}} > 0$ and
	\begin{equation}\label{eq:abstract-gap-condition}
		\|P^{12}\|_\infty \;<\; \frac{1-\alpha}{4\alpha}\cdot \Delta_{\mathbf{s}},
	\end{equation}
	then $\min_{i \leq N}\pi_{\mathbf{s}, M}(i) > \max_{j > N}\pi_{\mathbf{s}, M}(j)$, so Algorithm~\ref{algor:PR-EDMD} (using ${\pi}_{\mathbf{s}, M}$) correctly identifies $\{1,\ldots,N\}$ as the top-$N$ coordinates. In particular, this holds for standard PR ($\mathbf{s} = \mathbf{1}/\tilde{N}$, $\Delta_{\mathbf{s}} = \Delta^{(0)}_{\textnormal{PR}}$) and multi-seed PPR ($\mathbf{s} = |\mathcal{S}|^{-1}\sum_{v \in \mathcal{S}}\mathbf{e}_v$ with $\mathcal{S} \subset \{1,\ldots,N\}$, $\Delta_{\mathbf{s}} = \Delta^{(0)}_{\textnormal{PPR}}$).
\end{lemma}

\begin{proof}
	Algorithm~\ref{algor:PR-EDMD} correctly identifies $\{1,\ldots,N\}$ as the top-$N$ coordinates when $\min_{i\leq N}\pi_{\mathbf{s},M}(i) > \max_{j>N}\pi_{\mathbf{s},M}(j)$. Since
	\begin{equation*}
		\pi_{\mathbf{s},M}(i) \geq \pi_{\mathbf{s},M}^{(0)}(i) - \norm{{\pi}_{\mathbf{s},M} - {\pi}_{\mathbf{s},M}^{(0)}}_\infty, \qquad \pi_{\mathbf{s},M}(j) \leq \pi_{\mathbf{s},M}^{(0)}(j) + \norm{{\pi}_{\mathbf{s},M} - {\pi}_{\mathbf{s},M}^{(0)}}_\infty,
	\end{equation*}
	it is sufficient to have $\norm{{\pi}_{\mathbf{s},M} - {\pi}_{\mathbf{s},M}^{(0)}}_\infty < \Delta_{\mathbf{s}}/2$, and a fortiori $\norm{{\pi}_{\mathbf{s},M} - {\pi}_{\mathbf{s},M}^{(0)}}_1 < \Delta_{\mathbf{s}}/2$. By Lemma~\ref{lem:pr-ppr-perturbation},
	\begin{equation*}
		\norm{{\pi}_{\mathbf{s},M} - {\pi}_{\mathbf{s},M}^{(0)}}_1 \leq \frac{2\alpha}{1-\alpha}\norm{P^{12}}_\infty,
	\end{equation*}
	so the sufficient condition $\norm{{\pi}_{\mathbf{s},M} - {\pi}_{\mathbf{s},M}^{(0)}}_1 < \Delta_{\mathbf{s}}/2$ is implied by \eqref{eq:abstract-gap-condition}. Standard PR and multi-seed PPR follow by linearity of $\mathbf{s}\mapsto{\pi}_{\mathbf{s},M}$. \qedhere
\end{proof}

With Lemma~\ref{lem:abstract-detection} in hand, the detection theorem follows by substituting the closed-form gap lower bounds from the preceding subsection.

\begin{proof}[Proof of Theorem~\ref{thm:detection}]
	Following the main-text sketch (\S\ref{sec:lemmas}), Lemma~\ref{lem:abstract-detection} reduces detection on $P$ to a sufficient condition on $\|P^{12}\|_\infty$ in terms of the auxiliary gap on $P^{(0)}$: for either method,
		\begin{equation}\label{eq:detection-reduction-proof}
			\Delta_{\textnormal{method}} > 0 \;\Longleftarrow\; \|P^{12}\|_\infty \;<\; \frac{1-\alpha}{4\alpha}\,\Delta^{(0)}_{\textnormal{method}}.
		\end{equation}
		It remains to substitute explicit lower bounds on $\Delta^{(0)}_{\textnormal{method}}$ to derive the conditions stated in Theorem~\ref{thm:detection}.
	
	\emph{Part (i): Standard PR.} Under the mixing condition---equivalent to positivity of the bracket in \eqref{eq:pr-gap-condition-explicit}---Lemma~\ref{lem:pagerank-combined} gives the explicit lower bound
		\begin{equation*}
			\Delta^{(0)}_{\textnormal{PR}} \;\geq\; \frac{N}{\tilde{N}}\left[q_{\min}^\alpha - \frac{(\tilde{N}-N)(1-\alpha)}{N(1-\alpha+\alpha\eta)}\right] \;>\; 0,
		\end{equation*}
		so substituting into \eqref{eq:detection-reduction-proof} yields \eqref{eq:pr-gap-condition-explicit}.
	
	\emph{Part (ii): Multi-seed PPR.} Under the reachability hypothesis, Lemma~\ref{lem:ppr-combined} gives the closed-form lower bound
		\begin{equation*}
			\Delta^{(0)}_{\textnormal{PPR}} \;=\; \min_{i\leq N}\frac{1-\alpha}{|\mathcal{S}|}\sum_{v\in\mathcal{S}}\bigl[(I - \alpha P^{11,(0)})^{-1}\bigr]_{vi} \;>\; 0,
		\end{equation*}
		so substituting into \eqref{eq:detection-reduction-proof} yields \eqref{eq:ppr-gap-condition-explicit}.
	
	\emph{Invariant subspace case.} When $\Fc_N$ is $\Koop$-invariant, Theorem~\ref{thm:blockstr} gives the bottom-left block $K_{\tilde{N},M}^{21} = O$, hence $P^{12} = O$, so \eqref{eq:ppr-gap-condition-explicit} holds trivially. By Lemma~\ref{lem:pr-ppr-perturbation}, ${\pi}_{\mathcal{S}, M} = {\pi}_{\mathcal{S}, M}^{(0)}$, and by Lemma~\ref{lem:ppr-combined}, $\pi_{\mathcal{S}, M}^{(0)}(j) = 0$ for all $j > N$; hence $\pi_{\mathcal{S}, M}(j) = 0$ for all $j > N$. \qedhere
\end{proof}

\section{Finite-sample extensions}\label{sec:proofs-finite}

\emph{Goal.} This subsection states and proves the EDMD concentration bound (Proposition~\ref{prop:finite-sample-edmd}) and proves the end-to-end PR/PPR perturbation bound (Theorem~\ref{thm:end-to-end}). The PR/PPR sample complexity (Corollary~\ref{cor:sample-complexity}) and the finite-sample leakage bound (Corollary~\ref{cor:finite-sample-leakage}) are proved in \S\ref{app:finite-sample} below, combining Proposition~\ref{prop:finite-sample-edmd} with the row-normalization and PPR perturbation lemmas of Appendix~\ref{sec:detection_proofs}.

\emph{Strategy.} A single triangle inequality organizes the argument:
\[
\|{\pi}_M - \pi_{\textnormal{pop}}^{(0)}\|_1 \;\leq\; \underbrace{\|{\pi}_M - \pi_{\textnormal{pop}}\|_1}_{\text{finite-sample error}} \;+\; \underbrace{\|\pi_{\textnormal{pop}} - \pi_{\textnormal{pop}}^{(0)}\|_1}_{\text{population approximation error}},
\]
where $\pi_{\textnormal{pop}}$ is the PR (or PPR) vector of the row-normalized \arXivtag[population-level matrix $P_{\textnormal{pop}} = \mathcal{R}(|K_{\tilde{N}}^\tp|)$].
The population term is controlled by $\varepsilon_0 = \|K_{\tilde{N}}^\tp - (K_{\tilde{N}}^{(0)})^\tp\|_\infty$ through the shared perturbation bound (Lemma~\ref{lem:pr-ppr-perturbation}). The finite-sample term is controlled by $\varepsilon_M = \|K_{\tilde{N},M}^\tp - K_{\tilde{N}}^\tp\|_\infty$, which we obtain from matrix Bernstein (Proposition~\ref{prop:finite-sample-edmd}) and then plug into the same two-step composition (row-normalization stability + resolvent perturbation) used in the population case. The sample complexity and leakage corollaries then follow by inverting $\varepsilon_M = O(1/\sqrt{M})$ against the relevant gap.

We begin by stating the EDMD concentration bound, which supplies $\varepsilon_M$ via matrix Bernstein applied to the Gram and cross-covariance matrices. The boundedness assumption (Assumption~\ref{assume:bounded-dict}) activates the concentration machinery.

\paragraph{Relation to existing finite-sample EDMD bounds.}
Concentration-type finite-sample guarantees for EDMD have been developed in several complementary settings. Korda and Mezi\'c~\cite{korda2018convergence} established in-probability convergence of $K_{N,M}$ to the $L^2$-projected Koopman operator as $M\to\infty$, without an explicit rate. Quantitative rates appeared first for the Gram/cross-covariance blocks under sub-Gaussian and boundedness hypotheses: Llamazares-Elias et al.~\cite{elias2024datadriven},  N\"uske et al.~\cite{nuske2023finite}, Kostic et al.~\cite{Kostic-NeuRips2023} give operator-norm bounds of order $\arXivtag[1/\sqrt{M}]$ with explicit constants.  
The proposition below specializes such arguments to the infinity-norm quantity required by the row-normalization and PPR perturbation machinery of this paper.

\begin{proposition}[EDMD concentration]\label{prop:finite-sample-edmd}
	Let $\Dc_{\tilde{N}}$ be a linearly independent set satisfying Assumption~\ref{assume:bounded-dict} with constant $D > 0$.
	Let $\lambda_{\min}$ be the minimum of the moduli of the eigenvalues of the Gram matrix $G_{\tilde{N}} \coloneqq [\langle \psi_i,  \psi_j\rangle_{L^2(\mu)}]_{i,j=1}^{\tilde{N}} \in \C^{\tilde N \by \tilde N}$.  Let $\rho \in (0, 1/2)$.
	Then for $M \geq \frac{32\,\tilde{N}^2 D^4}{\lambda_{\min}^2}\log\frac{2\tilde{N}}{\rho}$, with probability $\geq 1 - 2\rho$:
	\begin{equation}\label{eq:finite-sample-EDMD-bound}
		\|K_{\tilde{N},M}^\tp - K_{\tilde{N}}^\tp\|_\infty \leq \varepsilon_M
	\end{equation}
	where
	\begin{equation}\label{eq:epsM-def}
		\varepsilon_M = C_{\text{EDMD}}\sqrt{2\log(2\tilde{N}/\rho)/M} \quad \text{and} \quad C_{\text{EDMD}} = \frac{4\tilde{N}^{3/2}D^2}{\lambda_{\min}}\!\left(1 + \frac{\|G_{\tilde{N}}\|_2}{\lambda_{\min}}\right).
	\end{equation}
\end{proposition}

\begin{proof}
	Under Assumption~\ref{assume:bounded-dict},
	Proposition~5.6 of~\cite{elias2024datadriven} gives, with probability $\geq 1-2\rho$,
	\[
	\|K_{\tilde{N},M} - K_{\tilde{N}}\|_2 \leq \frac{4\tilde{N}D^2}{\lambda_{\min}}\!\left(1+\frac{\|G_{\tilde{N}}\|_2}{\lambda_{\min}}\right)\sqrt{\frac{2\log(2\tilde{N}/\rho)}{M}}.
	\]
	Since $\|B^\tp\|_\infty \leq \sqrt{\tilde{N}}\|B\|_2$ for any $\tilde{N}\times\tilde{N}$ matrix $B$, this yields \eqref{eq:finite-sample-EDMD-bound} with $C_{\textnormal{EDMD}} = 4\tilde{N}^{3/2}D^2(1+\|G_{\tilde{N}}\|_2/\lambda_{\min})/\lambda_{\min}$.
\end{proof}

Plugging the EDMD bound into the two-step (row-normalization $+$ resolvent) composition yields the end-to-end perturbation bound.

\begin{proof}[Proof of Theorem~\ref{thm:end-to-end}]
	By the triangle inequality:
	\begin{equation*}
		\norm{{\pi}_M - \pi_{\textnormal{pop}}^{(0)}}_1 \leq \norm{{\pi}_M - \pi_{\textnormal{pop}}}_1 + \norm{\pi_{\textnormal{pop}} - \pi_{\textnormal{pop}}^{(0)}}_1.\arXivtag
	\end{equation*}
	
	\textit{Population approximation error.} By Lemma~\ref{lem:pr-ppr-perturbation} applied to $P_{\textnormal{pop}}$,
		\begin{equation}\label{eq:pop-error}
			\|\pi_{\textnormal{pop}} - \pi_{\textnormal{pop}}^{(0)}\|_1 \leq \frac{2\alpha}{1-\alpha}\|P^{12}_{\textnormal{pop}}\|_\infty \leq \frac{2\alpha\,\varepsilon_0}{(1-\alpha)\,r^{(0)}_{\min}},
		\end{equation}
		using $\|P^{12}_{\textnormal{pop}}\|_\infty \leq \varepsilon_0/r^{(0)}_{\min}$.
	\textit{Finite-sample error.} By Proposition~\ref{prop:finite-sample-edmd}, on an event of probability at least $1-2\rho$: $\norm{K_{\tilde{N},M}^\tp - K_{\tilde{N}}^\tp}_\infty \leq \varepsilon_M.$ 
	To apply Lemma~\ref{lem:row-norm-perturbation} with tight constants, we use $K_{\tilde{N}}^\tp$ as the reference matrix and verify that its minimum absolute row sum (equivalently, the minimum absolute column sum of $|K_{\tilde{N}}|$) is at least $r^{(0)}_{\min}$. For columns $i \leq N$, the block structure of $K_{\tilde{N}}^{(0)}$ gives
	\[
	\sum_j |K_{\tilde{N}}^\tp[i,j]| = \sum_j |K_{\tilde{N}}[j,i]| \geq \sum_j |K_{\tilde{N}}^{(0)}[j,i]| = \sum_j |(K_{\tilde{N}}^{(0)})^\tp[i,j]| \geq r^{(0)}_{\min},
	\]
	since $K_{\tilde{N}}$ and $K_{\tilde{N}}^{(0)}$ agree outside the bottom-left block and any additional entries only increase the absolute column sum. For columns $i > N$, $K_{\tilde{N}}$ and $K_{\tilde{N}}^{(0)}$ agree on column $i$, so $\sum_j |K_{\tilde{N}}^\tp[i,j]| = \sum_j |(K_{\tilde{N}}^{(0)})^\tp[i,j]| \geq r^{(0)}_{\min}$.
	Therefore the minimum absolute row sum of $|K_{\tilde{N}}^\tp|$ is at least $r^{(0)}_{\min}$. Applying Lemma~\ref{lem:row-norm-perturbation} with $A = K_{\tilde{N}}^\tp$ and $B = K_{\tilde{N},M}^\tp$ (so $r_{\min}^A \geq r^{(0)}_{\min}$):
	\[
	\|\mathcal{R}(|K_{\tilde{N},M}^\tp|) - \mathcal{R}(|K_{\tilde{N}}^\tp|)\|_\infty
	\leq \frac{2}{r^{(0)}_{\min}}\,\varepsilon_M.
	\]
	Since \arXivtag[$P = \mathcal{R}(|K_{\tilde{N},M}^\tp|)$ and $P_{\textnormal{pop}} = \mathcal{R}(|K_{\tilde{N}}^\tp|)$] are both row stochastic (hence $\|\cdot\|_\infty = 1$), Lemma~\ref{lem:pagerank-perturbation} applies with \arXivtag[$S = P$ and $S' = P_{\textnormal{pop}}$]:
	\begin{equation}\label{eq:sample-error}
		\|{\pi}_M - \pi_{\textnormal{pop}}\|_1 \;\leq\; \frac{\alpha}{1-\alpha}\|P - P_{\textnormal{pop}}\|_\infty
		\;\leq\; \frac{2\alpha\,\varepsilon_M}{(1-\alpha)\,r^{(0)}_{\min}}.
	\end{equation}
	
	Combining \eqref{eq:pop-error} and \eqref{eq:sample-error}:
		\[
		\|\pi_M - \pi_{\textnormal{pop}}^{(0)}\|_1 \leq \frac{2\alpha\,\varepsilon_0}{(1-\alpha)\,r^{(0)}_{\min}} + \frac{2\alpha\,\varepsilon_M}{(1-\alpha)\,r^{(0)}_{\min}} = \frac{2\alpha(\varepsilon_0 + \varepsilon_M)}{(1-\alpha)\,r^{(0)}_{\min}},
		\]
		which is the desired bound 
			\begin{equation}\label{eq:end-to-end-bound}
			\norm{{\pi}_M - \pi_{\textnormal{pop}}^{(0)}}_1 \;\leq\; \frac{2\alpha(\varepsilon_0 + \varepsilon_M)}{(1-\alpha)\,r^{(0)}_{\min}}.
		\end{equation}
\end{proof}

The sample-complexity corollary (Corollary~\ref{cor:sample-complexity}) follows by requiring the right-hand side of \eqref{eq:end-to-end-bound} to be smaller than half the PPR detection gap and inverting $\varepsilon_M = O(1/\sqrt{M})$; its proof, the analogous PR bound (Corollary~\ref{cor:pr-sample-complexity}), and the finite-sample leakage bound (Corollary~\ref{cor:finite-sample-leakage}) are given in \S\ref{app:finite-sample} below.

\subsection{Sample complexity and finite-sample leakage}\label{app:finite-sample}

Building on Theorem~\ref{thm:end-to-end}, we now prove the three finite-sample corollaries: the PPR sample complexity (Corollary~\ref{cor:sample-complexity}, stated in Section~\ref{subsec:end-to-end}), the PR sample complexity (Corollary~\ref{cor:pr-sample-complexity}), and the finite-sample leakage guarantee (Corollary~\ref{cor:finite-sample-leakage}, the empirical analogue of Proposition~\ref{prop:unconditional}). For both sample-complexity corollaries, requiring the total perturbation $\|{\pi}_M - \pi_{\textnormal{pop}}^{(0)}\|_1 < \Delta^{(0)}/2$ (where $\Delta$ is the relevant gap) preserves the top-$N$ ordering; the general bound \eqref{eq:cor-sc-general} is derived once in the PPR proof and then specialized to PR. For Corollary~\ref{cor:finite-sample-leakage}, $\varepsilon_M$ is combined with the population leakage bound via the triangle inequality applied to the PPR mass on $S_N$.

\begin{proof}[Proof of Corollary~\ref{cor:sample-complexity}]
	Let $\pi_{\textnormal{pop}}^{(0)}$ be the PPR vector of $P^{(0)}_{\textnormal{pop}} = \mathcal{R}(|K^{(0)}_{\tilde{N}}|^\tp)$ and ${\pi}_M$ be the PPR vector of $P = \mathcal{R}(|K_{\tilde{N},M}^\tp|)$. 
	
	\emph{Step 1: gap-based identification criterion.} On the event $\{\norm{{\pi}_M - \pi_{\textnormal{pop}}^{(0)}}_1 < \Delta^{(0)}_{\textnormal{PPR}}/2\}$, for any $i \leq N$ and $j > N$,
		\begin{equation*}
			({\pi}_M)_i \;\geq\; (\pi_{\textnormal{pop}}^{(0)})_i - \tfrac{\Delta^{(0)}_{\textnormal{PPR}}}{2}, \qquad ({\pi}_M)_j \;\leq\;(\pi_{\textnormal{pop}}^{(0)})_j + \tfrac{\Delta^{(0)}_{\textnormal{PPR}}}{2},
		\end{equation*}
		since $|({\pi}_M)_k - (\pi_{\textnormal{pop}}^{(0)})_k| \leq \norm{{\pi}_M - \pi_{\textnormal{pop}}^{(0)}}_1 < \Delta^{(0)}_{\textnormal{PPR}}/2$ for every coordinate $k$. By Lemma~\ref{lem:ppr-combined}, $(\pi_{\textnormal{pop}}^{(0)})_i \geq \Delta^{(0)}_{\textnormal{PPR}}$ for all $i \leq N$ and $(\pi_{\textnormal{pop}}^{(0)})_j = 0$ for all $j > N$, so $({\pi}_M)_i - ({\pi}_M)_j > 0$ for every such pair, i.e.\ $\Delta_{\textnormal{PPR}} > 0$ (detection success).
	
	\emph{Step 2: bound $\norm{{\pi}_M - \pi_{\mathrm{pop}}^{(0)}}_1$ in terms of $\|P^{12}_{\mathrm{pop}}\|_\infty$ and $\varepsilon_M$.} Provided $M \geq 32\,\tilde{N}^2 D^4 \log(2\tilde{N}/\rho)/\lambda_{\min}^2$, Theorem~\ref{thm:end-to-end} applies and gives, with probability $\geq 1-2\rho$:
		\begin{equation}\label{eq:cor-sc-eps0}
			\norm{{\pi}_M - \pi_{\mathrm{pop}}^{(0)}}_1 \;\leq\; \frac{2\alpha}{1-\alpha}\,\|P^{12}_{\mathrm{pop}}\|_\infty + \frac{2\alpha\,\varepsilon_M}{(1-\alpha)\,r^{(0)}_{\min}}.
		\end{equation}
		Since $r_{\max} \geq r^{(0)}_{\min}$, the first term satisfies $\frac{2\alpha}{1-\alpha}\|P^{12}_{\mathrm{pop}}\|_\infty \leq \frac{2\alpha\,r_{\max}}{(1-\alpha)\,r^{(0)}_{\min}}\|P^{12}_{\mathrm{pop}}\|_\infty$, so we may combine:
	\begin{equation}\label{eq:cor-sc-total}
		\norm{{\pi}_M - \pi_{\mathrm{pop}}^{(0)}}_1 \;\leq\; \frac{2\alpha}{(1-\alpha)r^{(0)}_{\min}}\bigl(r_{\max}\|P^{12}_{\mathrm{pop}}\|_\infty + \varepsilon_M\bigr).
	\end{equation}
	
	\emph{Step 3: invert against the gap.} Requiring \eqref{eq:cor-sc-total} $< \Delta_{\textnormal{PPR}}^{(0)}/2$, substituting $\varepsilon_M = C_{\text{EDMD}}\sqrt{2\log(2\tilde N/\rho)/M}$ from \eqref{eq:epsM-def}, squaring, and solving for $M$ yields
	\begin{equation}\label{eq:cor-sc-general}
		M \;\geq\; \frac{32\,\alpha^2\log(2\tilde N/\rho)\,C_{\text{EDMD}}^2}{(1-\alpha)^2 (r^{(0)}_{\min})^2\,\bigl(\Delta_{\textnormal{PPR}}^{(0)} - 4\alpha r_{\max}\|P^{12}_{\mathrm{pop}}\|_\infty/((1-\alpha)r^{(0)}_{\min})\bigr)^2}.
	\end{equation}
	Setting $\alpha = 1/2$ so that $\alpha^2/(1-\alpha)^2 = 1$ and $4\alpha/((1-\alpha)r^{(0)}_{\min}) = 4/r^{(0)}_{\min}$ gives the second argument of the max in \eqref{eq:ppr-sample-complexity}; combining with the Step~2 requirement $M \geq 32\,\tilde{N}^2 D^4 \log(2\tilde{N}/\rho)/\lambda_{\min}^2$ yields the full bound.
\end{proof}

\begin{corollary}[Sample complexity for PR detection]\label{cor:pr-sample-complexity}
	Let $\Delta^{(0)}_{\textnormal{PR}}$ be the PR auxiliary gap of Section~\ref{subsec:block-theory} computed from $P^{(0)}_{\textnormal{pop}}$.
		If $\Delta^{(0)}_{\textnormal{PR}} > 0$ and $\|P^{12}_{\mathrm{pop}}\|_\infty < (1-\alpha)r^{(0)}_{\min}\Delta^{(0)}_{\textnormal{PR}}/(4\alpha r_{\max})$, detection ($\Delta_{\textnormal{PR}} > 0$) holds with probability $\geq 1-2\rho$ for
		\begin{equation}\label{eq:pr-sample-complexity}
			M \;\geq\; \max\!\left\{\frac{32\,\tilde{N}^2 D^4}{\lambda_{\min}^2}\log\frac{2\tilde{N}}{\rho},\;\; \frac{32\,\alpha^2\log(2\tilde{N}/\rho)\, C_{\text{EDMD}}^2}{(1-\alpha)^2 (r^{(0)}_{\min})^2\bigl(\Delta^{(0)}_{\textnormal{PR}} - 4\alpha r_{\max}\|P^{12}_{\mathrm{pop}}\|_\infty/((1-\alpha)r^{(0)}_{\min})\bigr)^2}\right\}.
		\end{equation}
		Compared with Corollary~\ref{cor:sample-complexity}, this bound carries an additional $\alpha^2/(1-\alpha)^2$ prefactor. Since $\Delta^{(0)}_{\textnormal{PR}} > 0$ requires the mixing condition of Theorem~\ref{thm:detection}(i), 
		which may force $\alpha$ close to $1$, it is not always possible to set $\alpha = 1/2$. Therefore, the prefactor $\alpha^2/(1-\alpha)^2$ can be significantly larger than $1$. 
\end{corollary}

\begin{proof}
	Specialize \eqref{eq:cor-sc-general} (derived in the proof above) to $\pi_{\textnormal{pop}}^{(0)}$ the standard PR vector and $\Delta = \Delta^{(0)}_{\textnormal{PR}}$, giving \eqref{eq:pr-sample-complexity}.
\end{proof}

\begin{corollary}[Finite-sample leakage]\label{cor:finite-sample-leakage}
	Let $\Lambda_{\mathbf{s}}^{\gamma,\textnormal{pop}}(S_N)$ denote the leakage defined with $K_{\tilde{N}} \coloneqq \lim_{M\to\infty} K_{\tilde{N},M}$ as in  \eqref{eq:leakage-ppr} in place of $K_{\tilde{N},M}$. For any $\mathbf{s}$ supported on $S_N$ and any $M \geq \frac{32\,\tilde{N}^2 D^4}{\lambda_{\min}^2}\log\frac{2\tilde{N}}{\rho}$, with probability $\geq 1-2\rho$:
	\begin{equation}\label{eq:finite-sample-leakage}
		\Lambda_{\mathbf{s}}^{\gamma,\textnormal{pop}}(S_N) \coloneqq(1-\gamma)\sum_{k=1}^{\infty} \gamma^k \sum_{j \notin S_N} [\mathbf{s}^\tp  |K_{\tilde{N}}^{\tp}|^k]_j \leq \frac{1-\gamma}{1-\alpha}\!\left(1 - \pi_{\mathbf{s}, M}(S_N) + \frac{2\alpha\varepsilon_M}{(1-\alpha)r^{(0)}_{\min}}\right).
	\end{equation}
\end{corollary}

\begin{proof}
	Let \arXivtag[$P = \mathcal{R}(|K_{\tilde{N},M}^\tp|)$ and $P_{\textnormal{pop}} = \mathcal{R}(|K_{\tilde{N}}^\tp|)$] denote the population and empirical row-normalized matrices, and let $\pi_{\mathbf{s}}$, $\pi_{\mathbf{s}, M}$ denote their respective PPR vectors with preference $\mathbf{s}$ and damping $\alpha$.
	
	\emph{Step 1: population leakage bound.} The proof of Proposition~\ref{prop:unconditional} applies verbatim with $K_{\tilde{N},M}$ replaced by $K_{\tilde{N}}$ and \arXivtag[$\pi_{\mathbf{s}, M}$ replaced by $\pi_{\mathbf{s}}$], yielding
	\begin{equation}\label{eq:cor-fsl-pop}
		\Lambda_{\mathbf{s}}^{\gamma,\textnormal{pop}}(S_N) \;\leq\; \frac{1-\gamma}{1-\alpha}\bigl(1 - \pi_{\mathbf{s}}(S_N)\bigr),
	\end{equation}
	where $\pi_{\mathbf{s}}$ is the PPR vector of $P_{\textnormal{pop}}$.
	
	\emph{Step 2: bound $\pi_{\mathbf{s}}(S_N) - \pi_{\mathbf{s}, M}(S_N)$.} By Proposition~\ref{prop:finite-sample-edmd}, $\|K_{\tilde{N},M}^\tp - K_{\tilde{N}}^\tp\|_\infty \leq \varepsilon_M$ with probability $\geq 1-2\rho$. Lemma~\ref{lem:row-norm-perturbation} with \arXivtag[$A = K_{\tilde{N},M}^\tp$ and $B = K_{\tilde{N}}^\tp$] gives $\|P - P_{\textnormal{pop}}\|_\infty \leq (2/r^{(0)}_{\min})\varepsilon_M$. Lemma~\ref{lem:pagerank-perturbation} then yields
	\begin{equation}\label{eq:cor-fsl-ppr}
		\|\pi_{\mathbf{s}} - \pi_{\mathbf{s}, M}\|_1 \;\leq\; \frac{2\alpha\,\varepsilon_M}{(1-\alpha)\,r^{(0)}_{\min}},
	\end{equation}
	and since $|\pi_{\mathbf{s}}(S_N) - \pi_{\mathbf{s}, M}(S_N)| \leq \|\pi_{\mathbf{s}} - \pi_{\mathbf{s}, M}\|_1$,
	\begin{equation}\label{eq:cor-fsl-diff}
		\pi_{\mathbf{s}}(S_N) \;\geq\; \pi_{\mathbf{s}, M}(S_N) - \frac{2\alpha\,\varepsilon_M}{(1-\alpha)\,r^{(0)}_{\min}}.
	\end{equation}
	
	\emph{Step 3: combine.} Substituting \eqref{eq:cor-fsl-diff} into \eqref{eq:cor-fsl-pop}:
	\begin{equation*}
		\Lambda_{\mathbf{s}}^{\gamma,\textnormal{pop}}(S_N) \;\leq\; \frac{1-\gamma}{1-\alpha}\Bigl(1 - \pi_{\mathbf{s}, M}(S_N) + \frac{2\alpha\,\varepsilon_M}{(1-\alpha)\,r^{(0)}_{\min}}\Bigr),
	\end{equation*}
	which is \eqref{eq:finite-sample-leakage}.
\end{proof}

\begin{remark}[Practical seed selection]\label{rem:practical}
	The seed set $\mathcal{S}$ should be chosen to reflect the practitioner's prediction objective. Natural choices include state coordinates (when state prediction is the goal), slow observables identified by a preliminary spectral analysis, or any observables whose future evolution is of primary interest. When no domain knowledge is available, one can use standard PR ($\mathcal{S} = \{1,\ldots,\tilde{N}\}$) or run standard PR first to identify candidate seeds, then refine with PPR. The Andersen--Chung--Lang theory~\cite{andersen2006local} provides complementary guarantees via sweep cuts.
\end{remark}

\end{document}